\documentclass{article}

\usepackage[english]{babel}
\usepackage[title]{appendix}
\usepackage[utf8]{inputenc}
\usepackage{amsmath}
\usepackage{amsfonts}
\usepackage{graphicx}
\usepackage[colorinlistoftodos]{todonotes}
\usepackage{blindtext}
\numberwithin{equation}{section}
\usepackage{amssymb}
\usepackage{algorithm}
\usepackage{authblk}
\usepackage{bbm}
\usepackage{amsthm}
\usepackage{mathrsfs}  

\makeatletter
\newcommand*{\rom}[1]{\expandafter\@slowromancap\romannumeral #1@}

\newtheorem{theorem}{Theorem}[section]

\newtheorem{lemma}[theorem]{Lemma}
\newtheorem{proposition}{Proposition}[section]

\theoremstyle{remark}

\theoremstyle{definition}
\newtheorem{definition}{Definition}[section]
\title{Cycle structure of Mallows permutation model with the $L^1$ distance}

\author{Chenyang Zhong}
\affil{Department of Statistics, Columbia University}

\date{\today}

\begin{document}
\maketitle
\begin{abstract}

Introduced by Mallows as a ranking model in statistics, Mallows permutation model is a class of non-uniform probability distributions on the symmetric group $S_n$. The model depends on a distance metric on $S_n$ and a scale parameter $\beta$. In this paper, we take the distance metric to be the $L^1$ distance (also known as Spearman's footrule in the statistics literature), and investigate the cycle structure of random permutations drawn from Mallows permutation model with the $L^1$ distance. 

We focus on the parameter regime where $\beta>0$. We show that the expected length of the cycle containing a given point is of order $\min\{\max\{\beta^{-2},1\},n\}$, and the expected diameter of the cycle containing a given point is of order $\min\{e^{-2\beta}\max\{\beta^{-2},1\}, n-1\}$. Moreover, when $\beta\ll n^{-1\slash 2}$, the sorted cycle lengths (in descending order) normalized by $n$ converge in distribution to the Poisson-Dirichlet law with parameter $1$. The proofs of the results rely on the hit and run algorithm, a Markov chain for sampling from the model. 
\end{abstract}

\section{Introduction}\label{Sect.1}

Recently, there has been much interest from the probability and combinatorics community in understanding the cycle structure of various models of random permutations. For uniformly random permutations, many aspects of the cycle structure are known. For example, for a uniformly random permutation of $\{1,2,\cdots,n\}$, the length of the cycle that contains a given point from $\{1,2,\cdots,n\}$ has a uniform distribution on $\{1,2,\cdots,n\}$, and the sorted cycle lengths (in descending order) normalized by $n$ converge in distribution to the Poisson-Dirichlet law with parameter $1$ as $n$ approaches infinity (see e.g. \cite{SL,ABT}). In this paper, we investigate the cycle structure of a \emph{non-uniform} probability distribution on permutations called Mallows permutation model with the $L^1$ distance. We introduce this permutation model in the following. 

Mallows permutation model was originally introduced by Mallows \cite{Mal} as a ranking model in statistics. Since its introduction, the model has been applied to statistics, social science, machine learning, and various other fields. The model is a class of non-uniform probability distributions on permutations, and depends on a distance metric $d(\sigma,\tau)$ on the symmetric group $S_n$, a scale parameter $\beta\in\mathbb{R}$, and a location parameter $\sigma_0\in S_n$. The probability of picking $\sigma\in S_n$ under the model is proportional to $\exp(-\beta d(\sigma,\sigma_0))$. We refer to \cite[Section 1]{Zho1} for a review of the literature on Mallows permutation model.

The distance metric $d(\sigma,\tau)$ can be chosen from a collection of metrics on permutations. Below we list several popular choices of such metrics (see \cite[Chapter 6]{D} for further discussions):
\begin{itemize}
    \item $L^1$ distance, or Spearman's footrule: $d(\sigma,\tau)=\sum_{i=1}^n|\sigma(i)-\tau(i)|$;
    \item $L^2$ distance, or Spearman's rank correlation: $d(\sigma,\tau)=\sum_{i=1}^n (\sigma(i)-\tau(i))^2$;
    \item Kendall's $\tau$: $d(\sigma,\tau)=$ minimum number of pairwise adjacent transpositions taking $\sigma^{-1}$ to $\tau^{-1}$;
    \item Cayley distance: $d(\sigma,\tau)=$ minimum number of transpositions taking $\sigma$ to $\tau$. 
\end{itemize}
In this paper, we focus on Mallows permutation model with the $L^1$ distance, which will also be referred to as ``the $L^1$ model'' in the rest of this paper. We restrict our attention to the parameter regime where $\beta>0$ and $\sigma_0=Id$, the identity permutation. For any $\sigma,\tau\in S_n$, we denote by $H(\sigma,\tau):=\sum_{j=1}^n |\sigma(j)-\tau(j)|$ the $L^1$ distance between them. We also denote by $\mathbb{P}_{n,\beta}$ the probability measure that corresponds to the $L^1$ model. Thus for any $\sigma\in S_n$,
\begin{equation*}
    \mathbb{P}_{n,\beta}(\sigma)=Z_{n,\beta}^{-1}\exp(-\beta H(\sigma,Id)),
\end{equation*}
where $Z_{n,\beta}$ is the normalizing constant. Note that the model is biased towards the identity permutation. Due to this spatial structure, the model is known as a ``spatial random permutation'' in mathematical physics (see e.g. \cite{GRU,BV,FM}).

For a random permutation picked from Mallows permutation model, Borodin, Diaconis, and Fulman \cite{BDJ2} asked ``What is the distribution of the cycle structure, longest increasing subsequence, \dots?'' In this paper, we address the question regarding the cycle structure for Mallows permutation model with the $L^1$ distance. Our results show that the expected length of the cycle containing a given point is of order $\min\{\max\{\beta^{-2},1\},n\}$, and the expected diameter of the cycle containing a given point is of order $\min\{e^{-2\beta}\max\{\beta^{-2},1\}, n-1\}$. Moreover, when $\beta\ll n^{-1\slash 2}$, the sorted cycle lengths (in descending order) normalized by $n$ converge in distribution to the Poisson-Dirichlet law with parameter $1$.

In the following, we introduce some notations that will be used throughout this paper. We denote by $\mathbb{N}$ the set of nonnegative integers and $\mathbb{N}^{*}$ the set of positive integers. We denote $[0]:=\emptyset$ and $[n]:=\{1,2,\cdots,n\}$ for any $n \in \mathbb{N}^{*}$. We use $C,c$ to denote positive absolute constants, and the values of these constants may change from line to line. For any two quantities $x,y$, we write $x\asymp y$ if and only if $cy\leq c\leq Cy$ for some positive absolute constants $C,c$. We fix a probability space $(\Omega,\mathcal{F},\mathbb{P})$.

We also introduce the following definition.

\begin{definition}
For any $n\in\mathbb{N}^{*}$, $\sigma\in S_n$, and $j\in [n]$, we denote by $\mathcal{C}_j(\sigma)$ the set of points in the cycle of $\sigma$ that contains $j$.
\end{definition}

In the following, we present the main results of this paper.

\subsection{Main results}\label{Sect.1.1}

The following theorem gives the order of the expected length of the cycle containing a given point. It implies the one-dimensional case of \cite[Conjecture 1.19]{FM}. 

\begin{theorem}\label{Thm1.1}
Assume that $n\in \mathbb{N}^{*}$ and $\beta>0$, and let $\sigma$ be drawn from $\mathbb{P}_{n,\beta}$. For any $s\in [n]$, we have 
\begin{equation*}
    \mathbb{E}[|\mathcal{C}_s(\sigma)|]\asymp \min\{\max\{\beta^{-2},1\},n\}.
\end{equation*}
\end{theorem}

The following theorem gives the order of the expected diameter of the cycle containing a given point.

\begin{theorem}\label{Thm1.3.1}
Assume that $n\in \mathbb{N}^{*}$ and $\beta>0$, and let $\sigma$ be drawn from $\mathbb{P}_{n,\beta}$. For any $s\in [n]$, we have
\begin{equation*}
  \mathbb{E}[\max(\mathcal{C}_s(\sigma))-\min(\mathcal{C}_s(\sigma))] \asymp \min\{e^{-2\beta}\max\{\beta^{-2},1\}, n-1\}.
\end{equation*}
\end{theorem}

The following theorem gives the limiting distribution of the lengths of long cycles and the length of the cycle containing a given point when $\beta>0$ and $\beta\ll n^{-1\slash 2}$. The result shows that, for this parameter regime, the limiting distributions are the same as those for uniformly random permutations. Throughout this paper, we denote by $U([0,1])$ the uniform distribution on $[0,1]$.

\begin{theorem}\label{Thm1.2}
Let $(\beta_n)_{n=1}^{\infty}$ be an arbitrary sequence of positive numbers such that $\lim_{n\rightarrow\infty} n^{1\slash 2} \beta_n=0$. Let $\sigma$ be drawn from $\mathbb{P}_{n,\beta_n}$, and let $l_1(\sigma)\geq l_2(\sigma)\geq \cdots$ be the sorted cycle lengths of $\sigma$ (taking $l_k(\sigma)=0$ for any integer $k$ that is larger than the number of cycle of $\sigma$). Then as $n\rightarrow\infty$,
\begin{eqnarray}\label{PoissonDirichlet}
&&    \frac{1}{n}(l_1(\sigma),l_2(\sigma),\cdots)\xrightarrow{d}\text{ Poisson-Dirichlet law with parameter } 1.
\end{eqnarray}
Moreover, for any deterministic sequence $(s_n)_{n=1}^{\infty}$ such that $s_n\in [n]$ for every $n\in\mathbb{N}^{*}$, as $n\rightarrow\infty$,
\begin{equation}\label{Unif}
    \frac{|\mathcal{C}_{s_n}(\sigma)|}{n}\xrightarrow{d} U([0,1]).
\end{equation}
\end{theorem}

\subsection{Acknowledgement}

The author wishes to thank his PhD advisor, Persi Diaconis, for encouragement, support, and many helpful conversations. The author also thanks Sumit Mukherjee and Wenpin Tang for their helpful comments.

\section{Hit and run algorithm and preliminary results}\label{Sect.2}

We review the hit and run algorithm for sampling from the $L^1$ model in Section \ref{Sect.2.1}, and present some preliminary results in Section \ref{Sect.2.2}. Throughout this section, we assume that $n\in \mathbb{N}^{*}$ and $\beta>0$. 

\subsection{Hit and run algorithm for sampling from the $L^1$ model}\label{Sect.2.1}

In this subsection, we review the hit and run algorithm for
sampling from the $L^1$ model. This algorithm, introduced in \cite{Zho1}, is a Markov chain having $\mathbb{P}_{n,\beta}$ as its stationary distribution. Each step of the Markov chain consists of two sequential parts as given below:
\begin{itemize}
  \item Starting from $\sigma$, for each $i\in [n]$, independently sample $u_i$ from the uniform distribution on $[0, e^{-2\beta (\sigma(i)-i)_{+} }]$. Let $b_i=i-\frac{\log(u_i)}{2\beta}$ for every $i\in [n]$.
  \item Sample $\sigma'$ uniformly from the set $\{\tau\in S_n: \tau(i)\leq b_i\text{ for every }i\in [n] \}$. The next state of the Markov chain is $\sigma'$.
\end{itemize}
The second part can be done as follows: Look at the places $i$ where $b_i\geq n$, and place the symbol $n$ at a uniform choice among these places; look at the places $i$ where $b_i\geq n-1$, and place the symbol $n-1$ at a uniform choice among these places (with the place where $n$ was placed deleted); and so on. This gives the permutation $\sigma'$. Here, we say that the symbol $j$ is placed at the place $i$ if $\sigma'(i)=j$.

\subsection{Preliminary results}\label{Sect.2.2}

In this subsection, we present some preliminary results, which will be used in the proofs of the main results. We start with the following two definitions.

\begin{definition}\label{Def2.2}
For any $\sigma\in S_n$, we define $S(\sigma):=\{(j,\sigma(j)):j\in [n]\}$. For any $j\in [n]$ and any $\sigma\in S_n$, we define
\begin{equation}\label{Defnd}
     \mathcal{D}_j(\sigma):=\{k\in [n]: k\leq j, \sigma(k)\geq j+1\}, 
\end{equation}
\begin{equation}
    \mathcal{D}_j'(\sigma):=\{k\in [n]: k\geq j+1, \sigma(k)\leq j\}.
\end{equation}
Note that
\begin{equation}\label{Eq5.1.1}
    |\mathcal{D}_j(\sigma)|=j-|\{k\in [n]:k\leq j,\sigma(k)\leq j\}|=|\mathcal{D}'_j(\sigma)|.
\end{equation}
We also define $\mathcal{D}_0(\sigma):=\emptyset$ and $\mathcal{D}'_0(\sigma):=\emptyset$.
\end{definition}

\begin{definition}
For any $j\in [n]$, $\Delta>0$, and $\sigma\in S_n$, we define
\begin{equation*}
    \mathcal{R}_{j,\Delta}(\sigma):=S(\sigma)\cap \{(x,y)\in [n]^2: x\geq j+\Delta,y\leq j-\Delta\},
\end{equation*}
\begin{equation*}
    \mathcal{R}'_{j,\Delta}(\sigma):=S(\sigma) \cap \{(x,y)\in [n]^2: x\leq j-\Delta,y\geq j+\Delta\}. 
\end{equation*}
\end{definition}

We have the following three propositions. Proposition \ref{P2.1} follows from the proof of \cite[Proposition 5.3.1]{Zhong2} (when $\beta>1$, replace $\beta$ by $1$ in proper places of the proof), and the proof of Proposition \ref{P2.1_n} is given in the appendix.

\begin{proposition}\label{P2.1}
Assume that $\beta>0$. There exists a positive absolute constant $C_0$, such that for any $r\geq C_0\max\{\beta^{-1},1\}$ and any $j\in [n]$, 
\begin{equation*}
     \mathbb{P}_{n,\beta}(|\mathcal{D}_j(\sigma)|\geq r)\leq 3\exp(-r\slash 4). 
\end{equation*}
\end{proposition}

\begin{proposition}\label{P2.1_n}
Assume that $\beta\geq c_0$ for a fixed positive constant $c_0$. There exists a positive constant $C_0$ that only depends on $c_0$, such that for any $r\in \mathbb{N}^{*}$ and any $j\in [n]$, we have
\begin{equation*}
    \mathbb{P}_{n,\beta}(|\mathcal{D}_j(\sigma)|\geq r)\leq C_0e^{-2\beta r}.
\end{equation*}
\end{proposition}

\begin{proposition}[Proposition 5.3.3 of \cite{Zhong2}]\label{P2.3}
Assume that $0<\beta\leq C_0$ for a fixed positive constant $C_0$. Then there exist positive constants $C_1,C_2,C_3,c_1$ that only depend on $C_0$, such that for any $j\in [n]$, $\Delta\geq C_1\beta^{-1}$, and $r\geq e^{-C_1^{-1}\beta\Delta}$, we have
\begin{eqnarray*}
   && \max\{\mathbb{P}_{n,\beta}(|\mathcal{R}_{j,\Delta}(\sigma)|\geq C_2 r\beta^{-1}),\mathbb{P}_{n,\beta}(|\mathcal{R}'_{j,\Delta}(\sigma)|\geq C_2 r\beta^{-1})\}\nonumber\\
    &\leq & C_3\max\{\beta r^{-1}, \Delta\} \exp(-c_1\min\{r,1\}\beta^{-1}). 
\end{eqnarray*}
\end{proposition}

\section{Proofs of the main results}\label{Sect.3}

In this section, we present the proofs of the main results stated in Section \ref{Sect.1.1}. Throughout this section, we assume that $n\in \mathbb{N}^{*}$ and $\beta>0$. 

\subsection{Arc evolution for the hit and run algorithm}\label{Sect.3.1}

The proofs of our main results rely on keeping track of the sampling process of the hit and run algorithm for the $L^1$ model. We first sample $\sigma_0$ from $\mathbb{P}_{n,\beta}$, and then run one iteration of the hit and run algorithm for the $L^1$ model (as introduced in Section \ref{Sect.2.1}) started from $\sigma_0$ to obtain $\sigma$. This can be specifically done as follows:
\begin{itemize}
    \item[(a)] For each $j\in [n]$, we independently sample $u_j$ from the uniform distribution on $[0, e^{-2\beta (\sigma_0(j)-j)_{+}}]$. Let $b_j:=j-\log(u_j)\slash (2\beta)$ for each $j\in [n]$.
    \item[(b)] For each $j\in [n]$, let
    \begin{equation}\label{DefinitionNj}
        N_j:=|\{k\in [n]:b_k\geq j\}|-n+j.
    \end{equation}
    We sample $\sigma$ uniformly from the set 
\begin{equation*}
    \{\tau\in S_n: \tau(j)\leq b_j\text{ for every }j\in [n]\}
\end{equation*}
through the following procedure. Look at the $N_n$ integers $j\in [n]$ with $b_j\geq n$, and pick $Y_n$ uniformly from these integers; then look at the $N_{n-1}$ remaining integers $j\in[n]$ with $b_j\geq n-1$ (with $Y_n$ deleted from the list), and pick $Y_{n-1}$ uniformly from these integers; and so on. In this way we obtain $\{Y_j\}_{j=1}^n$. Finally, we let $\sigma\in S_n$ be such that $\sigma(Y_j)=j$ for every $j\in [n]$.
\end{itemize}

We denote by $\mathcal{B}_n$ the $\sigma$-algebra generated by $\sigma_0$ and $\{b_j\}_{j=1}^n$. For each $l\in \{0\}\cup [n-1]$, we denote by $\mathcal{B}_l$ the $\sigma$-algebra generated by $\sigma_0$, $\{b_j\}_{j=1}^n$, and $\{Y_j\}_{j=l+1}^n$. We also introduce the following definition.

\begin{definition}\label{Defi3.1}
Consider any $l\in [n]$ and any sequence $\mathbf{a}=(a_1,a_2,\cdots,a_s)$ (where $s\in \mathbb{N}^{*}$ and $a_1,a_2,\cdots,a_s$ are distinct integers from $[n]$). We denote by $|\mathbf{a}|=s$ the number of elements in the sequence $\textbf{a}$. We say that $\mathbf{a}$ forms an \textbf{open arc at step $l$} if $a_1,\cdots,a_{s-1}\geq l$, $a_s\leq l-1$, $a_1\notin \{Y_{l},\cdots,Y_n\}$, and $Y_{a_i}=a_{i+1}$ for every $i\in [s-1]$. In this case, we call $a_1$ the \textbf{head} of the open arc and $a_s$ the \textbf{tail} of the open arc. We say that $\textbf{a}$ forms a \textbf{closed arc at step $l$} if $a_1,\cdots,a_s\geq l$, $Y_{a_i}=a_{i+1}$ for every $i\in [s-1]$, and $Y_{a_s}=a_1$. Both open arcs and closed arcs are referred to as \textbf{arcs}. 

Two sequences that are both closed arcs at step $l$ and differ only by a sequence of circular shifts are viewed as the same closed arc. We denote by $\mathcal{A}_O(l)$ the set of open arcs at step $l$ and $\mathcal{A}_C(l)$ the set of closed arcs at step $l$. 
\end{definition}

Note that for any $l\in [n]$, the collection of open and closed arcs at step $l$ forms a partition of $[n]$. Moreover, $\mathcal{A}_O(1)=\emptyset$ and each closed arc in $\mathcal{A}_C(1)$ corresponds to a cycle of $\sigma$. We define the set of open arcs at step $n+1$ as $\mathcal{A}_O(n+1):=\{(1),(2),\cdots,(n)\}$ and the set of closed arcs at step $n+1$ as $\mathcal{A}_C(n+1):=\emptyset$.

Now we study the evolution of open and closed arcs in part (b) of the above sampling process. We sequentially consider $l=n,n-1,\cdots,1$. The transition from $\mathcal{A}_O(l+1)$ and $\mathcal{A}_O(l+1)$ to $\mathcal{A}_O(l)$ and $\mathcal{A}_O(l)$ (referred to as ``step $l$ transition'' hereafter) can be described as follows. Suppose that $Y_{n}, Y_{n-1}, \cdots, Y_{l+1}$ have been sampled and $Y_1,\cdots,Y_{l}$ are not yet generated; to determine $Y_l$, we sample $Y_l$ uniformly from $\{j\in [n]: b_j\geq l\}\backslash \{Y_{l+1},\cdots,Y_{n}\}$. This yields the following changes for open and closed arcs (note that each of $l$ and $Y_l$ is contained in an open arc at step $l+1$): 
\begin{itemize}
    \item If $l$ and $Y_l$ are contained in the same open arc at step $l+1$ (denoted by $\mathbf{a}=(a_1,\cdots,a_s)$), then necessarily $a_1=Y_l$ and $a_s=l$, and at step $l$ $\mathbf{a}$ forms a closed arc.
    \item If $l$ and $Y_l$ are contained in two different open arcs at step $l+1$ (denoted by $\mathbf{a}=(a_1,\cdots,a_s)$ and $\mathbf{c}=(c_1,\cdots,c_t)$, respectively), then necessarily $a_s=l$ and $c_1=Y_l$, and at step $l$ $\mathbf{a}$ and $\mathbf{c}$ are merged into a new open arc $(a_1,\cdots,a_s,c_1,\cdots,c_t)$.  
\end{itemize}

For any $j\in [n]$, we define $W_t^{(j)}\in \{0,1\}$ and $Z_t^{(j)}\in [n]$ for any $t\in \mathbb{N}$ as follows. Let $W_0^{(j)}=1$ and $Z_0^{(j)}=j$. Below we define $W_t^{(j)}\in \{0,1\}$ and $Z_t^{(j)}\in [n]$ for any $t\in\mathbb{N}^{*}$ inductively. For any $t \in \mathbb{N}^{*}$, suppose that $W_{t-1}^{(j)}\in \{0,1\}$ and $Z_{t-1}^{(j)}\in [n]$ have been defined. If $W_{t-1}^{(j)}=0$, we let $W_t^{(j)}=0$ and $Z_t^{(j)}=Z_{t-1}^{(j)}$. Below we consider the case where $W_{t-1}^{(j)}=1$. If $j$ is contained in a closed arc at step $Z_{t-1}^{(j)}$, we let $W_t^{(j)}=0$ and $Z_t^{(j)}=Z_{t-1}^{(j)}$; if $j$ is contained in an open arc at step $Z_{t-1}^{(j)}$, we take $W_t^{(j)}=1$ and let $Z_t^{(j)}$ be the tail of this open arc.

In the rest of this subsection, we state and prove some preliminary results that will be used in the proofs of our main results.

The following lemma on $\{b_j\}_{j=1}^n$ is straightforward.

\begin{lemma}\label{L3.1}
For any $j\in [n]$, we have $b_j\geq \max\{j,\sigma_0(j)\}$ and
\begin{equation*}
    \mathbb{P}(b_j\geq x|\sigma_0)=e^{-2\beta (x-\max\{j,\sigma_0(j)\})_{+}}, \quad  \forall x\in\mathbb{R}.
\end{equation*}
\end{lemma}

Note that by (\ref{DefinitionNj}) and Lemma \ref{L3.1}, for any $j\in [n]$, 
\begin{equation}\label{Basic_n}
  N_j=1+\sum_{k=1}^{j-1}\mathbbm{1}_{b_k\geq j}\geq 1.
\end{equation}
Propositions \ref{P3.1}-\ref{P3.5} below bound $\{N_j\}_{j=1}^n$. 

\begin{proposition}\label{P3.1}
For any $j\in [n]$ such that $j\geq \beta^{-1}$, we have
\begin{equation}\label{Eq3.6}
    \mathbb{P}(N_j\leq e^{-2}\beta^{-1} \slash 4 )\leq \exp(-c\beta^{-1}),
\end{equation}
where $c$ is a positive absolute constant.
\end{proposition}

\begin{proof}

Consider any $j\in [n]$ such that $j\geq \beta^{-1}$. If $\beta> 1\slash 2$, then $e^{-2}\beta^{-1}\slash 4<1$; by (\ref{Basic_n}), $\mathbb{P}(N_j\leq e^{-2}\beta^{-1}\slash 4)=\mathbb{P}(N_j=0)=0$ and (\ref{Eq3.6}) holds. Below we assume that $\beta\leq 1\slash 2$; note that this implies $\lfloor \beta^{-1}\rfloor \geq 
 \beta^{-1}\slash 2$.

By Lemma \ref{L3.1}, for any $k\in \{j-\lfloor \beta^{-1}\rfloor +1,\cdots, j\}$, 
\begin{equation*}
    \mathbb{P}(b_k\geq j|\sigma_0)=e^{-2\beta(j-\max\{k,\sigma_0(k)\})_{+}}\geq e^{-2}. 
\end{equation*}
By (\ref{Basic_n}) and Hoeffding's inequality (see e.g. \cite[Theorem 2.8]{BLM}), for any $t\geq 0$, 
\begin{equation*}
    \mathbb{P}(N_j\leq (e^{-2}-t)\lfloor \beta^{-1}\rfloor|\sigma_0)\leq  \exp(-2\lfloor \beta^{-1}\rfloor t^2)\leq \exp(-\beta^{-1}t^2).
\end{equation*}
Taking $t=e^{-2}\slash 2$, we obtain that
\begin{equation*}
    \mathbb{P}(N_j\leq e^{-2}\beta^{-1}\slash 4|\sigma_0)\leq \exp(-c\beta^{-1}),
\end{equation*}
hence
\begin{equation*}
    \mathbb{P}(N_j\leq e^{-2}\beta^{-1}\slash 4)=\mathbb{E}[\mathbb{P}(N_j\leq e^{-2}\beta^{-1}\slash 4|\sigma_0)]\leq \exp(-c\beta^{-1}).
\end{equation*}

\end{proof}

\begin{proposition}\label{P3.1.v}
For any $j\in [n]$ such that $j\leq \beta^{-1}$, we have
\begin{equation*}
   \mathbb{P}(N_j\leq e^{-2} j\slash 2)\leq \exp(-cj),
\end{equation*}
where $c$ is a positive absolute constant.
\end{proposition}
\begin{proof}

Consider any $j\in [n]$ such that $j\leq \beta^{-1}$. By Lemma \ref{L3.1}, for any $k\in [j]$,
\begin{equation*}
    \mathbb{P}(b_k\geq j |\sigma_0)= e^{-2\beta (j-\max\{k,\sigma_0(k)\})_{+}}\geq e^{-2\beta (j-k)}\geq e^{-2}. 
\end{equation*}
By (\ref{Basic_n}) and Hoeffding's inequality, for any $t\geq 0$, 
\begin{equation*}
    \mathbb{P}(N_j\leq (e^{-2}-t)j|\sigma_0)\leq \exp(-2 j t^2).
\end{equation*}
Taking $t=e^{-2}\slash 2$, we obtain that
\begin{equation*}
    \mathbb{P}(N_j\leq e^{-2} j\slash 2|\sigma_0) \leq \exp(-cj),
\end{equation*}
hence
\begin{equation*}
    \mathbb{P}(N_j\leq e^{-2} j\slash 2)=\mathbb{E}[\mathbb{P}(N_j\leq e^{-2} j\slash 2|\sigma_0)]\leq \exp(-cj).
\end{equation*}

\end{proof}

\begin{proposition}\label{P3.3}
There exists a positive absolute constant $C_0\geq 1$, such that for any $j\in [n]$ and any $u\geq C_0\max\{\beta^{-1},1\}$, we have
\begin{equation*}
    \mathbb{P}(N_j\geq u)\leq 4\exp(-u\slash 8). 
\end{equation*}
\end{proposition}
\begin{proof}

If $\beta\in (0,1]$, the conclusion follows from \cite[Lemma 5.3.2]{Zhong2}. If $\beta>1$, the conclusion follows by modifying the proof of \cite[Lemma 5.3.2 and Proposition 5.3.1]{Zhong2} (replacing $\beta$ by $1$ in proper places).
 
\end{proof}

\begin{proposition}\label{P3.5}
Assume that $\beta\geq c_0$ for a fixed positive constant $c_0$. Then there exists a positive constant $C_0\geq 1$ that only depends on $c_0$, such that for any $j\in [n]$ and $u\in \mathbb{N}^{*}$,
\begin{equation*}
    \mathbb{P}(N_j\geq 1+u)\leq C_0\Big(\frac{2 e^{-2\beta}}{1-e^{-2\beta}}\Big)^{u}.
\end{equation*}
\end{proposition}

\begin{proof}

Recall Definition \ref{Def2.2}. By (\ref{Basic_n}), we have   
\begin{equation*}
    N_j\leq  1+|\mathcal{D}_{j-1}(\sigma_0)|+\sum_{\substack{k\in [j-1]:\\\sigma_0(k)\leq j-1}}\mathbbm{1}_{b_k\geq j}.
\end{equation*}
Let $X:=\sum\limits_{\substack{k\in [j-1]:\\\sigma_0(k)\leq j-1}}\mathbbm{1}_{b_k\geq j}$. We have
\begin{equation}\label{Eq3.82}
    \mathbb{P}(N_j\geq 1+u)\leq \sum_{v=0}^u \mathbb{P}(|\mathcal{D}_{j-1}(\sigma_0)|\geq u-v, X\geq v).
\end{equation}

For any $v\in \mathbb{N}^{*}$, by Lemma \ref{L3.1}, we have
\begin{eqnarray}\label{Eq3.81}
   \mathbb{P}(X\geq v|\sigma_0) &\leq& \sum_{\substack{k_1,\cdots,k_v\in [j-1]:\\\sigma_0(k_1),\cdots,\sigma_0(k_v)\leq j-1,\\k_1<\cdots<k_v}}\mathbb{P}(b_{k_1}\geq j,\cdots,b_{k_v}\geq j|\sigma_0) \nonumber\\
   &=& \sum_{\substack{k_1,\cdots,k_v\in [j-1]:\\\sigma_0(k_1),\cdots,\sigma_0(k_v)\leq j-1,\\k_1<\cdots<k_v}}\prod_{t=1}^v e^{-2\beta(j-\max\{k_t,\sigma_0(k_t)\})_{+}} \nonumber\\
   &\leq& \Big(2\sum_{k=1}^{n} e^{-2\beta k}\Big)^v\leq \Big(\frac{2 e^{-2\beta}}{1-e^{-2\beta}}\Big)^v.
\end{eqnarray}
Note that (\ref{Eq3.81}) also holds for $v=0$. Hence for any $v \in\{0\}\cup [u]$, by Proposition \ref{P2.1_n}, we have
\begin{eqnarray}
  &&  \mathbb{P}(|\mathcal{D}_{j-1}(\sigma_0)|\geq u-v, X\geq v)=\mathbb{E}[\mathbbm{1}_{|\mathcal{D}_{j-1}(\sigma_0)|\geq u-v}\mathbb{P}(X\geq v|\sigma_0)] \nonumber\\
  &\leq& \Big(\frac{2 e^{-2\beta}}{1-e^{-2\beta}}\Big)^v\mathbb{P}(|\mathcal{D}_{j-1}(\sigma_0)|\geq u-v)\leq C_1\Big(\frac{2 e^{-2\beta}}{1-e^{-2\beta}}\Big)^v e^{-2\beta(u-v)}\nonumber\\
  &=& C_1 \Big(\frac{2}{1-e^{-2\beta}}\Big)^v e^{-2\beta u},
\end{eqnarray}
where $C_1$ is a positive constant that only depends on $c_0$. Hence by (\ref{Eq3.82}),
\begin{eqnarray}
    \mathbb{P}(N_j\geq 1+u)&\leq& C_1 e^{-2\beta u} \sum_{v=0}^u \Big(\frac{2}{1-e^{-2\beta}}\Big)^v\leq   C_1 e^{-2\beta}  \Big(\frac{2}{1-e^{-2\beta}}\Big)^{u+1}  \nonumber\\
    &\leq& C_1 e^{-2\beta}  \frac{2}{1-e^{-2c_0}}\Big(\frac{2}{1-e^{-2\beta}}\Big)^{u} \leq C_0\Big(\frac{2e^{-2\beta}}{1-e^{-2\beta}}\Big)^{u},
\end{eqnarray}
where $C_0\geq 1$ is a positive constant that only depends on $c_0$.

\end{proof}

For $\sigma$ drawn from the $L^1$ model, we bound $|\sigma(j)-j|$ for any $j\in [n]$ in Propositions \ref{P3.4}-\ref{P3.6} below.

\begin{proposition}\label{P3.4}
Let $\sigma$ be drawn from $\mathbb{P}_{n,\beta}$, and let $C_0$ be the constant in Proposition \ref{P3.3}. For any $j\in [n]$, $u>0$, and $v\geq C_0\max\{\beta^{-1},1\}$, we have 
\begin{equation*}
   \mathbb{P}(|\sigma(j)-j|\geq u) \leq 8\lceil u \rceil \exp(-v\slash 8)+2\exp(-u v^{-1}).
\end{equation*}
\end{proposition}

\begin{proof}

Consider any $j\in [n]$, $u>0$, and $v\geq C_0\max\{\beta^{-1},1\}$. Let $\sigma$ be generated as in the preceding (so the distribution of $\sigma$ is given by $\mathbb{P}_{n,\beta}$). We bound $\mathbb{P}(\sigma(j)\leq j-u)$ in the following. If $\lceil u\rceil \geq j$, $\mathbb{P}(\sigma(j)\leq j-u)=0$. Below we assume that $\lceil u\rceil\leq j-1$. 

Let $\mathcal{E}_0$ be the event that $Y_l\neq j$ for any $l\in [j+1,n]\cap\mathbb{N}$. For any $l\in [j]$, let $\mathcal{E}_l$ be the event that $Y_l\neq j$. Note that
\begin{equation}\label{Eq3.8}
    \{\sigma(j)\leq j-u\}=\{\sigma(j)\leq j-\lceil u\rceil\} \subseteq \mathcal{E}_0\bigcap \Big( \bigcap_{l=j-\lceil u \rceil+1}^{j} \mathcal{E}_l\Big).  
\end{equation}
By Lemma \ref{L3.1}, we have $b_j\geq j$, hence
\begin{eqnarray*}
    \mathbb{E}\Big[\mathbbm{1}_{\mathcal{E}_0}\prod_{l=j-\lceil u\rceil +1}^j \mathbbm{1}_{\mathcal{E}_l}\Big|\mathcal{B}_{j-\lceil u \rceil+1}\Big]&=&\mathbbm{1}_{\mathcal{E}_0}\Big(\prod_{l=j-\lceil u\rceil +2}^j \mathbbm{1}_{\mathcal{E}_l}\Big)\mathbb{E}[\mathbbm{1}_{\mathcal{E}_{j-\lceil u \rceil+1}}|\mathcal{B}_{j-\lceil u \rceil+1}] \nonumber\\
    &=& \mathbbm{1}_{\mathcal{E}_0}\Big(\prod_{l=j-\lceil u\rceil +2}^j \mathbbm{1}_{\mathcal{E}_l}\Big)\Big(1-\frac{1}{N_{j-\lceil u \rceil+1}}\Big).
\end{eqnarray*}
Hence
\begin{equation*}
    \mathbb{E}\Big[\mathbbm{1}_{\mathcal{E}_0}\prod_{l=j-\lceil u\rceil +1}^j \mathbbm{1}_{\mathcal{E}_l}\Big|\mathcal{B}_j\Big]=\Big(1-\frac{1}{N_{j-\lceil u \rceil+1}}\Big)\mathbb{E}\Big[\mathbbm{1}_{\mathcal{E}_0}\Big(\prod_{l=j-\lceil u\rceil +2}^j \mathbbm{1}_{\mathcal{E}_l}\Big)\Big|\mathcal{B}_j\Big].
\end{equation*}
Continuing in the same manner, we obtain that
\begin{equation*}
     \mathbb{E}\Big[\mathbbm{1}_{\mathcal{E}_0}\prod_{l=j-\lceil u\rceil +1}^j \mathbbm{1}_{\mathcal{E}_l}\Big|\mathcal{B}_j\Big]=\mathbbm{1}_{\mathcal{E}_0}\prod_{l=j-\lceil u \rceil+1}^j \Big(1-\frac{1}{N_l}\Big),
\end{equation*}
which leads to
\begin{eqnarray}\label{Eq3.7}
 \mathbb{P}\Big(\mathcal{E}_0\bigcap \Big(\bigcap_{l=j-\lceil u \rceil+1}^{j}\mathcal{E}_l\Big)\Big|\mathcal{B}_n \Big)
 &=& \prod_{l=j-\lceil u \rceil+1}^j \Big(1-\frac{1}{N_l}\Big) \mathbb{E}[\mathbbm{1}_{\mathcal{E}_0}|\mathcal{B}_n] \nonumber\\
 &\leq& \prod_{l=j-\lceil u \rceil+1}^j \Big(1-\frac{1}{N_l}\Big).
\end{eqnarray}

Now let $\mathcal{W}$ be the event that $N_l\leq v$ for any $l\in [j-\lceil u \rceil+1,j]\cap\mathbb{N}$. By Proposition \ref{P3.3} and the unioun bound, 
\begin{eqnarray}\label{Eq3.10}
    \mathbb{P}(\mathcal{W}^c) \leq  4\lceil u \rceil \exp(-v\slash 8).
\end{eqnarray}
By (\ref{Eq3.7}), 
\begin{equation}
    \mathbb{P}\Big(\mathcal{W}\bigcap\mathcal{E}_0\bigcap \Big(\bigcap_{l=j-\lceil u \rceil+1}^{j}\mathcal{E}_l\Big)\Big|\mathcal{B}_n\Big)\leq \mathbbm{1}_{\mathcal{W}} \prod_{l=j-\lceil u \rceil+1}^j \Big(1-\frac{1}{N_l}\Big)\leq (1-v^{-1})^{\lceil u \rceil},
\end{equation}
hence
\begin{equation}\label{Eq3.9}
    \mathbb{P}\Big(\mathcal{W}\bigcap\mathcal{E}_0\bigcap \Big(\bigcap_{l=j-\lceil u \rceil+1}^{j}\mathcal{E}_l\Big)\Big)\leq (1-v^{-1})^{\lceil u \rceil}. 
\end{equation}
By (\ref{Eq3.8}), (\ref{Eq3.10}), (\ref{Eq3.9}), and the union bound,
\begin{equation}\label{Eq3.11}
    \mathbb{P}(\sigma(j)\leq j-u)\leq \mathbb{P}\Big(\mathcal{E}_0\bigcap \Big( \bigcap_{l=j-\lceil u \rceil+1}^{j} \mathcal{E}_l\Big)\Big)\leq   4\lceil u \rceil \exp(-v\slash 8)+(1-v^{-1})^{\lceil u \rceil}.
\end{equation}

Let $\bar{\sigma}\in S_n$ be such that $\bar{\sigma}(j)=n+1-\sigma(n+1-j)$ for any $j\in [n]$. As the distribution of $\sigma$ is $\mathbb{P}_{n,\beta}$, it can be checked that the distribution of $\bar{\sigma}$ is also $\mathbb{P}_{n,\beta}$. By (\ref{Eq3.11}) (with $j$ replaced by $n+1-j$),
\begin{eqnarray}\label{Eq3.12}
    \mathbb{P}(\sigma(j)\geq j+u)&=& \mathbb{P}(\bar{\sigma}(n+1-j)\leq n+1-j-u) \nonumber\\
     &=&  \mathbb{P}(\sigma(n+1-j)\leq n+1-j-u) \nonumber\\
     &  \leq & 4\lceil u \rceil \exp(-v\slash 8)+(1-v^{-1})^{\lceil u \rceil}.
\end{eqnarray}
By (\ref{Eq3.11}), (\ref{Eq3.12}), and the union bound, we conclude that
\begin{eqnarray*}
    \mathbb{P}(|\sigma(j)-j|\geq u) &\leq& 8\lceil u \rceil \exp(-v\slash 8)+2(1-v^{-1})^{\lceil u \rceil}\nonumber\\
    &\leq&  8\lceil u \rceil \exp(-v\slash 8)+2\exp(-u v^{-1}) .
\end{eqnarray*}

\end{proof}

\begin{proposition}\label{P3.6}
Assume that $\beta\geq 2$, and let $\sigma$ be drawn from $\mathbb{P}_{n,\beta}$. For any $j\in [n]$ and $u\in \mathbb{N}^{*}$, we have
\begin{equation*}
   \mathbb{P}(|\sigma(j)-j|\geq  u)\leq C\exp(-(2\beta+\sqrt{u}\slash 2)).
\end{equation*}
\end{proposition}
\begin{proof}

Assume that $\beta\geq 2$, and consider any $j\in [n]$ and $u\in \mathbb{N}^{*}$. Let $\sigma$ be generated as in the preceding (so the distribution of $\sigma$ is given by $\mathbb{P}_{n,\beta}$). We bound $\mathbb{P}(\sigma(j)\leq j-u)$ in the following. If $u\geq j$, $\mathbb{P}(\sigma(j)\leq j-u)=0$. Below we assume that $u\leq j-1$. 

Following the argument between (\ref{Eq3.8}) and (\ref{Eq3.7}), we obtain that
\begin{equation}\label{Eq3.83}
    \mathbb{P}(\sigma(j)\leq j-u|\mathcal{B}_n)\leq  \prod_{l=j- u 
 +1}^j \Big(1-\frac{1}{N_l}\Big). 
\end{equation}
For any $v\in \mathbb{N}^{*}$, let $\mathcal{W}_v$ be the event that $N_l\leq v$ for any $l\in [j-u+1,j]\cap \mathbb{N}$. By Proposition \ref{P3.5} and the union bound, 
\begin{equation}\label{Eq3.84}
    \mathbb{P}((\mathcal{W}_v)^c) \leq \sum_{l=j-u+1}^j \mathbb{P}(N_l\geq v+1) \leq C u \Big(\frac{2 e^{-2\beta}}{1-e^{-2\beta}}\Big)^{v}.
\end{equation}
By (\ref{Eq3.83}) and Proposition \ref{P3.5}, 
\begin{eqnarray}\label{Eq3.85}
  &&  \mathbb{P}(\{\sigma(j)\leq j-u\}\cap\mathcal{W}_v)=\mathbb{E}[\mathbbm{1}_{\mathcal{W}_v}\mathbb{P}(\sigma(j)\leq j-u|\mathcal{B}_n)] \nonumber\\
  &\leq& \mathbb{E}\Big[\mathbbm{1}_{\mathcal{W}_v} \prod_{l=j- u 
 +1}^j \Big(1-\frac{1}{N_l}\Big)\Big]\leq \mathbb{E}[(1-v^{-1})^u\mathbbm{1}_{N_j\geq 2}]\nonumber\\
 &=& (1-v^{-1})^u \mathbb{P}(N_j\geq 2)\leq C(1-v^{-1})^u e^{-2\beta }.
\end{eqnarray}
By (\ref{Eq3.84}), (\ref{Eq3.85}), and the union bound, for any $v\in \mathbb{N}^{*}$, 
\begin{eqnarray*}
    \mathbb{P}(\sigma(j)\leq j-u) &\leq& Cu \Big(\frac{2 e^{-2\beta}}{1-e^{-2\beta}}\Big)^{v}+C(1-v^{-1})^u e^{-2\beta }\nonumber\\
    &\leq& Cu(e^{1-2\beta})^v+C\exp(-u v^{-1})e^{-2\beta}\nonumber\\
    &\leq& Ce^{-2\beta}(ue^{-v}+\exp(-u v^{-1})),
\end{eqnarray*}
where we note that $(2\beta-1)v\geq 2\beta+v-2$ in the third line. Taking $v=\lfloor \sqrt{u} \rfloor\in \mathbb{N}^{*}$, we obtain that
\begin{equation*}
    \mathbb{P}(\sigma(j)\leq j-u)\leq Cu\exp(-(2\beta+\sqrt{u}))\leq C\exp(-(2\beta+\sqrt{u}\slash 2)). 
\end{equation*}

Arguing similarly as in (\ref{Eq3.12}), we obtain the conclusion of the proposition.

\end{proof}

Throughout the rest of the section, we assume the setup given in this subsection. 

\subsection{Proof of the upper bound in Theorem \ref{Thm1.3.1}}\label{Sect.3.2}

In this subsection, we give the proof of the upper bound in Theorem \ref{Thm1.3.1}. Assuming an auxiliary lemma (Lemma \ref{L3.5}), we present this proof in Section \ref{Sect.3.2.1}. The proof of Lemma \ref{L3.5} is given in Section \ref{Sect.3.2.3}.

\subsubsection{Proof of the upper bound in Theorem \ref{Thm1.3.1}}\label{Sect.3.2.1}

In the following, we fix an arbitrary $s\in [n]$. For any $t\in \mathbb{N}$, we denote $W_t^{(s)},Z_t^{(s)}$ by $W_t,Z_t$, respectively (recall the notations in Section \ref{Sect.3.1}).

We note that $W_n=0$, and let $T:=\min\{t\in [n]: W_t=0\}$. We have $\min(\mathcal{C}_s(\sigma))=Z_{T-1}$ and $Z_{T-1}<\cdots<Z_0$. Hence
\begin{equation}\label{Eq5.69}
  \mathbb{E}[s-\min(\mathcal{C}_s(\sigma))]=\mathbb{E}[s-Z_{T-1}]=\sum_{t=1}^n \mathbb{E}[(s-Z_{t-1})\mathbbm{1}_{T=t}].
\end{equation}
Let $\mathcal{F}_0:=\mathcal{B}_n$. For any $t\in \mathbb{N}^{*}$, let $\mathcal{F}_t$ be the $\sigma$-algebra generated by $\sigma_0$, $\{b_j\}_{j=1}^n$, $\{W_j\}_{j=1}^t$, and $\{Z_j\}_{j=1}^{t}$.

For any $l\in [n]$, we denote $\mathcal{H}_l:=\{j\in [n]: j\leq l-1, b_j \geq l \}$. We let $\mathcal{O}_l$ be the set of open arcs $\mathbf{a}=(a_1,\cdots,a_s)$ (where $s\in \mathbb{N}^{*}$) at step $l+1$ such that $a_1 \in \{j\in [n]:b_j\geq l\}\backslash \{Y_{l+1},\cdots,Y_n\}$, and denote by $\mathcal{H}_l'$ the set of tails of the open arcs in $\mathcal{O}_l$. 

\begin{lemma}\label{L3.2}
For any $l\in [n]$, $|\mathcal{H}_l|=N_l-1$, $|\mathcal{H}_l'|=N_l$, and $\mathcal{H}_l'=\mathcal{H}_l\cup \{l\}$.
\end{lemma}

\begin{proof}
Consider any $j\in \mathcal{H}_l\cup\{l\}$. As $j\leq l$, $j$ is the tail of some open arc at step $l+1$. We denote this open arc by $\mathbf{a}=(a_1,a_2,\cdots,a_q)$, where $q\in \mathbb{N}^{*}$ and $a_q=j$. Note that $b_j\geq l$ (if $j\in \mathcal{H}_l$, this follows from the definition of $\mathcal{H}_l$; if $j=l$, by Lemma \ref{L3.1}, $b_j\geq j=l$). If $q\geq 2$, then $a_1\geq l+1$, and by Lemma \ref{L3.1}, $b_{a_1}\geq a_1\geq l+1$; if $q=1$, then $a_1=j$ and $b_{a_1}=b_j\geq l$. As $\mathbf{a}$ is an open arc at step $l+1$, $a_1\notin \{Y_{l+1},\cdots,Y_n\}$. Hence $a_1\in \{j\in [n]:b_j\geq l\}\backslash \{Y_{l+1},\cdots,Y_n\}$ and $j=a_q\in \mathcal{H}_l'$. Therefore, we have $\mathcal{H}_l\cup \{l\}\subseteq \mathcal{H}_l'$. 

By Lemma \ref{L3.1}, we have 
\begin{eqnarray}\label{Eq3.4}
    |\mathcal{H}_l|&=&|\{j\in [n]:b_j\geq l\}|-|\{j\in [n]:j\geq l,b_j\geq l\}|\nonumber\\
    &=& |\{j\in [n]:b_j\geq l\}|-|\{j\in [n]:j\geq l\}|\nonumber\\
    &=& |\{j\in [n]:b_j\geq l\}|-(n-l+1).
\end{eqnarray}
Moreover, by (\ref{DefinitionNj}), 
\begin{eqnarray}\label{Eq3.5}
    |\mathcal{H}_l'|&=&|\mathcal{O}_l|=|\{j\in [n]:b_j\geq l\}\backslash \{Y_{l+1},\cdots,Y_n\}|\nonumber\\
    &=& N_l=|\{j\in [n]:b_j\geq l\}|-(n-l).
\end{eqnarray}
Comparing (\ref{Eq3.4}) and (\ref{Eq3.5}), we obtain that $|\mathcal{H}_l'|=|\mathcal{H}_l|+1=|\mathcal{H}_l\cup \{l\}|$. Hence $|\mathcal{H}_l|=N_l-1$, $|\mathcal{H}_l'|=N_l$, and $\mathcal{H}_l'=\mathcal{H}_l\cup \{l\}$.

\end{proof}

By Lemma \ref{L3.2}, for any $t\in [T-1]$,
\begin{equation}\label{Eq3.57}
    Z_t\in \mathcal{H}'_{Z_{t-1}}\backslash\{Z_{t-1}\}=\mathcal{H}_{Z_{t-1}}.
\end{equation}
We also have the following lemma.

\begin{lemma}\label{L3.6}
For any $t\in\mathbb{N}^{*}$ and $a\in [n]$,
\begin{equation}\label{EEq3.1}
    \mathbb{P}(Z_t=a, W_t=1|\mathcal{F}_{t-1})=\frac{\mathbbm{1}_{a\in \mathcal{H}_{Z_{t-1}}}\mathbbm{1}_{W_{t-1}=1}}{N_{Z_{t-1}}},
\end{equation}
\begin{equation}\label{c2}
    \mathbb{P}(W_t=0|\mathcal{F}_{t-1})=\frac{\mathbbm{1}_{W_{t-1}=1}}{N_{Z_{t-1}}}+\mathbbm{1}_{W_{t-1}=0}.
\end{equation}
Note that by Lemma \ref{L3.2}, (\ref{EEq3.1}) implies that
\begin{equation*}
    \mathbb{P}(W_t=1|\mathcal{F}_{t-1})=\Big(1-\frac{1}{N_{Z_{t-1}}}\Big)\mathbbm{1}_{W_{t-1}=1}.
\end{equation*}
\end{lemma}

\begin{proof}

For any $t\in\mathbb{N}^{*}$ and $a,l\in [n]$, we have $\{Z_{t-1}=l,W_{t-1}=1\}\in\mathcal{B}_l$ and 
\begin{equation*}
    \mathbb{P}(Z_t=a,W_t=1|\mathcal{B}_l)\mathbbm{1}_{Z_{t-1}=l,W_{t-1}=1}=\frac{\mathbbm{1}_{a\in \mathcal{H}_l}}{N_l}\mathbbm{1}_{Z_{t-1}=l,W_{t-1}=1}.
\end{equation*}
Consider any $A\in \mathcal{F}_{t-1}$. We identify $\sigma_0$ with $(\sigma_0(1),\cdots,\sigma_0(n))\in \mathbb{R}^n$. There exists a Borel measurable function $g:\mathbb{R}^{2(n+t-1)}\rightarrow\mathbb{R}$, such that 
\begin{equation*}
    \mathbbm{1}_A=g(\sigma_0,b_1,\cdots,b_n,W_1,\cdots,W_{t-1},Z_1,\cdots,Z_{t-1}).
\end{equation*}
We assume that $g\in [0,1]$ (otherwise we can replace $g$ by $\max\{0,\min\{g,1\}\}$). For any $t\in\mathbb{N}^{*}$ and $a\in [n]$,
\begin{eqnarray*}
 &&\mathbb{E}[\mathbbm{1}_{Z_t=a,W_t=1} \mathbbm{1}_A]=\mathbb{E}[\mathbbm{1}_{Z_t=a,W_t=1}g(\sigma_0,b_1,\cdots,b_n,W_1,\cdots,W_{t-1},Z_1,\cdots,Z_{t-1})] \nonumber\\
 &=& \sum_{l=1}^n \mathbb{E}[\mathbbm{1}_{Z_t=a,W_t=1}\mathbbm{1}_{Z_{t-1}=l,W_{t-1}=1}g(\sigma_0,b_1,\cdots,b_n,W_1,\cdots,W_{t-1},Z_1,\cdots,Z_{t-1})] \nonumber\\
 &=& \sum_{l=1}^n \mathbb{E}[\mathbb{P}(Z_t=a,W_t=1|\mathcal{B}_{l})\mathbbm{1}_{Z_{t-1}=l,W_{t-1}=1}g(\sigma_0,b_1,\cdots,b_n,W_1,\cdots,W_{t-1},Z_1,\cdots,Z_{t-1})] \nonumber\\
 &=& \sum_{l=1}^n \mathbb{E}\Big[\frac{\mathbbm{1}_{a\in \mathcal{H}_l}}{N_l} 
\mathbbm{1}_{Z_{t-1}=l,  W_{t-1}=1}g(\sigma_0,b_1,\cdots,b_n,W_1,\cdots,W_{t-1},Z_1,\cdots,Z_{t-1})\Big] \nonumber\\
 &=& \sum_{l=1}^n \mathbb{E}\Big[\frac{\mathbbm{1}_{a\in \mathcal{H}_{Z_{t-1}}}}{N_{Z_{t-1}}} 
\mathbbm{1}_{Z_{t-1}=l,W_{t-1}=1}\mathbbm{1}_A \Big] =
\mathbb{E}\Big[\frac{\mathbbm{1}_{a\in \mathcal{H}_{Z_{t-1}}}\mathbbm{1}_{W_{t-1}=1}}{N_{Z_{t-1}}}\mathbbm{1}_A\Big],
\end{eqnarray*}
where we use the fact that $W_t=1$ implies $W_{t-1}=1$ in the second line and note that $\mathbbm{1}_{Z_{t-1}=l,W_{t-1}=1}g(\sigma_0,b_1,\cdots,b_n,W_1,\cdots,W_{t-1},Z_1,\cdots,Z_{t-1})$ is $\mathcal{B}_{l}$-measurable in the third line. As $\mathbbm{1}_{a\in\mathcal{H}_{Z_{t-1}}}\slash N_{Z_{t-1}}$ is $\mathcal{F}_{t-1}$-measurable, we have
\begin{equation*}
    \mathbb{P}(Z_t=a,W_t=1|\mathcal{F}_{t-1})=\frac{\mathbbm{1}_{a\in \mathcal{H}_{Z_{t-1}}}\mathbbm{1}_{W_{t-1}=1}}{N_{Z_{t-1}}}.
\end{equation*}

Now note that for any $t\in\mathbb{N}^{*}$ and $l\in [n]$,
\begin{equation*}
  \mathbb{P}(W_t=0|\mathcal{B}_{l})\mathbbm{1}_{Z_{t-1}=l,W_{t-1}=1}=\frac{1}{N_l}\mathbbm{1}_{Z_{t-1}=l,W_{t-1}=1}. 
\end{equation*}
Arguing similarly as above, we obtain (\ref{c2}).
    
\end{proof}

For any $t\in [n]$ and $r>0$, we define $\mathcal{K}_{t,r}$ to be the event that for any $l\in [s-r t \min\{\beta,1\}^{-4} ,s]\cap [n]$, $\max_{j\in \mathcal{H}_l} \{l-j\}\leq r$ and $N_l\leq r\min\{\beta,1\}^{-1}$. Note that $\mathcal{K}_{t,r}\in \mathcal{B}_n$.

\begin{lemma}\label{L3.3}
For any $t,t'\in [n]$ satisfying $t\geq t'$ and any $r>0$, when the event $\{T\geq t'\}\cap \mathcal{K}_{t,r}$ holds, we have $s-Z_k\leq rk$ for any $k\in \{0\}\cup [t'-1]$.
\end{lemma}
\begin{proof}

We prove by induction on $k$. When $k=0$, $s-Z_0=0$. Now assume that $k\in [t'-1]$ and $s-Z_{k-1}\leq r(k-1)$. We have $Z_{k-1}\geq s-rt\geq s-rt\min\{\beta,1\}^{-4}$, hence $Z_{k-1}\in [s-r t \min\{\beta,1\}^{-4} ,s]\cap [n]$. By (\ref{Eq3.57}), $Z_k\in \mathcal{H}_{Z_{k-1}}$, hence 
\begin{equation*}
    Z_{k-1}-Z_k\leq \max_{j\in \mathcal{H}_{Z_{k-1}}}\{ Z_{k-1}-j \}\leq r,
\end{equation*}
and $s-Z_k=(s-Z_{k-1})+(Z_{k-1}-Z_k)\leq rk$. Therefore, $s-Z_k\leq rk$ for any $k\in \{0\}\cup [t'-1]$. 
    
\end{proof}

\begin{lemma}\label{L3.8}
For any $t\in [n]$, $r\geq 1$, and $A\in \mathcal{B}_n$ such that $A\subseteq \mathcal{K}_{t,r}$, we have 
\begin{equation*}
    \mathbb{E}[(s-Z_{t-1})\mathbbm{1}_{T=t}\mathbbm{1}_{A}]\leq  r(t-1)(1-\min\{\beta,1\}r^{-1})^{t-1}\mathbb{P}(A\cap\{N_s\geq 2\}).
\end{equation*}
\end{lemma}

\begin{proof}

When $t=1$, $\mathbb{E}[(s-Z_{t-1})\mathbbm{1}_{T=t}\mathbbm{1}_{A}]=0$. Below we consider the case where $t\geq 2$. Note that $N_s=1$ implies $W_1=0$ and $T=1$. Hence $\{T=t\}\subseteq \{N_s\geq 2\}$. By Lemma \ref{L3.3}, as $A\subseteq \mathcal{K}_{t,r}$,
\begin{eqnarray}\label{Eq3.60}
   \mathbb{E}[(s-Z_{t-1})\mathbbm{1}_{T=t}\mathbbm{1}_{A}]&\leq& r(t-1)\mathbb{E}[\mathbbm{1}_{T=t}\mathbbm{1}_{A}]= r(t-1)\mathbb{E}[\mathbbm{1}_{T=t}\mathbbm{1}_{A}\mathbbm{1}_{N_s\geq 2}]  \nonumber\\
  &\leq& r(t-1)  \mathbb{E}[\mathbbm{1}_{A\cap\{N_s\geq 2\}}\mathbbm{1}_{W_1=1}\cdots\mathbbm{1}_{W_{t-1}=1}].
\end{eqnarray}
By Lemma \ref{L3.6}, we have
\begin{eqnarray}\label{Eq3.58}
&&  \mathbb{E}[\mathbbm{1}_{A\cap\{N_s\geq 2\}}\mathbbm{1}_{W_1=1}\cdots\mathbbm{1}_{W_{t-1}=1}]  \nonumber\\
&=& \mathbb{E}[\mathbbm{1}_{A\cap\{N_s\geq 2\}}\mathbbm{1}_{W_1=1}\cdots\mathbbm{1}_{W_{t-2}=1}\mathbb{P}(W_{t-1}=1|\mathcal{F}_{t-2})]\nonumber\\
  &=& \mathbb{E}\Big[\mathbbm{1}_{A\cap\{N_s\geq 2\}}\mathbbm{1}_{W_1=1}\cdots\mathbbm{1}_{W_{t-2}=1}\Big(1-\frac{1}{N_{Z_{t-2}}}\Big)\Big].
\end{eqnarray}
When the event $A\cap \{W_{t-2}=1\}$ holds, the event $\mathcal{K}_{t,r}\cap \{T\geq t-1\}$ also holds, hence by Lemma \ref{L3.3}, $s-Z_{t-2}\leq r(t-2)$ and $Z_{t-2}\in [s-rt\min\{\beta,1\}^{-4},s]\cap [n]$; by the definition of $\mathcal{K}_{t,r}$, $N_{Z_{t-2}}\leq r\min\{\beta,1\}^{-1}$, hence
\begin{equation}\label{Eq3.59}
    \mathbbm{1}_{A\cap\{N_s\geq 2\}}\mathbbm{1}_{W_{t-2}=1} \Big(1-\frac{1}{N_{Z_{t-2}}}\Big)\leq (1-\min\{\beta,1\}r^{-1})\mathbbm{1}_{A\cap\{N_s\geq 2\}}\mathbbm{1}_{W_{t-2}=1}.
\end{equation}
By (\ref{Eq3.58}) and (\ref{Eq3.59}),
\begin{eqnarray}
  &&  \mathbb{E}[\mathbbm{1}_{A\cap\{N_s\geq 2\}}\mathbbm{1}_{W_1=1}\cdots\mathbbm{1}_{W_{t-1}=1}] \nonumber\\
  &\leq& (1-\min\{\beta,1\}r^{-1})\mathbb{E}[\mathbbm{1}_{A\cap\{N_s\geq 2\}}\mathbbm{1}_{W_1=1}\cdots\mathbbm{1}_{W_{t-2}=1}].
\end{eqnarray}

Arguing similarly as above, we obtain that
\begin{equation}\label{Eq3.62}
    \mathbb{E}[\mathbbm{1}_{A\cap\{N_s\geq 2\}}\mathbbm{1}_{W_1=1}\cdots\mathbbm{1}_{W_{t-1}=1}]\leq (1-\min\{\beta,1\}r^{-1})^{t-1} \mathbb{P}(A\cap\{N_s\geq 2\}).
\end{equation}
By (\ref{Eq3.60}) and (\ref{Eq3.62}), we conclude that
\begin{equation}
    \mathbb{E}[(s-Z_{t-1})\mathbbm{1}_{T=t}\mathbbm{1}_{A}]\leq r(t-1)(1-\min\{\beta,1\}r^{-1})^{t-1}\mathbb{P}(A\cap\{N_s\geq 2\}).
\end{equation}

\end{proof}

\begin{lemma}\label{L3.7}
There exists a positive absolute constant $C_0\geq 2$, such that for any $t\in [n]$ and $r\geq C_0\max\{\beta^{-1},1\}$, we have
\begin{equation*}
    \mathbb{P}((\mathcal{K}_{t,r})^c) \leq  Cr^2t\max\{\beta^{-6},1\}\exp(-c\min\{\beta^{1\slash 2},  1\} \sqrt{r}
    ).
\end{equation*}
\end{lemma}

\begin{proof}

We let $C_0$ be the constant appearing in Proposition \ref{P3.3}. Without loss of generality, we assume that $C_0\geq 2$. By the union bound, we have 
\begin{eqnarray}\label{Eq5.66}
  &&  \mathbb{P}((\mathcal{K}_{t,r})^c) \nonumber\\
  &\leq& \sum_{l\in [s-r t \min\{\beta,1\}^{-4} ,s]\cap [n]} (\mathbb{P}(\max_{j\in\mathcal{H}_l}\{l-j\}>r)+\mathbb{P}(N_l> r\min\{\beta,1\}^{-1})). \nonumber\\
  && 
\end{eqnarray} 

By Proposition \ref{P3.3}, as $r\geq C_0$, for any $l\in [s-r t \min\{\beta,1\}^{-4} ,s]\cap [n]$,
\begin{equation}\label{Eq5.67}
    \mathbb{P}(N_l>r\min\{\beta,1\}^{-1})\leq C\exp(-c r \max\{\beta^{-1},1\}). 
\end{equation}

Consider any $l\in [s-r t \min\{\beta,1\}^{-4} ,s]\cap [n]$. Note that when the event $\max_{j\in\mathcal{H}_l}\{l-j\}>r$ holds, there exists some $j\in [n]$ such that $j\leq l-r$ and $b_j\geq l$. Hence by the union bound and Lemma \ref{L3.1}, 
\begin{equation}\label{Eq3.86}
    \mathbb{P}(\max_{j\in\mathcal{H}_l}\{l-j\}>r|\sigma_0) \leq\sum_{\substack{j\in [n]:\\ j\leq l-r}} \mathbb{P}(b_j\geq l|\sigma_0) =\sum_{\substack{j\in [n]:\\ j\leq l-r}} e^{-2\beta(l-\max\{j,\sigma_0(j)\})_{+}}.
\end{equation}
Hence
\begin{eqnarray}\label{Eq3.63}
   && \mathbb{P}(\max_{j\in\mathcal{H}_l}\{l-j\}>r)
    \leq \sum_{\substack{j\in [n]:\\ j\leq l-r}} \mathbb{E}[e^{-2\beta(l-\max\{j,\sigma_0(j)\})_{+}}] \nonumber\\
    &\leq& \sum_{\substack{j\in [n]:\\ j\leq l-r}}(\mathbb{E}[e^{-2\beta(l-\max\{j,\sigma_0(j)\})_{+}}\mathbbm{1}_{\sigma_0(j)\leq l-r\slash 2}]+\mathbb{P}(\sigma_0(j)>l-r\slash 2))\nonumber\\
    &\leq& \sum_{\substack{j\in [n]:\\ j\leq l-r}} e^{-2\beta(l-j)}+\sum_{\substack{j\in [n]:\\ j\leq l-r\slash 2}}e^{-2\beta(l-j)}+\sum_{\substack{j\in [n]:\\ j\leq l-r}}\mathbb{P}(\sigma_0(j)>l-r\slash 2).\nonumber\\
    &&
\end{eqnarray}
Note that
\begin{eqnarray}\label{Eq3.64}
  && \sum_{\substack{j\in [n]:\\ j\leq l-r}} e^{-2\beta(l-j)}+\sum_{\substack{j\in [n]:\\ j\leq l-r\slash 2}}e^{-2\beta(l-j)}\leq 2\sum_{j=\lceil r\slash 2 \rceil}^{\infty} e^{-2\beta j}\nonumber\\
 &=&\frac{2e^{-2\beta\lceil r\slash 2\rceil}}{1-e^{-2\beta}}\leq \frac{2e^{-\beta r}}{1-e^{-\min\{\beta,1\}}}\leq C\max\{\beta^{-1},1\}e^{-\beta r}.
\end{eqnarray}
For any $j\in [n]$ such that $j\leq l-r$, by Proposition \ref{P3.4} (where we take $u=l-j-r\slash 2$ and $v=C_0\max\{\beta^{-1\slash 2},1\}\sqrt{u}$; as $r\geq 2\max\{\beta^{-1},1\}$, we have $v\geq C_0\max\{\beta^{-1},1\}$),
\begin{eqnarray*}
  &&  \mathbb{P}(\sigma_0(j)>l-r\slash 2)\leq \mathbb{P}(|\sigma_0(j)-j|>l-j-r\slash 2)   \nonumber\\
  &\leq& C(l-j-r\slash 2)\exp(-c\min\{\beta^{1\slash 2},1\}\sqrt{l-j-r\slash 2}).
\end{eqnarray*}
Hence
\begin{eqnarray}\label{Eq3.65}
  &&  \sum_{\substack{j\in [n]:\\ j\leq l-r}}\mathbb{P}(\sigma_0(j)>l-r\slash 2)  \nonumber\\
  &\leq& C\sum_{\substack{j\in [n]:\\ j\leq l-r}}(l-j-r\slash 2)\exp(-c\min\{\beta^{1\slash 2},1\}\sqrt{l-j-r\slash 2})\nonumber\\
  &\leq& C\sum_{j=0}^{\infty} (r\slash 2+j) \exp(-c\min\{\beta^{1\slash 2},1\}\sqrt{r\slash 2+j}) \nonumber\\
  &\leq& Cr\exp(-c\min\{\beta^{1\slash 2},1\}\sqrt{r})\Big(1+\sum_{j=1}^{\infty} j \exp(-c\min\{\beta^{1\slash 2},1\}\sqrt{j})\Big)\nonumber\\
  &\leq&  Cr\exp(-c\min\{\beta^{1\slash 2},1\}\sqrt{r})\Big(\sum_{j=1}^{\infty} j^3 \exp(-c\min\{\beta^{1\slash 2},1\}(j-1))\Big)\nonumber\\
  &\leq& \frac{Cr\exp(-c\min\{\beta^{1\slash 2},1\}\sqrt{r})}{(1-\exp(-c\min\{\beta^{1\slash 2},1\}))^4}\nonumber\\
  &\leq& Cr\max\{\beta^{-2},  1\}\exp(-c\min\{\beta^{1\slash 2},1\}\sqrt{r}).
\end{eqnarray}
By (\ref{Eq3.63})-(\ref{Eq3.65}), as $r\geq 2\max\{\beta^{-1},1\}$, we have
\begin{equation}\label{Eq5.68}
    \mathbb{P}(\max_{j\in\mathcal{H}_l}\{l-j\}>r)\leq Cr\max\{\beta^{-2},  1\}\exp(-c\min\{\beta^{1\slash 2},  1\} \sqrt{r}
    ).
\end{equation}

By (\ref{Eq5.66}), (\ref{Eq5.67}), and (\ref{Eq5.68}), we conclude that
\begin{equation}
    \mathbb{P}((\mathcal{K}_{t,r})^c)\leq Cr^2t\max\{\beta^{-6},1\}\exp(-c\min\{\beta^{1\slash 2},  1\} \sqrt{r}
    ).
\end{equation}

\end{proof}

\begin{lemma}\label{L3.9}
Assume that $\beta\geq 2$. For any $t\in [n]$ and $r\geq 2$, we have
\begin{equation*}
    \mathbb{P}((\mathcal{K}_{t,r})^c )\leq Crt e^{-2\beta-c\sqrt{r}}.
\end{equation*}
\end{lemma}

\begin{proof}

By the union bound, we have
\begin{equation}\label{Eq3.90}
   \mathbb{P}((\mathcal{K}_{t,r})^c)
  \leq \sum_{l\in [s-r t,s]\cap [n]} (\mathbb{P}(\max_{j\in\mathcal{H}_l}\{l-j\}>r)+\mathbb{P}(N_l> r)). 
\end{equation} 

By Proposition \ref{P3.5}, for any $l\in [s-rt,s]\cap [n]$, 
\begin{equation}\label{Eq3.91}
    \mathbb{P}(N_l>r)\leq \mathbb{P}(N_l\geq \lceil r\rceil)  \leq C\Big(\frac{2e^{-2\beta}}{1-e^{-2\beta}}\Big)^{\lceil r\rceil-1}\leq Ce^{-2\beta-r}.
\end{equation}

Consider any $l\in [s-rt,s]\cap [n]$. Following the argument in (\ref{Eq3.86}) and (\ref{Eq3.63}), we obtain that 
\begin{eqnarray}\label{Eq3.87}
\mathbb{P}(\max_{j\in\mathcal{H}_l}\{l-j\}>r) 
&\leq& \sum_{\substack{j\in [n]:\\j\leq  l-r}} e^{-2\beta(l-j)}+\sum_{\substack{j\in [n]:\\j\leq  l-r\slash 2}}e^{-2\beta(l-j)}\nonumber\\
&& +\sum_{\substack{j\in [n]:\\j\leq  l-r}}\mathbb{P}(\sigma_0(j)>l-r\slash 2).
\end{eqnarray}
Note that 
\begin{eqnarray}\label{Eq3.88}
  && \sum_{\substack{j\in [n]:\\j\leq  l-r}} e^{-2\beta(l-j)}+\sum_{\substack{j\in [n]:\\j\leq  l-r\slash 2}}e^{-2\beta(l-j)}\leq 2\sum_{j=\lceil r\slash 2 \rceil}^{\infty} e^{-2\beta j}\nonumber\\
 &=&\frac{2e^{-2\beta\lceil r\slash 2\rceil}}{1-e^{-2\beta}}\leq 4e^{-\beta r}\leq C e^{-2\beta-r}.
\end{eqnarray}
For any $j\in [n]$ such that $j\leq l-r$, by Proposition \ref{P3.6} (taking $u=l-j-\lfloor r\slash 2 \rfloor\geq r-\lfloor r\slash 2 \rfloor   \geq 1$), 
\begin{eqnarray*}
   \mathbb{P}(\sigma_0(j)>l-r\slash 2)&\leq& \mathbb{P}(|\sigma_0(j)-j|\geq l-j-\lfloor r\slash 2 \rfloor) \nonumber\\
   &\leq& C e^{-2\beta-\sqrt{l-j-\lfloor r\slash 2\rfloor}\slash 2}.
\end{eqnarray*}
Hence
\begin{eqnarray}\label{Eq3.89}
   \sum_{\substack{j\in [n]:\\j\leq  l-r}} \mathbb{P}(\sigma_0(j)>l-r\slash 2)&\leq& Ce^{-2\beta}\sum_{\substack{j\in [n]:\\j\leq  l-r}} e^{-\sqrt{l-j-\lfloor r\slash 2\rfloor}\slash 2} \nonumber\\
  &\leq& Ce^{-2\beta}\sum_{j=\lceil r\slash 2\rceil}^{\infty} e^{-\sqrt{j}\slash 2}\leq  Ce^{-2\beta-c\sqrt{r}}.
\end{eqnarray}
By (\ref{Eq3.87})-(\ref{Eq3.89}), we have
\begin{equation}\label{Eq3.92}
    \mathbb{P}(\max_{j\in\mathcal{H}_l}\{l-j\}>r) \leq Ce^{-2\beta-c\sqrt{r}}. 
\end{equation}

By (\ref{Eq3.90}), (\ref{Eq3.91}), and (\ref{Eq3.92}), we conclude that
\begin{equation}
    \mathbb{P}((\mathcal{K}_{t,r})^c)\leq Crt e^{-2\beta-c\sqrt{r}}. 
\end{equation}
    
\end{proof}

Let $C_0\geq 2$ be the constant appearing in Lemma \ref{L3.7}. Note that for any $t\in [n]$ and $r\geq n$, $\mathcal{K}_{t,r}=\Omega$. Let $l_0\in \mathbb{N}$ be such that $2^{l_0+1} C_0\max\{\beta^{-2},1\}\geq n$. For any $t\in [n]$, we have 
\begin{eqnarray*}
  &&  \mathbb{E}[(s-Z_{t-1})\mathbbm{1}_{T=t}]= \mathbb{E}[(s-Z_{t-1})\mathbbm{1}_{T=t}\mathbbm{1}_{\mathcal{K}_{t, 2^{l_0+1}C_0\max\{\beta^{-2},1\}}}]   \nonumber\\
  &\leq& \sum_{l=0}^{l_0} \mathbb{E}[(s-Z_{t-1})\mathbbm{1}_{T=t}\mathbbm{1}_{\mathcal{K}_{t,2^{l+1}C_0\max\{\beta^{-2},1\}}\backslash \mathcal{K}_{t,2^{l}C_0\max\{\beta^{-2},1\}}}]\nonumber\\
  && +\mathbb{E}[(s-Z_{t-1})\mathbbm{1}_{T=t}\mathbbm{1}_{\mathcal{K}_{t,C_0\max\{\beta^{-2},1\}}}].
\end{eqnarray*}
Hence by (\ref{Eq5.69}),
\begin{eqnarray}\label{Eq3.95}
  &&   \mathbb{E}[s-\min(\mathcal{C}_s(\sigma))] \nonumber\\
  &\leq& \sum_{t=1}^n\sum_{l=0}^{l_0} \mathbb{E}[(s-Z_{t-1})\mathbbm{1}_{T=t}\mathbbm{1}_{\mathcal{K}_{t,2^{l+1}C_0\max\{\beta^{-2},1\}}\backslash \mathcal{K}_{t,2^{l}C_0\max\{\beta^{-2},1\}}}] \nonumber\\
  && + \sum_{t=1}^n \mathbb{E}[(s-Z_{t-1})\mathbbm{1}_{T=t}\mathbbm{1}_{\mathcal{K}_{t,C_0\max\{\beta^{-2},1\}}}].
\end{eqnarray}  

By Lemmas \ref{L3.8} and \ref{L3.7}, for any $t\in [n]$ and $r\geq C_0\max\{\beta^{-1},1\}$, we have
\begin{eqnarray*}
  &&  \mathbb{E}[(s-Z_{t-1})\mathbbm{1}_{T=t}\mathbbm{1}_{\mathcal{K}_{t,2r}\backslash \mathcal{K}_{t,r}}] \nonumber\\
  &\leq& 2r(t-1)\Big(1-\frac{\min\{\beta,1\}}{2r}\Big)^{t-1}\mathbb{P}(\mathcal{K}_{t,2r}\backslash \mathcal{K}_{t,r}) \nonumber\\
  &\leq& Cr^3t^2\max\{\beta^{-6},1\}\Big(1-\frac{\min\{\beta,1\}}{2r}\Big)^{t-1} \exp(-c\min\{\beta^{1\slash 2},1\}\sqrt{r}). 
 \nonumber\\
 &&
\end{eqnarray*}
Thus for any $r\geq C_0\max\{\beta^{-1},1\}$,
\begin{eqnarray*}
 &&  \sum_{t=1}^n \mathbb{E}[(s-Z_{t-1})\mathbbm{1}_{T=t}\mathbbm{1}_{\mathcal{K}_{t,2r}\backslash \mathcal{K}_{t,r}}]\nonumber\\
 &\leq& Cr^3\max\{\beta^{-6},1\} \exp(-c\min\{\beta^{1\slash 2},1\}\sqrt{r})\Big(\sum_{t=1}^n t^2\Big(1-\frac{\min\{\beta,1\}}{2r}\Big)^{t-1}\Big)\nonumber\\
 &\leq& Cr^6\max\{\beta^{-9},1\}\exp(-c\min\{\beta^{1\slash 2},1\}\sqrt{r}).
\end{eqnarray*}
Hence we have
\begin{eqnarray}\label{Eq3.94}
    && \sum_{t=1}^n\sum_{l=0}^{l_0} \mathbb{E}[(s-Z_{t-1})\mathbbm{1}_{T=t}\mathbbm{1}_{\mathcal{K}_{t,2^{l+1}C_0\max\{\beta^{-2},1\}}\backslash \mathcal{K}_{t,2^{l}C_0\max\{\beta^{-2},1\}}}] \nonumber\\
    &\leq& C\max\{\beta^{-21},1\}\sum_{l=0}^{l_0} 2^{6l}  \exp(-c\max\{\beta^{-1\slash 2},1\}2^{l\slash 2}) \nonumber\\
    &\leq& C\max\{\beta^{-21},1\}\sum_{l=0}^{\infty}\exp(-c \max\{\beta^{-1\slash 2},1\} (l+1))\nonumber\\
    &\leq&   C\max\{  \beta^{-21},1 \}\exp(-c\max\{\beta^{-1\slash 2},1\}).
\end{eqnarray}

If $\beta\geq 2$, repeating the above argument but using Lemma \ref{L3.9} in place of Lemma \ref{L3.7}, we obtain that
\begin{equation}\label{Eq3.97}
    \sum_{t=1}^n\sum_{l=0}^{l_0} \mathbb{E}[(s-Z_{t-1})\mathbbm{1}_{T=t}\mathbbm{1}_{\mathcal{K}_{t,2^{l+1}C_0\max\{\beta^{-2},1\}}\backslash \mathcal{K}_{t,2^{l}C_0\max\{\beta^{-2},1\}}}] \leq  Ce^{-2\beta}. 
\end{equation}

By Lemma \ref{L3.8}, for any $t\in [n]$, we have
\begin{eqnarray*}
  &&  \mathbb{E}[(s-Z_{t-1})\mathbbm{1}_{T=t}\mathbbm{1}_{\mathcal{K}_{t,C_0\max\{\beta^{-2},1\}}}] \nonumber\\
  &\leq& C_0\max\{\beta^{-2},1\}(t-1)(1-C_0^{-1}\min\{\beta^3,1\})^{t-1},
\end{eqnarray*}
hence
\begin{eqnarray}\label{Eq3.93}
  &&  \sum_{\substack{t\in [n]:\\ t\geq\beta^{-4}  }} \mathbb{E}[(s-Z_{t-1})\mathbbm{1}_{T=t}\mathbbm{1}_{\mathcal{K}_{t,C_0\max\{\beta^{-2},1\}}}] \nonumber\\
  &\leq& C_0\max\{\beta^{-2},1\}\sum_{\substack{t\in [n]:\\ t\geq \beta^{-4}}} (t-1)(1-C_0^{-1}\min\{\beta^3,1\})^{t-1}\nonumber\\
  &\leq& C_0\max\{\beta^{-2},1\}\sum_{t=\lceil \max\{\beta^{-4},1\}\rceil-1}^{\infty} t (1-C_0^{-1}\min\{\beta^3,1\})^t\nonumber\\
  &\leq& C\max\{\beta^{-12},1\} \exp(-c\max\{\beta^{-1},1\}).
\end{eqnarray}

If $\beta\geq 2$, by Proposition \ref{P3.5} and Lemma \ref{L3.8}, for any $t\in [n]$, we have
\begin{eqnarray*}
    \mathbb{E}[(s-Z_{t-1})\mathbbm{1}_{T=t}\mathbbm{1}_{\mathcal{K}_{t,C_0\max\{\beta^{-2},1\}}}]&\leq& C_0(t-1)(1-C_0^{-1})^{t-1} \mathbb{P}(N_s\geq 2)\nonumber\\
    &\leq& Ce^{-2\beta}(t-1)(1-C_0^{-1})^{t-1},
\end{eqnarray*}
hence
\begin{eqnarray}\label{Eq3.98}
   && \sum_{t=1}^n \mathbb{E}[(s-Z_{t-1})\mathbbm{1}_{T=t}\mathbbm{1}_{\mathcal{K}_{t,C_0\max\{\beta^{-2},1\}}}] \nonumber\\
   &\leq&  Ce^{-2\beta}\sum_{t=1}^{n }(t-1)(1-C_0^{-1})^{t-1}\leq C e^{-2\beta}. 
\end{eqnarray}

The following lemma bounds $\sum_{j\in \mathcal{H}_l} (l-j)$ for any $l\in [n]$. The proof of this lemma will be given in Section \ref{Sect.3.2.3}.

\begin{lemma}\label{L3.5}
Assume that $\beta\in (0,1]$. There exist positive absolute constants $C,C',c$, such that for any $l\in [n]$, 
\begin{equation}
    \mathbb{P}\Big(\sum_{j\in \mathcal{H}_l}(l-j)>C'\beta^{-2}\Big)\leq C\exp(-c\beta^{-1\slash 16}).
\end{equation}
\end{lemma}

In the following (until (\ref{Eq3.96})), we assume that $\beta\in (0,1]$. Let $C_1\geq 1$ be the maximum of $C_0$ appearing in Proposition \ref{P3.3} and $C'$ appearing in Lemma \ref{L3.5}. Note that $C_1$ is a positive absolute constant. Let $\mathcal{W}$ be the event that $\sum_{j\in \mathcal{H}_l}(l-j)\leq C_1\beta^{-2}$ and $N_l\leq C_1\beta^{-1}$ for any $l\in [s-C_0\beta^{-6},s]\cap [n]$ and $N_l\geq e^{-2}\beta^{-1}\slash 4$ for any $l\in [\max\{s-C_0\beta^{-6},\beta^{-1}\},s]\cap [n]$. Note that $\mathcal{W}\in \mathcal{B}_n$. By Propositions \ref{P3.1} and \ref{P3.3}, Lemma \ref{L3.5}, and the union bound, we have
\begin{equation}\label{Eq3.74}
    \mathbb{P}(\mathcal{W}^c)\leq C\beta^{-6}\exp(-c\beta^{-1\slash 16})\leq C\exp(-c\beta^{-1\slash 16}).
\end{equation}
Note that
\begin{eqnarray}\label{Eq3.75}
 &&   \sum_{\substack{t\in [n]:\\ t  <  \beta^{-4}}} \mathbb{E}[(s-Z_{t-1})\mathbbm{1}_{T=t}\mathbbm{1}_{\mathcal{K}_{t,C_0\max\{\beta^{-2},1\}}}]  \nonumber\\
 &\leq&\sum_{\substack{t\in [n]:\\ t  <  \beta^{-4}}}\sum_{j=1}^{t-1} \mathbb{E}[(Z_{j-1}-Z_j)\mathbbm{1}_{T=t}\mathbbm{1}_{\mathcal{K}_{t,C_0 \beta^{-2}}}] \nonumber\\
 &\leq& \sum_{\substack{j\in [n]:\\ j\leq\beta^{-4}}}\sum_{\substack{t\in[n]:\\t\geq j+1}} \mathbb{E}[(Z_{j-1}-Z_j)\mathbbm{1}_{T=t}\mathbbm{1}_{\mathcal{K}_{j,C_0 \beta^{-2}}}]\nonumber\\
 &\leq&  \sum_{\substack{j\in [n]:\\ j\leq\beta^{-4}}} \mathbb{E}[(Z_{j-1}-Z_j)\mathbbm{1}_{T\geq j+1}\mathbbm{1}_{\mathcal{K}_{j,C_0 \beta^{-2}}}]\nonumber\\
 &\leq& \sum_{\substack{j\in [n]:\\ j\leq\beta^{-4}}} \mathbb{E}[(Z_{j-1}-Z_j)\mathbbm{1}_{W_1=1,\cdots,W_{j}=1}\mathbbm{1}_{\mathcal{K}_{j,C_0 \beta^{-2}}}].
\end{eqnarray}

Below we consider any $j\in [n]$ such that $j\leq \beta^{-4}$. By Lemma \ref{L3.6}, 
\begin{eqnarray*}
   && \mathbb{E}[(Z_{j-1}-Z_j)\mathbbm{1}_{W_j=1}|\mathcal{F}_{j-1}]=\sum_{k=1}^{n}\mathbb{E}[(Z_{j-1}-k)\mathbbm{1}_{Z_j=k,W_j=1}|\mathcal{F}_{j-1}]\nonumber\\
   &=& \sum_{k=1}^n (Z_{j-1}-k)\mathbb{P}(Z_j=k,W_j=1|\mathcal{F}_{j-1})=\frac{\sum_{k\in\mathcal{H}_{Z_{j-1}}}(Z_{j-1}-k)}{N_{Z_{j-1}}}\cdot \mathbbm{1}_{W_{j-1}=1}.
\end{eqnarray*}
Hence
\begin{eqnarray}\label{Eq3.71}
   && \mathbb{E}[(Z_{j-1}-Z_j)\mathbbm{1}_{W_1=1,\cdots,W_{j}=1}\mathbbm{1}_{\mathcal{K}_{j,C_0 \beta^{-2}}}] \nonumber\\
   &=& \mathbb{E}[\mathbb{E}[(Z_{j-1}-Z_j)\mathbbm{1}_{W_j=1}|\mathcal{F}_{j-1}]\mathbbm{1}_{W_1=1,\cdots,W_{j-1}=1}\mathbbm{1}_{\mathcal{K}_{j,C_0 \beta^{-2}}}] \nonumber\\
   &=& \mathbb{E}\bigg[\frac{\sum_{k\in \mathcal{H}_{Z_{j-1}}}(Z_{j-1}-k)}{N_{Z_{j-1}}}\mathbbm{1}_{W_1=1,\cdots,W_{j-1}=1}\mathbbm{1}_{\mathcal{K}_{j,C_0 \beta^{-2}}}\bigg].
\end{eqnarray}

When the event $\{W_1=1,\cdots,W_{j-1}=1\}\cap\mathcal{K}_{j,C_0\beta^{-2}}$ holds, by Lemma \ref{L3.3}, 
\begin{equation}\label{Eq3.70}
    s-Z_{j-1}\leq C_0\beta^{-2}  (j-1)\leq C_0\beta^{-6}, \text{ hence } Z_{j-1}\in [s-C_0\beta^{-6},s]\cap [n].
\end{equation}
By (\ref{Eq3.70}) and the definition of $\mathcal{W}$, when the event
\begin{equation*}
    \{W_1=1,\cdots,W_{j-1}=1\}\cap\mathcal{K}_{j,C_0\beta^{-2}}\cap\mathcal{W}
\end{equation*}
holds, if $Z_{j-1}<\beta^{-1}$, we have
\begin{equation}\label{Eq3.72}
    \frac{\sum_{k\in\mathcal{H}_{Z_{j-1}}}(Z_{j-1}-k)}{N_{Z_{j-1}}}\leq \frac{\beta^{-1}|\mathcal{H}_{Z_{j-1}}|}{N_{Z_{j-1}}}\leq \beta^{-1};
\end{equation}
if $Z_{j-1}\geq \beta^{-1}$, we have $Z_{j-1}\in [\max\{s-C_0\beta^{-6},\beta^{-1}\},s]\cap[n]$ and 
\begin{equation}\label{Eq3.73}
    \frac{\sum_{k\in\mathcal{H}_{Z_{j-1}}}(Z_{j-1}-k)}{N_{Z_{j-1}}}\leq \frac{C_1\beta^{-2}}{e^{-2}\beta^{-1}\slash 4}\leq C\beta^{-1}.
\end{equation}

By (\ref{Eq3.72}), (\ref{Eq3.73}), and Lemma \ref{L3.6}, 
\begin{eqnarray}\label{Eq3.76}
    && \mathbb{E}\bigg[\frac{\sum_{k\in \mathcal{H}_{Z_{j-1}}}(Z_{j-1}-k)}{N_{Z_{j-1}}}\mathbbm{1}_{W_1=1,\cdots,W_{j-1}=1}\mathbbm{1}_{\mathcal{K}_{j,C_0 \beta^{-2}}} \mathbbm{1}_{\mathcal{W}}\bigg]  
 \nonumber\\
 &\leq& C\beta^{-1}\mathbb{E}[\mathbbm{1}_{W_1=1,\cdots,W_{j-1}=1}\mathbbm{1}_{\mathcal{K}_{j,C_0 \beta^{-2}}} \mathbbm{1}_{\mathcal{W}}]\nonumber\\
 &=& C\beta^{-1}\mathbb{E}[\mathbb{P}(W_{j-1}=1|\mathcal{F}_{j-2})\mathbbm{1}_{W_1=1,\cdots,W_{j-2}=1}\mathbbm{1}_{\mathcal{K}_{j,C_0 \beta^{-2}}} \mathbbm{1}_{\mathcal{W}}]\nonumber\\
 &=& C\beta^{-1}\mathbb{E}\Big[\Big(1-\frac{1}{N_{Z_{j-2}}}\Big)\mathbbm{1}_{W_1=1,\cdots,W_{j-2}=1}\mathbbm{1}_{\mathcal{K}_{j,C_0 \beta^{-2}}} \mathbbm{1}_{\mathcal{W}}\Big]\nonumber\\
 &\leq& C\beta^{-1}(1-C_1^{-1}\beta)\mathbb{E}[\mathbbm{1}_{W_1=1,\cdots,W_{j-2}=1}\mathbbm{1}_{\mathcal{K}_{j,C_0 \beta^{-2}}} \mathbbm{1}_{\mathcal{W}}]\nonumber\\
 &\leq& \cdots \leq C\beta^{-1}(1-C_1^{-1}\beta)^{j-1},
\end{eqnarray}
where we note that $Z_{j-2}\in [s-C_0\beta^{-6},s]\cap [n]$ when the event $\{W_1=1,\cdots,W_{j-2}=1\}\cap\mathcal{K}_{j,C_0\beta^{-2}}$ holds (using Lemma \ref{L3.3} and arguing as in (\ref{Eq3.70})) and use the definition of $\mathcal{W}$ in the fifth line. When the event $\{W_1=1,\cdots,W_{j-1}=1\}\cap\mathcal{K}_{j,C_0\beta^{-2}}$ holds, by (\ref{Eq3.70}), $Z_{j-1}\in [s-C_0\beta^{-2}j,s]\cap [n]$; by the definition of $\mathcal{K}_{j,C_0\beta^{-2}}$, we have $\max_{k\in \mathcal{H}_{Z_{j-1}}}\{Z_{j-1}-k\}\leq C_0\beta^{-2}$. Hence by Lemma \ref{L3.2} and (\ref{Eq3.74}),
\begin{eqnarray}\label{Eq3.77}
    && \mathbb{E}\bigg[\frac{\sum_{k\in \mathcal{H}_{Z_{j-1}}}(Z_{j-1}-k)}{N_{Z_{j-1}}}\mathbbm{1}_{W_1=1,\cdots,W_{j-1}=1}\mathbbm{1}_{\mathcal{K}_{j,C_0 \beta^{-2}}} \mathbbm{1}_{\mathcal{W}^c}\bigg]  
 \nonumber\\
 &\leq&  \mathbb{E}\Big[\frac{C_0\beta^{-2}|\mathcal{H}_{Z_{j-1}}|}{N_{Z_{j-1}}}\mathbbm{1}_{W_1=1,\cdots,W_{j-1}=1}\mathbbm{1}_{\mathcal{K}_{j,C_0 \beta^{-2}}} \mathbbm{1}_{\mathcal{W}^c}\Big]  \nonumber\\
 &\leq& C_0\beta^{-2}\mathbb{P}(\mathcal{W}^c) \leq C\beta^{-2}\exp(-c\beta^{-1\slash 16})\leq C\exp(-c\beta^{-1\slash 16}).
\end{eqnarray}
By (\ref{Eq3.75}), (\ref{Eq3.71}), (\ref{Eq3.76}), and (\ref{Eq3.77}), \begin{eqnarray}\label{Eq3.96}
    && \sum_{\substack{t\in [n]:\\ t  <  \beta^{-4}}} \mathbb{E}[(s-Z_{t-1})\mathbbm{1}_{T=t}\mathbbm{1}_{\mathcal{K}_{t,C_0\max\{\beta^{-2},1\}}}] \nonumber\\
    &\leq& C\sum_{\substack{j\in [n]:\\ j \leq  \beta^{-4}}}(\beta^{-1}(1-C_1^{-1}\beta)^{j-1}+\exp(-c\beta^{-1\slash 16}))\nonumber\\
    &\leq& C\beta^{-1}\sum_{j=1}^{\infty} (1-C_1^{-1}\beta)^{j-1}+C\beta^{-4}\exp(-c\beta^{-1\slash 16})  \leq C\beta^{-2}.
\end{eqnarray} 

By (\ref{Eq3.95}), (\ref{Eq3.94}), (\ref{Eq3.93}), and (\ref{Eq3.96}), we have
\begin{equation}\label{Eq3.99}
    \mathbb{E}[s-\min(\mathcal{C}_s(\sigma))]\leq C\max\{\beta^{-2},1\}.
\end{equation}
If $\beta\geq 2$, by (\ref{Eq3.95}), (\ref{Eq3.97}), (\ref{Eq3.98}), we have 
\begin{equation}\label{Eq3.100}
    \mathbb{E}[s-\min(\mathcal{C}_s(\sigma))]\leq Ce^{-2\beta}.   
\end{equation}
Combining (\ref{Eq3.99}) and (\ref{Eq3.100}), noting that $s-\min(\mathcal{C}_s(\sigma))\leq s-1$, we have
\begin{equation}\label{Eq3.101}
    \mathbb{E}[s-\min(\mathcal{C}_s(\sigma))]\leq C\min\{e^{-2\beta}\max\{\beta^{-2},1\}, s-1\}.
\end{equation}

Let $\bar{\sigma}\in S_n$ be such that $\bar{\sigma}(j)=n+1-\sigma(n+1-j)$ for every $j\in [n]$. Note that $\min(\mathcal{C}_{n+1-s}(\bar{\sigma}))=n+1-\max(\mathcal{C}_s(\sigma))$. As the distribution of $\bar{\sigma}$ is given by $\mathbb{P}_{n,\beta}$, by (\ref{Eq3.101}), we have
\begin{eqnarray}\label{Eq3.102}
  &&   \mathbb{E}[\max(\mathcal{C}_s(\sigma))-s]= \mathbb{E}[(n+1-s)-\min(\mathcal{C}_{n+1-s}(\bar{\sigma}))] \nonumber\\
  &=&  \mathbb{E}[(n+1-s)-\min(\mathcal{C}_{n+1-s}(\sigma))]\leq  C\min\{e^{-2\beta}\max\{\beta^{-2},1\}, n-s\}.  \nonumber\\
  &&
\end{eqnarray}
By (\ref{Eq3.101}) and (\ref{Eq3.102}),
\begin{equation}
    \mathbb{E}[\max(\mathcal{C}_s(\sigma))-\min(\mathcal{C}_s(\sigma))]\leq C\min\{e^{-2\beta}\max\{\beta^{-2},1\}, n-1\}. 
\end{equation}

\subsubsection{Proof of Lemma \ref{L3.5}}\label{Sect.3.2.3}

In this part, we give the proof of Lemma \ref{L3.5}.

\begin{proof}[Proof of Lemma \ref{L3.5}]

Let $C_0$ be the constant appearing in Proposition \ref{P2.1}. Let $C_1,C_2$ be the constants appearing in Proposition \ref{P2.3} (with $C_0=1$). We also denote by $C_0'$ the constant $C_0$ appearing in Proposition \ref{P3.4}. Without loss of generality, we assume that $C_0,C_1 \geq 1$. We assume that $\beta\in (0,1]$ and $l\in [n]$. 

By the definition of $\mathcal{H}_l$, noting Lemma \ref{L3.1}, we have
\begin{eqnarray}\label{Eq3.13}
  &&  \sum_{j\in \mathcal{H}_l}(l-j)=\sum_{j\in [n]: j\leq l-1} (l-j) \mathbbm{1}_{b_j\geq l}\nonumber\\
  &=& \sum_{\substack{j\in [l-1]:\\ \sigma_0(j)<j}} (l-j)\mathbbm{1}_{b_j\geq l}+\sum_{\substack{j\in [l-1]:\\ \sigma_0(j)\geq j}} (l-\sigma_0(j))\mathbbm{1}_{b_j\geq l}\nonumber\\ 
    &&   +\sum_{\substack{j\in [l-1]:\\ \sigma_0(j)\geq l}} (\sigma_0(j)-j)+\sum_{\substack{j\in [l-1]:\\ j\leq \sigma_0(j)\leq l-1}} (\sigma_0(j)-j)\mathbbm{1}_{b_j\geq l}.
\end{eqnarray}
We bound the four terms on the right-hand side of (\ref{Eq3.13}) in \textbf{Parts 1-4} below.

\paragraph{Part 1}

Note that for any $j\in [l-1]$ such that $\sigma_0(j)<j$, by Lemma \ref{L3.1}, we have $\mathbb{P}(b_j\geq l|\sigma_0)=e^{-2\beta (l-j)}$.

Let $\mathcal{E}$ be the event that for any $j\in [n]$ such that $j\leq l-\beta^{-3\slash 2}$ and $\sigma_0(j)<j$, we have $b_j< l$. By the union bound and noting that $\beta\in (0,1]$, we obtain that
\begin{eqnarray}\label{Eq3.15}
    \mathbb{P}(\mathcal{E}^c|\sigma_0) &\leq& \sum_{\substack{j\in [n]:j\leq l-\beta^{-3\slash 2},\\\sigma_0(j)<j}}\mathbb{P}(b_j\geq l|\sigma_0)=\sum_{\substack{j\in [n]:j\leq l-\beta^{-3\slash 2},\\\sigma_0(j)<j}} e^{-2\beta(l-j)} \nonumber\\
  &\leq& \sum_{j=\lceil \beta^{-3\slash 2} \rceil}^{\infty} e^{-2\beta j} = \frac{e^{-2\beta \lceil \beta^{-3\slash 2}\rceil}}{1-e^{-2\beta}} \leq 4\beta^{-1} e^{-2\beta^{-1\slash 2}}. 
\end{eqnarray}
Hence
\begin{equation}\label{Eq3.17}
    \mathbb{P}(\mathcal{E}^c)\leq 4\beta^{-1} e^{-2\beta^{-1\slash 2}}.
\end{equation}

Note that for any $j\in [n]$ such that $l-\beta^{-3\slash 2}<j\leq l-1$ and $\sigma_0(j)<j$, we have $1\leq l-j \leq \beta^{-3\slash 2}$. As $\beta\in (0,1]$, we have
\begin{eqnarray}\label{Eq3.18}
    &&  \sum_{\substack{j\in [n]:l-\beta^{-3\slash 2}<j\leq l-1,\\ \sigma_0(j)<j}} \mathbb{E}[(l-j) \mathbbm{1}_{b_j\geq l}|\sigma_0]=\sum_{\substack{j\in [n]:l-\beta^{-3\slash 2}<j\leq l-1,\\ \sigma_0(j)<j}}(l-j) e^{-2\beta(l-j)} \nonumber\\
    &&\leq\sum_{j=1}^{\infty} j e^{-2\beta j} = \frac{e^{-2\beta}}{(1-e^{-2\beta})^2}\leq 16\beta^{-2},
\end{eqnarray}
\begin{eqnarray}\label{Eq3.20}
  &&  \sum_{\substack{j\in [n]:l-\beta^{-3\slash 2}<j\leq l-1,\\ \sigma_0(j)<j}} \mathbb{E}[(l-j)^2 \mathbbm{1}_{b_j\geq l}|\sigma_0]=\sum_{\substack{j\in [n]:l-\beta^{-3\slash 2}<j\leq l-1,\\ \sigma_0(j)<j}}(l-j)^2 e^{-2\beta(l-j)} \nonumber\\
  &&\leq\sum_{j=1}^{\infty} j^2 e^{-2\beta j} 
  =  \frac{e^{-2\beta}(1+e^{-2\beta})}{(1-e^{-2\beta})^3}\leq 200\beta^{-3}.  
\end{eqnarray}
Hence by Bernstein's inequality (see e.g. \cite[Corollary 2.11]{BLM} and the remarks below it), for any $t\geq 0$, 
\begin{eqnarray}\label{Eq3.14}
   &&  \mathbb{P}\Big(\sum_{\substack{j\in [n]:l-\beta^{-3\slash 2}<j\leq l-1,\\\sigma_0(j)<j}}(l-j)\mathbbm{1}_{b_j\geq l}\geq 16\beta^{-2}+t \Big|\sigma_0\Big) \nonumber\\
   &\leq& \exp\Big(-\frac{t^2}{2(200\beta^{-3}+\beta^{-3\slash 2}t\slash 3)}\Big).
\end{eqnarray}
Let $\mathcal{E}'$ be the event that $\sum\limits_{\substack{j\in [n]:l-\beta^{-3\slash 2}<j\leq l-1,\\\sigma_0(j)<j}}(l-j)\mathbbm{1}_{b_j\geq  l}\leq 20\beta^{-2}$. Taking $t=4\beta^{-2}$ in (\ref{Eq3.14}), we obtain that $\mathbb{P}((\mathcal{E}')^c |\sigma_0)\leq \exp(-c\beta^{-1\slash 2})$, hence
\begin{equation}\label{Eq3.16}
    \mathbb{P}((\mathcal{E}')^c)\leq \exp(-c\beta^{-1\slash 2}).
\end{equation}

When the events $\mathcal{E}$ and $\mathcal{E}'$ hold, 
\begin{equation*}
    \sum_{\substack{j\in [l-1]:\\ \sigma_0(j)<j}} (l-j)\mathbbm{1}_{b_j\geq l} =\sum_{\substack{j\in [n]:l-\beta^{-3\slash 2}< j\leq l-1,\\\sigma_0(j)<j}} (l-j)\mathbbm{1}_{b_j\geq l}\leq 20\beta^{-2}.
\end{equation*}
Hence by (\ref{Eq3.17}), (\ref{Eq3.16}), and the union bound, we have
\begin{equation}\label{Eq3.53}
    \mathbb{P}\Big(\sum_{\substack{j\in [l-1]:\\ \sigma_0(j)<j}} (l-j)\mathbbm{1}_{b_j\geq l} >
    20\beta^{-2}\Big)\leq \mathbb{P}(\mathcal{E}^c)+\mathbb{P}((\mathcal{E}')^c)\leq C\exp(-c\beta^{-1\slash 2}). 
\end{equation}

\paragraph{Part 2}

Note that for any $j\in [l-1]$ such that $j\leq \sigma_0(j)\leq l$, by Lemma \ref{L3.1}, we have $\mathbb{P}(b_j\geq l|\sigma_0)=e^{-2\beta (l-\sigma_0(j))}$. 

Let $\tilde{\mathcal{E}}$ be the event that for any $j\in [l-1]$ such that $j\leq \sigma_0(j)\leq l-\beta^{-3\slash 2}$, we have $b_j<l$. Arguing similarly as in (\ref{Eq3.15}), we obtain that
\begin{eqnarray*}
    \mathbb{P}(\tilde{\mathcal{E}}^c|\sigma_0)&\leq& \sum_{\substack{j\in [l-1]:\\ j\leq \sigma_0(j)\leq l-\beta^{-3\slash 2}}} \mathbb{P}(b_j\geq l|\sigma_0)= \sum_{\substack{j\in [l-1]:\\ j\leq \sigma_0(j)\leq l-\beta^{-3\slash 2}}} e^{-2\beta(l-\sigma_0(j))} \nonumber\\
    &\leq& \sum_{j=\lceil \beta^{-3\slash 2}\rceil}^{\infty} e^{-2\beta j} \leq  4\beta^{-1} e^{-2\beta^{-1\slash 2}}.
\end{eqnarray*}
Hence
\begin{equation}\label{Eq3.21}
    \mathbb{P}(\tilde{\mathcal{E}}^c)   \leq  4\beta^{-1} e^{-2\beta^{-1\slash 2}}.
\end{equation}

Note that for any $j\in [l-1]$ such that $\max\{j,l-\beta^{-3\slash 2}\}\leq \sigma_0(j) \leq l$, we have $0\leq l-\sigma_0(j)\leq \beta^{-3\slash 2}$ and 
\begin{equation*}
    \mathbb{E}[(l-\sigma_0(j))^2\mathbbm{1}_{b_j\geq l}|\sigma_0]=(l-\sigma_0(j))^2\mathbb{P}(b_j\geq l|\sigma_0)=(l-\sigma_0(j))^2 e^{-2\beta (l-\sigma_0(j))}. 
\end{equation*}
Arguing similarly as in (\ref{Eq3.18}) and (\ref{Eq3.20}), we obtain that
\begin{eqnarray*}
   && \sum_{\substack{j\in [l-1]:\\ \max\{j,l-\beta^{-3\slash 2}\}\leq \sigma_0(j) \leq l}} \mathbb{E}[(l-\sigma_0(j))\mathbbm{1}_{b_j\geq l}|\sigma_0]  \nonumber\\
   &=&  \sum_{\substack{j\in [l-1]:\\ \max\{j,l-\beta^{-3\slash 2}\}\leq \sigma_0(j) \leq l}} (l-\sigma_0(j)) e^{-2\beta(l-\sigma_0(j))} \leq \sum_{j=1}^{\infty} je^{-2\beta j}\leq 16\beta^{-2},
\end{eqnarray*}
\begin{eqnarray*}
    && \sum_{\substack{j\in [l-1]:\\ \max\{j,l-\beta^{-3\slash 2}\}\leq \sigma_0(j) \leq l}} \mathbb{E}[(l-\sigma_0(j))^2\mathbbm{1}_{b_j\geq l}|\sigma_0]  \nonumber\\
   &=&  \sum_{\substack{j\in [l-1]:\\ \max\{j,l-\beta^{-3\slash 2}\}\leq \sigma_0(j) \leq l}} (l-\sigma_0(j))^2 e^{-2\beta(l-\sigma_0(j))} \leq \sum_{j=1}^{\infty} j^2 e^{-2\beta j}\leq 200\beta^{-3}.
\end{eqnarray*}
Hence by Bernstein's inequality, for any $t\geq 0$,
\begin{eqnarray}\label{Eq3.19}
 &&   \mathbb{P}\Big(\sum_{\substack{j\in [l-1]:\\ \max\{j,l-\beta^{-3\slash 2}\}\leq \sigma_0(j) \leq l}} (l-\sigma_0(j))\mathbbm{1}_{b_j\geq l} \geq 16\beta^{-2}+t \Big|\sigma_0\Big) \nonumber\\
 &\leq&  \exp\Big(-\frac{t^2}{2(200\beta^{-3}+\beta^{-3\slash 2}t\slash 3)}\Big).
\end{eqnarray}
Let $\tilde{\mathcal{E}}'$ be the event that $\sum\limits_{\substack{j\in [l-1]:\\ \max\{j,l-\beta^{-3\slash 2}\}\leq \sigma_0(j) \leq l}} (l-\sigma_0(j))\mathbbm{1}_{b_j\geq l}\leq 20\beta^{-2}$. Taking $t=4\beta^{-2}$ in (\ref{Eq3.19}), we obtain that $\mathbb{P}((\tilde{\mathcal{E}}')^c|\sigma_0)\leq \exp(-c\beta^{-1\slash 2})$, hence
\begin{equation}\label{Eq3.22}
    \mathbb{P}((\tilde{\mathcal{E}}')^c)\leq \exp(-c\beta^{-1\slash 2}).
\end{equation}

When the events $\tilde{\mathcal{E}}$ and $\tilde{\mathcal{E}}'$ hold, we have 
\begin{eqnarray}
   && \sum_{\substack{j\in [l-1]:\\ \sigma_0(j)\geq j}} (l-\sigma_0(j))\mathbbm{1}_{b_j\geq l}\leq \sum_{\substack{j\in [l-1]:\\ j \leq \sigma_0(j)\leq l}} (l-\sigma_0(j))\mathbbm{1}_{b_j\geq l}\nonumber\\
    &=& \sum\limits_{\substack{j\in [l-1]:\\ \max\{j,l-\beta^{-3\slash 2}\}\leq \sigma_0(j) \leq l}} (l-\sigma_0(j))\mathbbm{1}_{b_j\geq l}\leq 20\beta^{-2}.
\end{eqnarray}
Hence by (\ref{Eq3.21}), (\ref{Eq3.22}), and the union bound, we have
\begin{equation}\label{Eq3.54}
    \mathbb{P}\Big(\sum_{\substack{j\in [l-1]:\\ \sigma_0(j)\geq j}} (l-\sigma_0(j))\mathbbm{1}_{b_j\geq l}>20\beta^{-2}\Big)\leq \mathbb{P}(\tilde{\mathcal{E}}^c)+\mathbb{P}((\tilde{\mathcal{E}}')^c)\leq C\exp(-c\beta^{-1\slash 2}). 
\end{equation}

\paragraph{Part 3}

For each $k\in \mathbb{N}$, we let 
\begin{equation*}
 I_k:=\{j\in [n]: j\leq l-k\beta^{-1}, \sigma_0(j)\geq l\}, 
\end{equation*}
\begin{equation*}
J_k:=\{j\in [n]: j\leq l-1, \sigma_0(j)\geq l+k\beta^{-1}\}.
\end{equation*}

For any $j\in [n]$ such that $k\beta^{-1}\leq l-j<(k+1)\beta^{-1}$ (where $k\in \mathbb{N}$) and $\sigma_0(j)\geq l$, we have $j\in I_0\cap\cdots\cap I_k$. Hence
\begin{equation}\label{Eq3.25}
    \sum_{\substack{j\in [l-1]:\\\sigma_0(j)\geq l, l-j\leq \beta^{-9\slash 8}}} (l-j) \leq \beta^{-1} \sum_{k=0}^{\lfloor \beta^{-1\slash 8}\rfloor} |I_k|.
\end{equation}
For any $j\in [n]$ such that $k\beta^{-1}\leq \sigma_0(j)-l<(k+1)\beta^{-1}$ (where $k\in \mathbb{N}$) and $j\leq l-1$, we have $j\in J_0\cap \cdots\cap J_k$. Hence
\begin{equation}\label{Eq3.24}
    \sum_{\substack{j\in [l-1]:\\\sigma_0(j)\geq l, \sigma_0(j)-l\leq \beta^{-9\slash 8}}} (\sigma_0(j)-l) \leq \beta^{-1} \sum_{k=0}^{\lfloor \beta^{-1\slash 8}\rfloor} |J_k|.
\end{equation}
By (\ref{Eq3.25}) and (\ref{Eq3.24}), we have
\begin{eqnarray}\label{Eq3.34}
   && \sum_{\substack{j\in [l-1]:\\\sigma_0(j)\geq l}} (\sigma_0(j)-j)\leq \beta^{-1}\sum_{k=0}^{\lfloor \beta^{-1\slash 8} \rfloor} |I_k|+\beta^{-1}\sum_{k=0}^{\lfloor \beta^{-1\slash 8} \rfloor} |J_k|\nonumber\\
   &&\quad\quad +  \sum_{\substack{j\in [l-1]:\\\sigma_0(j)\geq l, l-j>\beta^{-9\slash 8}}}(\sigma_0(j)-j)+\sum_{\substack{j\in [l-1]:\\\sigma_0(j)>l+\beta^{-9\slash 8}}} (\sigma_0(j) -j ).
\end{eqnarray}

Recall Definition \ref{Def2.2}. For any $k\in \mathbb{N}$, $I_k,J_k\subseteq \{\sigma_0^{-1}(l)\}\cup \mathcal{D}_l(\sigma_0)$, hence $\max\{|I_k|,|J_k|\}\leq |\mathcal{D}_l(\sigma_0)|+1$. Let $\mathcal{T}_1$ be the event that for any $k\in \mathbb{N}$, $\max\{|I_k|,|J_k|\}\leq 2C_0\beta^{-1}$. By Proposition \ref{P2.1}, we have
\begin{equation}\label{Eq3.35}
    \mathbb{P}(\mathcal{T}_1^c) \leq \mathbb{P}(|\mathcal{D}_l(\sigma_0)|\geq 2C_0\beta^{-1}-1)\leq \mathbb{P}(|\mathcal{D}_l(\sigma_0)|\geq C_0\beta^{-1}) \leq C\exp(-c\beta^{-1}). 
\end{equation}

Now consider any $k\in \mathbb{N}$ such that $4C_1\leq k\leq \lfloor \beta^{-1\slash 8}\rfloor$. If $l-\lfloor k\beta^{-1}\slash 2\rfloor\leq 0$, then $|I_k|=0$. Below we assume that $l-\lfloor k\beta^{-1}\slash 2\rfloor \geq 1$. For any $j\in I_k$, we have
\begin{equation*}
    j\leq l-k\beta^{-1}\leq (l-\lfloor k\beta^{-1}\slash 2\rfloor)-k\beta^{-1}\slash 4,
\end{equation*}
\begin{equation*}
    \sigma_0(j)\geq l\geq (l-\lfloor k\beta^{-1}\slash 2\rfloor)+k\beta^{-1}\slash 4,
\end{equation*}
and $(j,\sigma_0(j))\in \mathcal{R}'_{l-\lfloor k\beta^{-1}\slash 2\rfloor, k\beta^{-1}\slash 4}(\sigma_0)$. Hence $|I_k|\leq |\mathcal{R}'_{l-\lfloor k\beta^{-1}\slash 2\rfloor, k\beta^{-1}\slash 4}(\sigma_0)|$. Note that there exists a positive absolute constant $C_3$, such that for any $t\in \mathbb{N}^{*}$, $C_3 t^{-4}\slash C_2 \geq e^{-C_1^{-1}t\slash 4}$. Hence by Proposition \ref{P2.3} (taking $C_0=1$, $\Delta=k\beta^{-1}\slash 4$, and $r=C_3 k^{-4}\slash C_2$), 
\begin{eqnarray}\label{Eq3.28}
  &&  \mathbb{P}(|I_k|\geq C_3 k^{-4} \beta^{-1})\leq  \mathbb{P}(|\mathcal{R}'_{l-\lfloor k\beta^{-1}\slash 2\rfloor, k\beta^{-1}\slash 4}(\sigma_0)|\geq C_3 k^{-4} \beta^{-1}) \nonumber\\
  &\leq& C\max\{\beta k^4,k\beta^{-1}\} \exp(-c k^{-4} \beta^{-1})\leq C\exp(-c\beta^{-1\slash 2}). 
\end{eqnarray}
Similarly, we can deduce that
\begin{equation}\label{Eq3.29}
    \mathbb{P}(|J_k|\geq C_3 k^{-4}\beta^{-1})\leq C\exp(-c\beta^{-1\slash 2}). 
\end{equation}
Let $\mathcal{T}_2$ be the event that $\max\{|I_k|,|J_k|\}\leq C_3 k^{-4}\beta^{-1}$ for any $k\in\mathbb{N}$ such that $4C_1\leq k\leq \lfloor \beta^{-1\slash 8}\rfloor$. By (\ref{Eq3.28}), (\ref{Eq3.29}), and the union bound,
\begin{equation}\label{Eq3.36}
    \mathbb{P}(\mathcal{T}_2^c)\leq 2(\lfloor \beta^{-1\slash 8}\rfloor+1)\cdot C\exp(-c\beta^{-1\slash 2})\leq C\exp(-c\beta^{-1\slash 2}). 
\end{equation}

For any $j\in [l-1]$ such that $l-j>\beta^{-9\slash 8}$, by Proposition \ref{P3.4} (where we take $u=l-j$ and $v=C_0'\sqrt{l-j}\beta^{-1\slash 2}$), 
\begin{eqnarray}\label{Eq3.30}
  &&  \mathbb{P}(\sigma_0(j)\geq l)\leq \mathbb{P}(|\sigma_0(j)-j|\geq l-j) \nonumber\\
  &\leq&  C(l-j)\exp(-c\sqrt{l-j}\beta^{-1\slash 2})+C\exp(-c\sqrt{l-j}\beta^{1\slash 2}). 
\end{eqnarray}
Let $\mathcal{T}_3$ be the event that $\sigma_0(j)<l$ for any $j\in [l-1]$ such that $l-j>\beta^{-9\slash 8}$. By (\ref{Eq3.30}) and the union bound,
\begin{eqnarray}\label{Eq3.32}
  &&  \mathbb{P}(\mathcal{T}_3^c)\leq \sum_{\substack{j\in [l-1]:\\l-j>\beta^{-9\slash 8}}} \mathbb{P}(\sigma_0(j)\geq l)\nonumber\\
  &\leq& C \sum_{\substack{j\in [l-1]:\\l-j>\beta^{-9\slash 8}}} ((l-j)\exp(-c\sqrt{l-j}\beta^{-1\slash 2})+\exp(-c\sqrt{l-j}\beta^{1\slash 2})) \nonumber\\
  &\leq& C\sum_{j=\lceil \beta^{-9\slash 8}\rceil}^{\infty} j \exp(-c j^{1\slash 2}\beta^{1\slash 2})\leq C\beta^{-1} \sum_{j=\lceil \beta^{-9\slash 8}\rceil}^{\infty} \exp(-cj^{1\slash 2}\beta^{1\slash 2})\nonumber\\
 &\leq& C\beta^{-1}\Big(\exp(-c\beta^{-1\slash 16})\cdot (\beta^{-1}+1)\nonumber\\
 &&\quad\quad\quad +\sum_{k=\lceil\beta^{-1\slash 8}\rceil}^{\infty}\sum_{j\in[k\beta^{-1},(k+1)\beta^{-1})\cap\mathbb{N}} \exp(-c j^{1\slash 2}\beta^{1\slash 2})\Big)\nonumber\\
 &\leq& C\beta^{-2}\Big(\exp(-c\beta^{-1\slash 16})+\sum_{k=\lceil\beta^{-1\slash 8}\rceil}^{\infty}\exp(-c \sqrt{k})\Big)\leq C\exp(-c\beta^{-1\slash 16}).\nonumber\\
 &&
\end{eqnarray}

For any $j\in [l-1]$, by Proposition \ref{P3.4} (where we take $u=l-j+\beta^{-9\slash 8}$ and $v=C_0'\sqrt{u}\beta^{-1\slash 2}$), 
\begin{eqnarray}\label{Eq3.31}
  &&  \mathbb{P}(\sigma_0(j)>l+\beta^{-9\slash 8})\leq \mathbb{P}(|\sigma_0(j)-j|\geq l-j+\beta^{-9\slash 8}) \nonumber\\
  &\leq& C(l-j+\beta^{-9\slash 8})\exp(-c\sqrt{l-j+\beta^{-9\slash 8}}\beta^{-1\slash 2})\nonumber\\
  && +C\exp(-c\sqrt{l-j+\beta^{-9\slash 8}}\beta^{1\slash 2}).
\end{eqnarray}
Let $\mathcal{T}_4$ be the event that $\sigma_0(j)\leq l+\beta^{-9\slash 8}$ for any $j\in [l-1]$. By (\ref{Eq3.31}) and the union bound, following the argument in (\ref{Eq3.32}), we obtain that
\begin{eqnarray}\label{Eq3.33}
   && \mathbb{P}(\mathcal{T}_4^c)\leq \sum_{j\in[l-1]} \mathbb{P}(\sigma_0(j)>l+\beta^{-9\slash 8})\nonumber\\
    &\leq& C\sum_{j\in [l-1]}(l-j+\beta^{-9\slash 8})\exp(-c\sqrt{l-j+\beta^{-9\slash 8}}\beta^{-1\slash 2})\nonumber\\
    &&+C\sum_{j\in [l-1]} \exp(-c\sqrt{l-j+\beta^{-9\slash 8}}\beta^{1\slash 2})\nonumber\\
    &\leq&  C\sum_{j=1}^{\infty} (j+\beta^{-9\slash 8})\exp(-c\sqrt{j+\beta^{-9\slash 8}}\beta^{1\slash 2})\nonumber\\
    &\leq& C\sum_{j=\lceil \beta^{-9\slash 8}\rceil}^{\infty} j\exp(-cj^{1\slash 2}\beta^{1\slash 2})\leq C\exp(-c\beta^{-1\slash 16}).
\end{eqnarray}

By (\ref{Eq3.34}), when the event $\mathcal{T}_1\cap\mathcal{T}_2\cap\mathcal{T}_3\cap\mathcal{T}_4$ holds,
\begin{eqnarray*}
  &&  \sum_{\substack{j\in [l-1]:\\\sigma_0(j)\geq l}} (\sigma_0(j)-j) \leq \beta^{-1}\sum_{\substack{k\in \mathbb{N}:\\k\leq 4C_1}}(|I_k|+|J_k|)+\beta^{-1}\sum_{\substack{k\in \mathbb{N}:\\4C_1\leq k\leq \lfloor \beta^{-1\slash 8} \rfloor}} (|I_k|+|J_k|) \nonumber\\
    && \leq  \beta^{-1} \cdot (4C_1+1)\cdot 4C_0\beta^{-1}+\beta^{-1}\cdot 2C_3\beta^{-1}\sum_{\substack{k\in \mathbb{N}:\\4C_1\leq k\leq \lfloor \beta^{-1\slash 8} \rfloor}}k^{-4} \leq C_4\beta^{-2},
\end{eqnarray*}
where $C_4$ is a positive absolute constant. Hence by (\ref{Eq3.35}), (\ref{Eq3.36}), (\ref{Eq3.32}), (\ref{Eq3.33}), and the union bound, 
\begin{eqnarray}\label{Eq3.55}
    \mathbb{P}\Big(\sum_{\substack{j\in [l-1]:\\\sigma_0(j)\geq l}} (\sigma_0(j)-j)>C_4\beta^{-2}\Big)\leq \sum_{j=1}^4\mathbb{P}(\mathcal{T}_j^c)\leq C\exp(-c\beta^{-1\slash 16}).
\end{eqnarray}

\paragraph{Part 4}

For any $k,k'\in \mathbb{N}$ such that $k\leq k'$, let
\begin{equation*}
    R_{k,k'}:=\{j\in [l-1]: j\leq \sigma_0(j)\leq l-1, j\leq l-k'\beta^{-1}, l-(k+1)\beta^{-1}< \sigma_0(j) \leq l-k\beta^{-1}\}.
\end{equation*}
Note that 
\begin{equation}\label{Eq3.38}
    |R_{k,k'}|\leq \beta^{-1}+1\leq 2\beta^{-1}. 
\end{equation}

Consider any $k,k'\in \mathbb{N}$ such that $4C_1\leq k'-k\leq 2\beta^{-1\slash 8}$. For any $j\in R_{k,k'}$,
\begin{equation*}
    j\leq l-k'\beta^{-1}\leq \Big\lfloor l-\frac{k+k'}{2}\beta^{-1}\Big\rfloor-\frac{k'-k-2}{2}\beta^{-1},
\end{equation*}
\begin{equation*}
    \sigma_0(j)\geq l-(k+1)\beta^{-1}\geq \Big\lfloor l-\frac{k+k'}{2}\beta^{-1}  \Big\rfloor+\frac{k'-k-2}{2}\beta^{-1}.
\end{equation*}
Note that if $l-(k+k')\beta^{-1}\slash 2<1$, we have $|R_{k,k'}|=0$. Below we assume that $l-(k+k')\beta^{-1}\slash 2\geq 1$. For any $j\in R_{k,k'}$, 
\begin{equation*}
    (j,\sigma_0(j))\in\mathcal{R}'_{\lfloor l-(k+k')\beta^{-1}\slash 2 \rfloor,(k'-k-2)\beta^{-1}\slash 2}(\sigma_0).
\end{equation*}
Hence $|R_{k,k'}|\leq |\mathcal{R}'_{\lfloor l-(k+k')\beta^{-1}\slash 2 \rfloor,(k'-k-2)\beta^{-1}\slash 2}(\sigma_0)|$. Note that there exists a positive absolute constant $C_5$, such that $C_5\slash (C_2 t^4)\geq e^{-(t-2)\slash (2C_1)}$ for any $t\in \mathbb{N}^{*}$. By Proposition \ref{P2.3} (taking $C_0=1$, $\Delta=(k'-k-2)\beta^{-1}\slash 2$, and $r=C_5(k'-k)^{-4}\slash C_2$), 
\begin{eqnarray}\label{Eq3.40}
  &&  \mathbb{P}(|R_{k,k'}|\geq C_5(k'-k)^{-4}\beta^{-1})\nonumber\\
  &\leq& C\max\{(k'-k)^4\beta,(k'-k)\beta^{-1}\}\exp(-c(k'-k)^{-4}\beta^{-1})\nonumber\\
  &\leq& C\exp(-c\beta^{-1\slash 2}).
\end{eqnarray}
By (\ref{Eq3.38}) and (\ref{Eq3.40}),
\begin{equation}\label{Eq3.41}
    \mathbb{E}[|R_{k,k'}|]\leq C_5(k'-k)^{-4}\beta^{-1}+2\beta^{-1}\cdot C\exp(-c\beta^{-1\slash 2})\leq C(k'-k)^{-4}\beta^{-1}.
\end{equation}

Consider any $k,k'\in \mathbb{N}$ such that $k'-k>2\beta^{-1\slash 8}$. Note that for any $j\in R_{k,k'}$, $\sigma_0(j)-j\geq (k'-k-1)\beta^{-1}$. Hence
\begin{equation}\label{Eq3.37}
    |R_{k,k'}|\leq |\{j\in [n]:l-(k+1)\beta^{-1}<j\leq l-k\beta^{-1},j-\sigma_0^{-1}(j)\geq (k'-k-1)\beta^{-1}\}|.
\end{equation}
By (\ref{Eq3.37}), Proposition \ref{P3.4} (with $u=(k'-k-1)\beta^{-1}$ and $v=C_0'\sqrt{k'-k-1}\beta^{-1}$), and the union bound, noting that the distribution of $\sigma_0^{-1}$ is $\mathbb{P}_{n,\beta}$, we obtain that
\begin{eqnarray}\label{Eq3.39}
   \mathbb{P}(|R_{k,k'}|\geq 1)&\leq& \sum_{\substack{j\in [n]:\\ l-(k+1)\beta^{-1}<j\leq l-k\beta^{-1}}} \mathbb{P}(|\sigma_0^{-1}(j)-j|\geq (k'-k-1)\beta^{-1}) \nonumber\\
   &\leq& C(k'-k-1)\exp(-c\sqrt{k'-k-1}).
\end{eqnarray}
By (\ref{Eq3.38}) and (\ref{Eq3.39}),
\begin{equation}\label{Eq3.42}
    \mathbb{E}[|R_{k,k'}|]\leq 2\beta^{-1}\mathbb{P}(|R_{k,k'}|\geq 1)\leq C(k'-k-1)\exp(-c\sqrt{k'-k-1}).
\end{equation}

Now consider any $j\in [l-1]$ such that $j\leq \sigma_0(j)\leq l-1$. Note that there exist $k_1,k_2\in \mathbb{N}$ such that $k_1\geq k_2$ and 
\begin{equation*}
     l-(k_1+1)\beta^{-1}<j\leq l-k_1\beta^{-1}, \quad l-(k_2+1)\beta^{-1}<\sigma_0(j)\leq l-k_2\beta^{-1}.
\end{equation*}
It can be checked that $\sigma_0(j)-j\leq (k_1-k_2+1)\beta^{-1}$ and $j\in R_{k_2,k'}$ for any $k'\in [k_2,k_1]\cap\mathbb{N}$. Hence 
\begin{equation}\label{Eq3.46}
    \sum_{\substack{j\in [l-1]:\\ j\leq \sigma_0(j)\leq l-1}} (\sigma_0(j)-j)\mathbbm{1}_{b_j\geq l}\leq \beta^{-1} \sum_{\substack{k,k'\in\mathbb{N}:\\ k\leq k' }} \sum_{j\in R_{k,k'} } \mathbbm{1}_{b_j\geq l}. 
\end{equation}

For any $k,k'\in\mathbb{N}$ such that $k\leq k'$, let $\mathcal{G}_{k,k'}$ be the event that $b_j<l$ for any $j\in R_{k,k'}$, and let $\mathcal{L}_{k,k'}$ be the event that 
\begin{equation*}
    \sum_{j\in R_{k,k'}}\mathbbm{1}_{b_j\geq l} \leq e^{-2k}|R_{k,k'}|+\beta^{-1\slash 4} \sqrt{|R_{k,k'}|}. 
\end{equation*}
For any $j\in R_{k,k'}$, by Lemma \ref{L3.1},
\begin{equation*}\label{Eq3.26}
\mathbb{P}(b_j\geq l |\sigma_0)= e^{-2\beta(l-\sigma_0(j))}  \leq e^{-2k}. 
\end{equation*}
By the union bound, 
\begin{equation*}
    \mathbb{P}((\mathcal{G}_{k,k'})^c|\sigma_0)\leq \sum_{j\in R_{k,k'}}\mathbb{P}(b_j\geq l|\sigma_0)\leq e^{-2k}|R_{k,k'}|. 
\end{equation*}
Hence
\begin{equation}\label{Eq3.43}
    \mathbb{P}((\mathcal{G}_{k,k'})^c)\leq e^{-2k}\mathbb{E}[|R_{k,k'}|]. 
\end{equation}
By Hoeffding's inequality, for any $t\geq 0$, 
\begin{equation*}
    \mathbb{P}\Big(\sum_{j\in R_{k,k'}} \mathbbm{1}_{b_j\geq l} \geq (e^{-2k}+t)|R_{k,k'}|
\Big|\sigma_0\Big) \leq \exp(-2|R_{k,k'}|t^2).
\end{equation*}
Hence for any $t\geq 0$,
\begin{equation}\label{Eq3.27}
    \mathbb{P}\Big(\sum_{j\in R_{k,k'}}\mathbbm{1}_{b_j\geq l} > e^{-2k}|R_{k,k'}|+t\sqrt{|R_{k,k'}|}\Big|\sigma_0\Big) \leq \exp(-2t^2).
\end{equation}
Taking $t=\beta^{-1\slash 4}$ in (\ref{Eq3.27}), we obtain that
\begin{equation*}
    \mathbb{P}((\mathcal{L}_{k,k'})^c|\sigma_0) \leq \exp(-2\beta^{-1\slash 2}). 
\end{equation*}
Hence
\begin{equation}\label{Eq3.45}
    \mathbb{P}((\mathcal{L}_{k,k'})^c)\leq \exp(-2\beta^{-1\slash 2}).
\end{equation}
For any $k\in \mathbb{N}$, let $\mathcal{G}_k:=\bigcap\limits_{k'=k}^{\infty}  \mathcal{G}_{k,k'}$. By (\ref{Eq3.38}), (\ref{Eq3.41}), (\ref{Eq3.42}), (\ref{Eq3.43}), and the union bound, 
\begin{eqnarray}\label{Eq3.44}
  &&  \mathbb{P}(\mathcal{G}_k^c)\leq e^{-2k}\sum_{k'=k}^{\infty} \mathbb{E}[|R_{k,k'}|] \nonumber\\
  &\leq& e^{-2k} \Big(2\beta^{-1}(4C_1+1)+C\sum_{\substack{k'\in \mathbb{N}:\\4C_1\leq k'-k\leq 2\beta^{-1\slash 8}}} (k'-k)^{-4}\beta^{-1}\nonumber\\
  && \quad\quad +C\sum_{\substack{k'\in \mathbb{N}:\\k'-k>2\beta^{-1\slash 8}}}(k'-k-1)\exp(-c\sqrt{k'-k-1})\Big)\nonumber\\
  &\leq& C e^{-2k} \beta^{-1}. 
\end{eqnarray}
Let $\mathcal{W}_1:=\bigcap\limits_{\substack{k\in\mathbb{N}:\\k>\beta^{-1\slash 8}}} \mathcal{G}_k$. By (\ref{Eq3.44}) and the union bound, 
\begin{equation}\label{Eq3.50}
    \mathbb{P}(\mathcal{W}_1^c)\leq C\beta^{-1}\sum_{k=\lceil \beta^{-1\slash 8}\rceil}^{\infty} e^{-2k}\leq C\exp(-c\beta^{-1\slash 8}). 
\end{equation}

Let $\mathcal{W}_2$ be the event that $|R_{k,k'}|=0$ for any $k,k'\in \mathbb{N}$ such that $k\leq \beta^{-1\slash 8}$ and $k'-k>2\beta^{-1\slash 8}$. By (\ref{Eq3.39}) and the union bound,
\begin{eqnarray}
   && \mathbb{P}(\mathcal{W}_2^c)\leq \sum_{\substack{k,k'\in \mathbb{N}:\\k\leq \beta^{-1\slash 8}, k'-k>2\beta^{-1\slash 8}}} \mathbb{P}(|R_{k,k'}|\geq 1) \nonumber\\
   &\leq& C\sum_{k=0}^{\lfloor \beta^{-1\slash 8}\rfloor}\sum_{\substack{k'\in\mathbb{N}:\\k'-k>2\beta^{-1\slash 8}}} (k'-k-1)\exp(-c\sqrt{k'-k-1}) \nonumber\\
   &\leq& C\beta^{-1\slash 8} \sum_{t=\lceil \beta^{-1\slash 8} \rceil}^{\infty} t\exp(-c\sqrt{t})\leq C\exp(-c\beta^{-1\slash 16}).
\end{eqnarray}

Let $\mathcal{W}_3$ be the event that $|R_{k,k'}|\leq C_5(k'-k)^{-4}\beta^{-1}$ for any $k,k'\in \mathbb{N}$ such that $k\leq \beta^{-1\slash 8}$ and $4C_1\leq k'-k\leq 2\beta^{-1\slash 8}$, and let 
\begin{equation*}
     \mathcal{W}_4:=\bigcap\limits_{\substack{k,k'\in \mathbb{N}:\\k\leq \beta^{-1\slash 8},0\leq k'-k\leq 2\beta^{-1\slash 8}}} \mathcal{L}_{k,k'}.
\end{equation*}
By (\ref{Eq3.40}) and the union bound,
\begin{equation}
    \mathbb{P}(\mathcal{W}_3^c)\leq (1+\beta^{-1\slash 8})(1+2\beta^{-1\slash 8})\cdot C\exp(-c\beta^{-1\slash 2}) \leq C\exp(-c\beta^{-1\slash 2}). 
\end{equation}
By (\ref{Eq3.45}) and the union bound,
\begin{equation}\label{Eq3.51}
     \mathbb{P}(\mathcal{W}_4^c)\leq (1+\beta^{-1\slash 8})(1+2\beta^{-1\slash 8})\cdot \exp(-2\beta^{-1\slash 2}) \leq C\exp(-c\beta^{-1\slash 2}). 
\end{equation}

Assume that the event $\mathcal{W}_1\cap\mathcal{W}_2\cap\mathcal{W}_3\cap\mathcal{W}_4$ holds. By (\ref{Eq3.46}) and the definitions of $\mathcal{W}_1,\mathcal{W}_2,\mathcal{W}_4$,
\begin{eqnarray}\label{Eq3.47}
  &&  \sum_{\substack{j\in [l-1]:\\ j\leq \sigma_0(j)\leq l-1}} (\sigma_0(j)-j)\mathbbm{1}_{b_j\geq l }\leq \beta^{-1}\sum_{\substack{k,k'\in \mathbb{N}:\\k\leq \beta^{-1\slash 8},0\leq k'-k\leq 2\beta^{-1\slash 8}}}\sum_{j\in R_{k,k'}} \mathbbm{1}_{b_j\geq l} \nonumber\\
  &\leq& \beta^{-1}\sum_{\substack{k,k'\in \mathbb{N}:\\k\leq \beta^{-1\slash 8},0\leq k'-k\leq 2\beta^{-1\slash 8}}}  (e^{-2k}|R_{k,k'}|+\beta^{-1\slash 4} \sqrt{|R_{k,k'}|}).
\end{eqnarray}
By (\ref{Eq3.38}), 
\begin{eqnarray}\label{Eq3.48}
   && \sum_{\substack{k,k'\in \mathbb{N}:\\k\leq \beta^{-1\slash 8},0\leq k'-k<  4C_1}}(e^{-2k}|R_{k,k'}|+\beta^{-1\slash 4} \sqrt{|R_{k,k'}|})\nonumber\\
   &\leq& \sum_{k=0}^{\lfloor \beta^{-1\slash 8} \rfloor} (4C_1+1)(2e^{-2k}\beta^{-1}+\sqrt{2}\beta^{-3\slash 4})\leq C\beta^{-1}.
\end{eqnarray}
By the definition of $\mathcal{W}_3$,
\begin{eqnarray}\label{Eq3.49}
     && \sum_{\substack{k,k'\in \mathbb{N}:\\k\leq \beta^{-1\slash 8},4C_1\leq k'-k\leq 2\beta^{-1\slash 8}}}(e^{-2k}|R_{k,k'}|+\beta^{-1\slash 4} \sqrt{|R_{k,k'}|})\nonumber\\
     &\leq& C \sum_{k=0}^{\lfloor \beta^{-1\slash 8} \rfloor} \sum_{\substack{k'\in \mathbb{N}:\\4C_1\leq k'-k\leq 2\beta^{-1\slash 8}}} (e^{-2k}(k'-k)^{-4}\beta^{-1}+(k'-k)^{-2}\beta^{-3\slash 4})\nonumber\\
     &\leq& C\sum_{k=0}^{\lfloor \beta^{-1\slash 8} \rfloor} (e^{-2k}\beta^{-1}+\beta^{-3\slash 4})\leq C\beta^{-1}. 
\end{eqnarray}
By (\ref{Eq3.47})-(\ref{Eq3.49}), 
\begin{equation}
    \sum_{\substack{j\in [l-1]:\\ j\leq \sigma_0(j)\leq l-1}} (\sigma_0(j)-j)\mathbbm{1}_{b_j\geq l }\leq C_6 \beta^{-2},
\end{equation}
where $C_6$ is a positive absolute constant. Hence
\begin{equation}\label{Eq3.52}
    \mathcal{W}_1\cap\mathcal{W}_2\cap\mathcal{W}_3\cap\mathcal{W}_4 \subseteq \Big\{\sum_{\substack{j\in [l-1]:\\ j\leq \sigma_0(j)\leq l-1}} (\sigma_0(j)-j)\mathbbm{1}_{b_j\geq l }\leq C_6 \beta^{-2}\Big\}. 
\end{equation}
By (\ref{Eq3.50})-(\ref{Eq3.51}), (\ref{Eq3.52}), and the union bound,
\begin{equation}\label{Eq3.56}
    \mathbb{P}\Big(\sum_{\substack{j\in [l-1]:\\ j\leq \sigma_0(j)\leq l-1}} (\sigma_0(j)-j)\mathbbm{1}_{b_j\geq l }> C_6 \beta^{-2}\Big)\leq C\exp(-c\beta^{-1\slash 16}). 
\end{equation}

\bigskip

Combining (\ref{Eq3.13}), (\ref{Eq3.53}), (\ref{Eq3.54}), (\ref{Eq3.55}), and (\ref{Eq3.56}), we conclude that
\begin{equation}
    \mathbb{P}\Big(\sum_{j\in \mathcal{H}_l}(l-j)>C'\beta^{-2}\Big)\leq C\exp(-c\beta^{-1\slash 16}),
\end{equation}
where $C'$ is a positive absolute constant.

\end{proof}

\subsection{Proofs of Theorem \ref{Thm1.1} and the lower bound in Theorem \ref{Thm1.3.1}}\label{Sect.3.3}

In this subsection, we give the proofs of Theorem \ref{Thm1.1} and the lower bound in Theorem \ref{Thm1.3.1}. Assuming two auxiliary lemmas (Lemmas \ref{L3.11} and \ref{L3.12}), we present the proof of Theorem \ref{Thm1.1} in Section \ref{Sect.3.3.1}. The proofs of Lemmas \ref{L3.11} and \ref{L3.12} are given in Section \ref{Sect.3.3.2}. Finally, the proof of the lower bound in Theorem \ref{Thm1.3.1} is presented in Section \ref{Sect.3.3.3}.

Throughout this subsection, we assume the notations in Section \ref{Sect.3.2.1}.

\subsubsection{Proof of Theorem \ref{Thm1.1}}\label{Sect.3.3.1}

We consider any $s\in [n]$. Note that $|\mathcal{C}_s(\sigma)|\leq \max(\mathcal{C}_s(\sigma))-\min(\mathcal{C}_s(\sigma))+1$. Hence by the upper bound in Theorem \ref{Thm1.3.1} (which we have established in Section \ref{Sect.3.2}), 
\begin{equation}\label{Eq3.5.13}
    \mathbb{E}[|\mathcal{C}_s(\sigma)|]\leq C\min\{e^{-2\beta}\max\{\beta^{-2},1\}, n-1\}+1\leq C\min\{\max\{\beta^{-2},1\},n\}.
\end{equation}

The rest of this proof is devoted to the proof of the lower bound on $\mathbb{E}[|\mathcal{C}_s(\sigma)|]$. We start with the following lemma.

\begin{lemma}\label{L3.10}
Assume that $\beta\in (0,1]$. There exist positive absolute constants $C,c$, such that for any $l\in [n]$ and $u\geq \beta^{-1}$, we have 
\begin{equation*}
    \mathbb{P}(\max_{j\in \mathcal{H}_l}\{l-j\}\geq u) \leq Cu\beta^{-1}\exp(-c\sqrt{\beta u}).     
\end{equation*}
\end{lemma}

\begin{proof}

Note that if $\max_{j\in \mathcal{H}_l}\{l-j\}\geq u$, there exists some $j\in [n]$ such that $j\leq l-u$ and $b_j\geq l$. Hence by the union bound and Lemma \ref{L3.1},
\begin{eqnarray}\label{Eq3.103}
   &&   \mathbb{P}(\max_{j\in \mathcal{H}_l}\{l-j\}\geq u )\leq \sum_{j\in [n]:j\leq l- u }\mathbb{P}(b_j\geq l)\nonumber\\
    &=& \sum_{j\in [n]:j\leq l-u }\mathbb{E}[e^{-2\beta(l-\max\{j,\sigma_0(j)\})_{+}}].
\end{eqnarray}
Now note that 
\begin{eqnarray}\label{Eq3.104}
  &&  \sum_{j\in [n]:j\leq l-u} e^{-2\beta(l-\max\{j,\sigma_0(j)\})_{+}} 
  \leq \sum_{\substack{j\in [n]:j\leq l-u,\\\sigma_0(j)\leq j }} e^{-2\beta(l-j)}\nonumber\\
  &&\quad\quad\quad\quad\quad\quad\quad\quad+\sum_{\substack{j\in [n]:j\leq l-u, \\ j<\sigma_0(j)\leq l-u\slash 2}}   e^{-2\beta(l-\sigma_0(j))}
  +\sum_{\substack{j\in [n]:j\leq l-u,\\\sigma_0(j)> l-u\slash 2}} 1 \nonumber\\
  &\leq & 2\sum_{ j\in [n]: j\leq l-u\slash 2 } e^{-2\beta (l-j)}+\sum_{j\in [n]:j\leq l-u} \mathbbm{1}_{\sigma_0(j)>l-u\slash 2}.
\end{eqnarray}
By (\ref{Eq3.103}) and (\ref{Eq3.104}),
\begin{eqnarray}\label{Eq3.105}
   &&  \mathbb{P}(\max_{j\in \mathcal{H}_l}\{l-j\}\geq u) \nonumber\\
   &\leq& 2\sum_{j\in [n]:j\leq l-u\slash 2} e^{-2\beta (l-j)}+\sum_{j\in [n]:  j\leq l-  u} \mathbb{P}(\sigma_0(j)>l-u\slash 2).
\end{eqnarray}

Let $C_0$ be the constant in Proposition \ref{P3.3}. For any $j\in [n]$ such that $j\leq l-u$, by Proposition \ref{P3.4} (replacing $u$ by $l-j-u\slash 2\geq u\slash 2\geq \beta^{-1}\slash 2$ and taking $v=2C_0\sqrt{l-j-u\slash 2}\beta^{-1\slash 2}\geq C_0\beta^{-1}$), we have
\begin{eqnarray*}
  &&  \mathbb{P}(\sigma_0(j)>l-u\slash 2) \leq \mathbb{P}(|\sigma_0(j)-j|\geq l-j-u\slash 2) \nonumber\\
  &\leq& C(l-j-u\slash 2)\exp(-c\sqrt{l-j-u\slash 2}\beta^{-1\slash 2})+ C\exp(-c\sqrt{l-j-u\slash 2}\beta^{1\slash 2}) \nonumber\\
  &\leq& C(l-j-u\slash 2)\exp(-c\sqrt{l-j-u\slash 2}\beta^{1\slash 2}).
\end{eqnarray*}
Hence
\begin{eqnarray*}
   &&   \sum_{j\in [n]:  j\leq l-  u} \mathbb{P}(\sigma_0(j)>l-u\slash 2)  \nonumber\\
   &\leq& C\sum_{j\in [n]:  j\leq l-  u} (l-j-u\slash 2)\exp(-c\sqrt{l-j-u\slash 2}\beta^{1\slash 2})\nonumber\\
   &\leq& C\sum_{j=0}^{\infty} (j+u\slash 2)\exp(-c\beta^{1\slash 2}\sqrt{j+ u \slash 2})\nonumber\\
   &\leq&  C\exp(-c\sqrt{\beta u}) \sum_{j=0}^{\infty}(j+u)\exp(-c\sqrt{\beta j}).
\end{eqnarray*}
Note that
\begin{eqnarray*}
   && \sum_{j=0}^{\infty}(j+u)\exp(-c\sqrt{\beta j})\leq \sum_{k=0}^{\infty} \sum_{\substack{j\in \mathbb{N}:\\ k\beta^{-1}\leq j\leq(k+1)\beta^{-1}}} (j+u)\exp(-c\sqrt{\beta j}) \nonumber\\
   &\leq& \sum_{k=0}^{\infty} ((k+1)\beta^{-1}+u)(\beta^{-1}+1)\exp(-c\sqrt{k})\nonumber\\
   &\leq& C\beta^{-2}\sum_{k=0}^{\infty} (k+1)\exp(-c\sqrt{k})+C u \beta^{-1}\sum_{k=0}^{\infty} \exp(-c\sqrt{k})\leq Cu\beta^{-1}.
\end{eqnarray*}
Hence
\begin{equation}\label{Eq3.106}
     \sum_{j\in [n]:  j\leq l-  u} \mathbb{P}(\sigma_0(j)>l-u\slash 2) \leq Cu\beta^{-1}\exp(-c\sqrt{\beta u}).
\end{equation}
We also note that
\begin{equation}\label{Eq3.107}
    \sum_{j\in [n]:j\leq l-u\slash 2} e^{-2\beta (l-j)} \leq  \sum_{j=\lceil u\slash 2 \rceil}^{\infty} e^{-2\beta j}=\frac{e^{-2\beta   \lceil u \slash 2\rceil}}{1-e^{-2\beta  }} \leq 5\beta^{-1}e^{-\beta u }.
\end{equation}
By (\ref{Eq3.105}), (\ref{Eq3.106}), and (\ref{Eq3.107}), we conclude that
\begin{equation}
    \mathbb{P}(\max_{j\in \mathcal{H}_l}\{l-j\}\geq u)\leq Cu\beta^{-1}\exp(-c\sqrt{\beta u}). 
\end{equation}

\end{proof}

If $n\leq 100$ or $\beta\geq 1\slash 10000$, for any $s\in [n]$, we have
\begin{equation}\label{Eq3.5.14}
    \mathbb{E}[|\mathcal{C}_s(\sigma)|]\geq 1 \geq c\min\{\max\{\beta^{-2},1\},n\}.
\end{equation}
In the following, we assume that $n>100$ and $\beta\in (0,1\slash 10000)$, and consider any $s\in [n]$ such that $s\geq  (n+1)\slash 2$. We consider two cases: $n<\beta^{-2}$ or $n\geq \beta^{-2}$. 

\paragraph{Case 1: $n<\beta^{-2}$}

Let $\mathcal{P}$ be the event that $b_j\geq\beta^{-7\slash 16}$ for any $j\in [n]$. By Lemma \ref{L3.1}, $\mathbb{P}(\mathcal{P}|\sigma_0)=\prod\limits_{j\in [n]:j<\beta^{-7\slash 16}}\mathbb{P}(b_j\geq \beta^{-7\slash 16}|\sigma_0)\geq \exp(-2\beta^{1\slash 8})$, hence
\begin{equation}\label{Eq3.3.14}
    \mathbb{P}(\mathcal{P})=\mathbb{E}[\mathbb{P}(\mathcal{P}|\sigma_0)]\geq \exp(-2\beta^{1\slash 8}).
\end{equation}
Let $\mathcal{P}'$ be the event that $N_j\geq e^{-2}j\slash 2$ for any $j\in [n]$ with $\lfloor\beta^{-7\slash 16}\rfloor \leq j\leq \beta^{-1}$, $N_j\geq   e^{-2}\beta^{-1}\slash 4$ for any $j\in [n]$ with $j>\beta^{-1}$, and $\max_{k\in \mathcal{H}_j}\{j-k\}\leq \beta^{-9\slash 8}$ for any $j\in [n]$. By Propositions \ref{P3.1}-\ref{P3.1.v}, Lemma \ref{L3.10}, and the union bound, noting that $n<\beta^{-2}$, we have 
\begin{equation}\label{Eq3.3.9}
    \mathbb{P}((\mathcal{P}')^c)\leq n\exp(-c\lfloor\beta^{-7\slash 16}\rfloor)+C n \beta^{-17\slash 8} \exp(-c\beta^{-1\slash 16})\leq C\exp(-c\beta^{-1\slash 16}).
\end{equation}
We also note that $\mathcal{P},\mathcal{P}'\in \mathcal{B}_n$. 

For any $j\in [n]$ and $j'\in [n+1]$, we denote by $\mathcal{S}_{j,j'}$ the event that $j$ is contained in an open arc at step $j'$. Note that $\mathcal{S}_{j,j'}\in \mathcal{B}_{j'-1}$. We also let $d:=\min\{ \lfloor\beta^{-7\slash 16}\rfloor, n\}$. We have the following lemma, whose proof will be given in Section \ref{Sect.3.3.2}.

\begin{lemma}\label{L3.11}
Assume that $n>100$ and $\beta\in (0,1\slash 10000)$. There exist positive absolute constants $C_0,c_0$, such that for any $j\in [n]$,
\begin{eqnarray*}
    \mathbb{P}(\mathcal{S}_{j,d+1}|\mathcal{B}_n)\mathbbm{1}_{\mathcal{P}'} &\geq& (\exp(-C_0\beta^{1\slash 32})-\exp(-c_0\beta^{-1\slash 32}))_{+}^3\nonumber\\
  &&  \times (\exp(-C_0(\beta^{1\slash 2}+\beta^2j))-\exp(-c_0\beta^{-1\slash 4}))_{+}\mathbbm{1}_{\mathcal{P}'}.
\end{eqnarray*}
\end{lemma}

By Lemma \ref{L3.11} and the assumption that $n<\beta^{-2}$, there exist positive absolute constants $c_1,c_2$ with $c_1<1\slash 10000$, such that when $\beta\in (0,c_1)$, we have $\mathbb{P}(\mathcal{S}_{j,d+1}|\mathcal{B}_n)\mathbbm{1}_{\mathcal{P}'}\geq 2 c_2 \mathbbm{1}_{\mathcal{P}'}$ for any $j\in [n]$ and $\mathbb{P}(\mathcal{S}_{j,d+1}|\mathcal{B}_n)\mathbbm{1}_{\mathcal{P}'}\geq (1-c_2)\mathbbm{1}_{\mathcal{P}'}$ for any $j\in [n]$ satisfying $j\leq c_1\beta^{-2}$.

Let $\mathcal{J}:=\{j\in [n]: j\leq c_1\beta^{-2} \}$. When $\beta\in (0,c_1)$, for any $j\in \mathcal{J}$,
\begin{equation}\label{Eq3.3.10}
    \mathbb{P}(\mathcal{S}_{j,d+1}\cap\mathcal{S}_{s,d+1}|\mathcal{B}_n) \mathbbm{1}_{\mathcal{P}'}\geq (\mathbb{P}(\mathcal{S}_{j,d+1}|\mathcal{B}_n)+\mathbb{P}(\mathcal{S}_{s,d+1}|\mathcal{B}_n)-1) \mathbbm{1}_{\mathcal{P}'}\geq c_2  \mathbbm{1}_{\mathcal{P}'}.
\end{equation}
We also note that when $\beta\in (0,c_1)$, $c_1\beta^{-2}\geq c_1^{-1}\geq 10000$, hence
\begin{equation}\label{Eq3.3.11}
    |\mathcal{J}|\geq \min\{n,\lfloor c_1\beta^{-2}\rfloor\}\geq  \min\{n,c_1\beta^{-2}-1\}\geq \frac{1}{2}\min\{n,c_1\beta^{-2}\}.
\end{equation}

Recall that $\mathcal{A}_O(d+1)$ is the set of open arcs at step $d+1$ and $\mathcal{A}_C(d+1)$ is the set of closed arcs at step $d+1$. Note that $|\mathcal{A}_O(d+1)|=d$. We denote by $\mathbf{a}_1,\mathbf{a}_2,\cdots,\mathbf{a}_d$ 
the elements of $\mathcal{A}_O(d+1)$, where $\mathbf{a}_1$ has the smallest tail, $\mathbf{a}_2$ has the second smallest tail, and so on. For each $j\in [d]$, we denote by $H_j$ and $T_j$ the head and the tail of $\mathbf{a}_j$, respectively. Note that for any $j\in [d]$, $T_j\in [d]$ and $Y_{T_j}\in\{H_1,\cdots,H_d\}$. Hence there exists $\tau\in S_d$, such that $Y_{T_j}=H_{\tau(j)}$ for any $j\in [d]$. 

Note that when the event $\mathcal{P}$ holds, $b_j\geq \beta^{-7\slash 16}\geq d$ for any $j\in [n]$. Hence conditional on $\mathcal{B}_d$, assuming that the event $\mathcal{P}$ holds, $\tau$ is a uniformly random permutation of $[d]$. The set of points in any cycle of $\sigma$ is either the set of points in a closed arc from $\mathcal{A}_C(d+1)$ or the set of points in a collection of open arcs $\mathbf{a}_{j_1},\cdots,\mathbf{a}_{j_k}$ (where $k\in \mathbb{N}^{*}$, $j_1,\cdots,j_k\in [d]$, $j_1<\cdots<j_k$, and $\{j_1,\cdots,j_k\}$ is the set of points in a cycle of $\tau$).

For any $j\in [n]$, we denote by $\mathcal{E}_j$ the event that $j$ and $s$ are contained in the same arc at step $d+1$. For any $j\in [n]$, $\mathcal{E}_j\in \mathcal{B}_d$ and $\mathcal{E}_j \subseteq \{j\in \mathcal{C}_s(\sigma)\}$, hence
\begin{eqnarray}\label{Eq3.3.7}
    \mathbb{P}(j\in \mathcal{C}_s(\sigma)|\mathcal{B}_d)\mathbbm{1}_{\mathcal{E}_j \cap \mathcal{S}_{j,d+1}\cap\mathcal{S}_{s,d+1}}&\geq& \mathbb{P}(\mathcal{E}_j|\mathcal{B}_d) \mathbbm{1}_{\mathcal{E}_j \cap \mathcal{S}_{j,d+1}\cap\mathcal{S}_{s,d+1}}\nonumber\\
    &=& \mathbbm{1}_{\mathcal{E}_j \cap \mathcal{S}_{j,d+1}\cap\mathcal{S}_{s,d+1}}.
\end{eqnarray}
Note that $d\geq 2$. For any $j,j'\in [d]$ with $j\neq j'$, the probability that $j,j'$ belong to the same cycle of a uniformly random permutation of $[d]$ is given by $1\slash 2$. For any $j\in [n]$, if $j$ is contained in an open arc at step $d+1$, we let $K_j\in [d]$ be such that $j\in \mathbf{a}_{K_j}$; otherwise we let $K_j=0$. For any $j\in [n]$, when the event $\mathcal{E}_j^c\cap \mathcal{S}_{j,d+1}\cap\mathcal{S}_{s,d+1}$ holds, $j$ and $s$ are contained in two different open arcs at step $d+1$, hence $j\in \mathcal{C}_s(\sigma)$ if and only if $K_j$ and $K_s$ belong to the same cycle of $\tau$. Therefore, for any $j\in [n]$, 
\begin{equation}\label{Eq3.3.8}
    \mathbb{P}(j\in \mathcal{C}_s(\sigma)|\mathcal{B}_d)\mathbbm{1}_{\mathcal{E}_j^c\cap \mathcal{S}_{j,d+1}\cap\mathcal{S}_{s,d+1}}\mathbbm{1}_{\mathcal{P}}= \frac{1}{2} \mathbbm{1}_{\mathcal{E}_j^c\cap \mathcal{S}_{j,d+1}\cap\mathcal{S}_{s,d+1}}\mathbbm{1}_{\mathcal{P}}.
\end{equation}
By (\ref{Eq3.3.7}) and (\ref{Eq3.3.8}), for any $j\in [n]$, we have
\begin{equation}\label{Eq3.3.12}
    \mathbb{P}(j\in \mathcal{C}_s(\sigma)|\mathcal{B}_d)\mathbbm{1}_{\mathcal{S}_{j,d+1}\cap\mathcal{S}_{s,d+1}}\mathbbm{1}_{\mathcal{P}}\geq \frac{1}{2} \mathbbm{1}_{\mathcal{S}_{j,d+1}\cap\mathcal{S}_{s,d+1}}\mathbbm{1}_{\mathcal{P}}.
\end{equation}

By (\ref{Eq3.3.9}), there exists a positive absolute constant $c_3\in (0,c_1)$, such that when $\beta\in (0,c_3)$, $\mathbb{P}((\mathcal{P}')^c)\leq e^{-2}\slash 2$. Hence by (\ref{Eq3.3.14}), when $\beta\in  (0,c_3)$,
\begin{equation}\label{Eq3.3.15}
    \mathbb{P}(\mathcal{P}\cap\mathcal{P}')\geq \mathbb{P}(\mathcal{P})-\mathbb{P}((\mathcal{P}')^c)\geq \frac{e^{-2}}{2}.
\end{equation}
By (\ref{Eq3.3.10}), (\ref{Eq3.3.12}), and (\ref{Eq3.3.15}), when $\beta\in (0,c_3)$, for any $j\in \mathcal{J}$, we have 
\begin{eqnarray}\label{Eq3.3.16}
  && \mathbb{P}(j\in \mathcal{C}_s(\sigma))\geq \mathbb{P}(\{j\in \mathcal{C}_s(\sigma)\}\cap \mathcal{S}_{j,d+1}\cap\mathcal{S}_{s,d+1}\cap\mathcal{P}\cap\mathcal{P}')\nonumber\\
  &=&\mathbb{E}[\mathbb{P}(j\in\mathcal{C}_s(\sigma)|\mathcal{B}_d)\mathbbm{1}_{\mathcal{S}_{j,d+1}\cap\mathcal{S}_{s,d+1}}\mathbbm{1}_{\mathcal{P}}\mathbbm{1}_{\mathcal{P}'}]\geq \frac{1}{2}\mathbb{E}[\mathbbm{1}_{\mathcal{S}_{j,d+1}\cap\mathcal{S}_{s,d+1}}\mathbbm{1}_{\mathcal{P}} \mathbbm{1}_{\mathcal{P}'}] \nonumber\\
  &=& \frac{1}{2}\mathbb{E}[\mathbb{P}(\mathcal{S}_{j,d+1}\cap\mathcal{S}_{s,d+1}|\mathcal{B}_n)\mathbbm{1}_{\mathcal{P}'}\mathbbm{1}_{\mathcal{P}}]\geq \frac{c_2}{2}\mathbb{P}(\mathcal{P}\cap\mathcal{P}')\geq \frac{e^{-2}c_2}{4}.
\end{eqnarray}
By (\ref{Eq3.3.11}) and (\ref{Eq3.3.16}), when $\beta\in (0,c_3)$, we have
\begin{eqnarray}\label{Eq3.3.17}
    \mathbb{E}[|\mathcal{C}_s(\sigma)|]\geq \sum_{j\in \mathcal{J}}\mathbb{P}(j\in \mathcal{C}_s(\sigma))\geq \frac{e^{-2}c_2}{4} |\mathcal{J}| \geq c\min\{\beta^{-2},n\}.
\end{eqnarray}
When $\beta\geq c_3$, we have
\begin{equation}\label{Eq3.3.18}
    \mathbb{E}[|\mathcal{C}_s(\sigma)|]\geq 1\geq c_3^2\beta^{-2}\geq c\min\{\beta^{-2},n\}.
\end{equation}
By (\ref{Eq3.3.17}) and (\ref{Eq3.3.18}), noting that $\beta\in (0,1\slash 10000)$, we conclude that
\begin{equation}\label{Eq3.5.15}
    \mathbb{E}[|\mathcal{C}_s(\sigma)|]\geq c\min\{\beta^{-2},n\}  = c\min\{\max\{\beta^{-2},1\},n\}.
\end{equation}

\paragraph{Case 2: $n\geq \beta^{-2}$} 

We fix two constants $c_4,c_5\in (0,1\slash 10)$ whose values will be determined later. Let
\begin{equation*}
    s':=\min\{j\in [n]: j\geq s- c_4 \beta^{-2}\}, \quad s'':=s'-\lceil \beta^{-2}\slash 10 \rceil.
\end{equation*}
Note that $s\geq s'>s''$ and 
\begin{equation}\label{Eq3.4.1}
    s''\geq s-(c_4+1\slash 10)\beta^{-2}-1 \geq s-n\slash 4\geq n\slash 4 \geq 25, \text{ hence } s''\in [n].
\end{equation}

Let $\tilde{Z}_0=s'$. If $s$ is contained in an open arc at step $s'+1$, we let $\tilde{W}_0=1$; if $s$ is contained in a closed arc at step $s'+1$, we let $\tilde{W}_0=0$. In the following, we define $\tilde{W}_t\in \{0,1\}$ and $\tilde{Z}_t\in [n]$ for each $t\in [n]$ inductively. For any $t\in [n]$, suppose that $\tilde{W}_{t-1}\in \{0,1\}$ and $\tilde{Z}_{t-1}\in [n]$ have been defined. If $\tilde{W}_{t-1}=0$, we let $\tilde{W}_t=0$ and $\tilde{Z}_t=\tilde{Z}_{t-1}$. Below we consider the case where $\tilde{W}_{t-1}=1$. If $s$ is contained in a closed arc at step $\tilde{Z}_{t-1}$, we let $\tilde{W}_t=0$ and $\tilde{Z}_t=\tilde{Z}_{t-1}$; if $s$ is contained in an open arc at step $\tilde{Z}_{t-1}$, we take $\tilde{W}_t=1$ and let $\tilde{Z}_t$ be the tail of this open arc (note that $\tilde{Z}_t<\tilde{Z}_{t-1}$ in this case).

For any $t\in [n]$, we define
\begin{equation*}
    U_t:=\{j\in [n]: s'\leq j\leq s-1, j\text{ is contained in an open arc at step }\tilde{Z}_t+1 \},
\end{equation*}
\begin{equation*}
    V_t:=\{j\in [n]: s'\leq j\leq s-1, j \text{ is contained in the arc that contains } s \text{ at step } \tilde{Z}_t+1\}.
\end{equation*}
We also define
\begin{equation*}
    U:=\{j\in [n]: s'\leq j\leq s-1, j\text{ is contained in an open arc at step } s''\},
\end{equation*}
\begin{equation*}
    V:=\{j\in [n]: s'\leq j\leq s-1, j \text{ is contained in the arc that contains } s \text{ at step } s''\}.
\end{equation*}
Note that for any $t\in [n]$, if $\tilde{Z}_t\geq s''$, then $U_t\supseteq U$ and $V_t\subseteq V$, hence 
\begin{equation}\label{Eq3.108}
    |U_t\backslash V_t|\geq |U\backslash V|.
\end{equation}

For any $t\in [n]$, we define $L_t$ to be the length of the arc containing $s$ at step $\tilde{Z}_{t}+1$ minus the length of the arc containing $s$ at step $\tilde{Z}_{t-1}+1$. By (\ref{Eq3.108}), recalling Definition \ref{Defi3.1}, we obtain that for any $t\in [n]\backslash \{1\}$,
\begin{eqnarray}\label{Eq3.4.2}
   && \mathbb{E}[L_t]\geq \mathbb{E}[L_t \mathbbm{1}_{\tilde{Z}_{t-1}\geq s'',\tilde{W}_{t-1}=1}]=\sum_{j=s''}^{s'} \mathbb{E}[L_t \mathbbm{1}_{\tilde{Z}_{t-1}=j,\tilde{W}_{t-1=1}}]\nonumber\\
   &=&\sum_{j=s''}^{s'} \mathbb{E}[\mathbb{E}[L_t|\mathcal{B}_j]\mathbbm{1}_{\tilde{Z}_{t-1}=j,\tilde{W}_{t-1}=1}]\nonumber\\
   &\geq& \sum_{j=s''}^{s'} \mathbb{E}\Big[N_j^{-1}\Big(\sum_{\substack{\mathbf{a}\in \mathcal{A}_O(j+1):|\mathbf{a}|\geq 2, \\\mathbf{a}\text{ does not contain }s}} |\mathbf{a}|\Big)\mathbbm{1}_{\tilde{Z}_{t-1}=j,\tilde{W}_{t-1}=1}\Big] \nonumber\\
   &=& \mathbb{E}\Big[N_{\tilde{Z}_{t-1}}^{-1}\Big(\sum_{\substack{\mathbf{a}\in \mathcal{A}_O(\tilde{Z}_{t-1}+1):|\mathbf{a}|\geq 2, \\\mathbf{a}\text{ does not contain }s}} |\mathbf{a}|\Big)\mathbbm{1}_{\tilde{Z}_{t-1}\geq s'',\tilde{W}_{t-1}=1}\Big] \nonumber\\
   &\geq& \mathbb{E}\Big[\frac{|U_{t-1}\backslash V_{t-1}|}{N_{\tilde{Z}_{t-1}}}\mathbbm{1}_{\tilde{Z}_{t-1}\geq s'',\tilde{W}_{t-1}=1}\Big]\geq \mathbb{E}\Big[\frac{|U\backslash V|}{N_{\tilde{Z}_{t-1}}}\mathbbm{1}_{\tilde{Z}_{t-1}\geq s'',\tilde{W}_{t-1}=1}\Big],  \nonumber\\
   &&
\end{eqnarray}
where we use the fact that $\{\tilde{Z}_{t-1}=j,\tilde{W}_{t-1}=1\}\in \mathcal{B}_{j}$ for any $j\in [s'',s']\cap\mathbb{N}$ in the second line.

We denote by $C_1\geq 1$ the constant $C_0$ appearing in Proposition \ref{P3.3}. Let $\mathcal{V}$ be the event that $e^{-2}\beta^{-1}\slash 4\leq N_l \leq C_1\beta^{-1}$ for any $l\in [s'',s]\cap\mathbb{N}$. By (\ref{Eq3.4.1}),
\begin{equation}\label{Eq3.4.10}
    s''\geq n\slash 4\geq \beta^{-2}\slash 4\geq \beta^{-1},  \quad s-s''+1\leq (c_4+1\slash 10)\beta^{-2}+2\leq C\beta^{-2}.
\end{equation}
Hence by Propositions \ref{P3.1} and \ref{P3.3} together with the union bound, we have
\begin{eqnarray}\label{Eq3.4.7}
     \mathbb{P}(\mathcal{V}^c)&\leq& C(s-s''+1)\exp(-c\beta^{-1})\nonumber\\
     &\leq& C\beta^{-2}\exp(-c\beta^{-1})\leq C\exp(-c\beta^{-1}). 
\end{eqnarray}

Note that $|\mathcal{C}_s(\sigma)|\geq \sum_{t=1}^n L_t$, $|\mathcal{C}_s(\sigma)|\geq |V|$, and $n\geq \beta^{-2}\geq c_5\beta^{-1}$. Hence
\begin{equation}\label{Eq3.4.3}
    \mathbb{E}[|\mathcal{C}_s(\sigma)|]\geq \frac{1}{2}\sum_{t=1}^n \mathbb{E}[L_t]+\frac{1}{2}\mathbb{E}[|V|]\geq \frac{1}{2}\sum_{t\in [2,c_5\beta^{-1}]\cap\mathbb{N}} \mathbb{E}[L_t]+\frac{1}{2}\mathbb{E}[|V|]. 
\end{equation}
We also note that
\begin{equation}\label{Eq3.4.6}
    |U\backslash V|\leq |U|\leq s-s'\leq c_4\beta^{-2}.
\end{equation}

For any $t\in [2,c_5\beta^{-1}]\cap\mathbb{N}$, when $\tilde{Z}_{t-1}\geq s''$ and the event $\mathcal{V}$ holds, we have $\tilde{Z}_{t-1}\in [s'',s]\cap\mathbb{N}$ and $N_{\tilde{Z}_{t-1}}\leq C_1\beta^{-1}$; hence by (\ref{Eq3.4.2}) and (\ref{Eq3.4.6}), we have
\begin{eqnarray}\label{Eq3.4.4}
 \mathbb{E}[L_t] &\geq& \mathbb{E}\Big[\frac{|U\backslash V|}{N_{\tilde{Z}_{t-1}}}\mathbbm{1}_{\tilde{Z}_{t-1}\geq s'',\tilde{W}_{t-1}=1}\mathbbm{1}_{\mathcal{V}}\Big]\geq C_1^{-1}\beta\mathbb{E}[|U\backslash V|\mathbbm{1}_{\tilde{Z}_{t-1}\geq s'',\tilde{W}_{t-1}=1}\mathbbm{1}_{\mathcal{V}}] \nonumber\\
  &\geq& C_1^{-1}\beta\mathbb{E}[|U\backslash V|(1-\mathbbm{1}_{\tilde{Z}_{t-1}< s''\text{ or }\tilde{W}_{t-1}=0}-\mathbbm{1}_{\mathcal{V}^c})]\nonumber\\
  &\geq& C_1^{-1}\beta\mathbb{E}[|U\backslash V|]-C_1^{-1}c_4\beta^{-1}\mathbb{P}(\{\tilde{Z}_{t-1}<s''\}\cup\{\tilde{W}_{t-1}=0\})\nonumber\\
  &&-C_1^{-1}c_4\beta^{-1}\mathbb{P}(\mathcal{V}^c).
\end{eqnarray}
By (\ref{Eq3.4.3}) and (\ref{Eq3.4.4}), noting that $|[2,c_5\beta^{-1}]\cap\mathbb{N}|\in [ c_5\beta^{-1}-3, c_5\beta^{-1}]$, we have
\begin{eqnarray}\label{Eq3.4.8}
     \mathbb{E}[|\mathcal{C}_s(\sigma)|]&\geq& \frac{1}{2}C_1^{-1}\beta(c_5\beta^{-1}-3)\mathbb{E}[|U\backslash V|]+\frac{1}{2}\mathbb{E}[|V|] -\frac{1}{2}C_{1}^{-1}c_4c_5\beta^{-2}\mathbb{P}(\mathcal{V}^c)\nonumber\\
     &&-\frac{1}{2}C_1^{-1}c_4 \beta^{-1}\sum_{t\in [2,c_5\beta^{-1}]\cap\mathbb{N}} \mathbb{P}(\{\tilde{Z}_{t-1}<s''\}\cup \{\tilde{W}_{t-1}=0\}).  \nonumber\\
     &&
\end{eqnarray}
By (\ref{Eq3.4.6}), we have 
\begin{eqnarray}\label{Eq3.4.9}
  &&  \frac{1}{2}C_1^{-1}\beta(c_5\beta^{-1}-3)\mathbb{E}[|U\backslash V|]+\frac{1}{2}\mathbb{E}[|V|]\nonumber\\
   &\geq& \frac{1}{2}C_1^{-1}c_5  \mathbb{E}[|U\backslash V|+|V|]-\frac{3}{2}C_1^{-1}\beta\cdot c_4\beta^{-2}\nonumber\\
   &\geq& \frac{1}{2}C_1^{-1}c_5\mathbb{E}[|U|]-2C_1^{-1}c_4\beta^{-1}.
\end{eqnarray}
By (\ref{Eq3.4.7}), (\ref{Eq3.4.8}), and (\ref{Eq3.4.9}), we obtain that
\begin{eqnarray}\label{Eq3.5.6}
    \mathbb{E}[|\mathcal{C}_s(\sigma)|]&\geq&-\frac{1}{2}C_1^{-1}c_4 \beta^{-1}\sum_{t\in [2,c_5\beta^{-1}]\cap\mathbb{N}} \mathbb{P}(\{\tilde{Z}_{t-1}<s''\}\cup \{\tilde{W}_{t-1}=0\}) \nonumber\\
    && +\frac{1}{2}C_1^{-1}c_5\mathbb{E}[|U|]-2C_1^{-1}c_4\beta^{-1}-C\beta^{-2}\exp(-c\beta^{-1})\nonumber\\
    &\geq& -\frac{1}{2}C_1^{-1}c_4 \beta^{-1}\sum_{t\in [2,c_5\beta^{-1}]\cap\mathbb{N}} \mathbb{P}(\{\tilde{Z}_{t-1}<s''\}\cup \{\tilde{W}_{t-1}=0\})\nonumber\\
    &&+\frac{1}{2}C_1^{-1}c_5\mathbb{E}[|U|]-C\beta^{-1}.
\end{eqnarray}

Let $\mathcal{V}'$ be the event that $\max_{k\in \mathcal{H}_l}\{l-k\}\leq \beta^{-9\slash 8}$ for any $l\in [s'',s]\cap\mathbb{N}$. By Lemma \ref{L3.10}, the union bound, and (\ref{Eq3.4.10}), we have
\begin{equation}\label{Eq3.4.14}
    \mathbb{P}((\mathcal{V}')^c)\leq C(s-s''+1)\beta^{-17\slash 8}\exp(-c\beta^{-1\slash 16})\leq C\exp(-c\beta^{-1\slash 16}).
\end{equation}

Recall the definition of $\mathcal{S}_{j,j'}$ for $j\in [n]$ and $j'\in [n+1]$ in \textbf{Case 1}. We have the following lemma, whose proof will be given in Section \ref{Sect.3.3.2}.

\begin{lemma}\label{L3.12}
Assume the above setup. There exist positive absolute constants $C_0',c_0'$, such that for any $j,j'\in [s'',s]\cap\mathbb{N}$ with $j\geq j'$, we have
\begin{eqnarray*}
    \mathbb{P}(\mathcal{S}_{j,j'})&\geq& (\exp(-C_0'\beta^2(j-j')-C_0'\beta^{1\slash 2})-\exp(-c_0'\beta^{-1\slash 4}))_{+} \nonumber\\
    && \times (1-C_0'\exp(-c_0'\beta^{-1\slash 16}))_{+}.
\end{eqnarray*}
\end{lemma}

By (\ref{Eq3.4.10}), for any $j\in [s',s-1]\cap\mathbb{N}$, $j-s''\leq s-s''\leq C\beta^{-2}$. Hence by Lemma \ref{L3.12}, there exists a positive absolute constant $c_6<1\slash 10000$, such that when $\beta\in (0,c_6)$, for any $j\in [s',s-1]\cap\mathbb{N}$, $\mathbb{P}(\mathcal{S}_{j,s''})\geq c_6$. Moreover, noting that $s'\geq s-c_4\beta^{-2}\geq \frac{n+1}{2}-\frac{n}{10}\geq \frac{n}{4}\geq 2$, by the definition of $s'$, we have $s'-1<s-c_4\beta^{-2}$ and $s-s'> c_4\beta^{-2}-1$. Therefore, when $\beta\in (0,c_6)$, 
\begin{equation}\label{Eq3.5.7}
    \mathbb{E}[|U|]=\sum_{j\in [s',s-1]\cap\mathbb{N}} \mathbb{P}(\mathcal{S}_{j,s''})\geq c_6(s-s')\geq c_6(c_4\beta^{-2}-1).
\end{equation}

In the following, we consider any $t\in [2,c_5\beta^{-1}]\cap\mathbb{N}$. We denote by $C_2$ the constant $C'$ appearing in Lemma \ref{L3.5}, and let $\mathcal{V}''$ be the event that for any $l\in [s'',s]\cap\mathbb{N}$, $\sum_{k\in \mathcal{H}_l}(l-k)\leq C_2\beta^{-2}$. By Lemma \ref{L3.5}, the union bound, and (\ref{Eq3.4.10}), we have
\begin{equation}\label{Eq3.5.8}
    \mathbb{P}((\mathcal{V}'')^c)\leq C  (s-s''+1)  \exp(-c\beta^{-1\slash 16})\leq C\exp(-c\beta^{-1\slash 16}).
\end{equation}
By the union bound, we have
\begin{eqnarray}\label{Eq3.5.9}
  && \mathbb{P}(\{\tilde{Z}_{t-1}<s''\}\cup \{\tilde{W}_{t-1}=0\})\nonumber\\
  &=&\mathbb{P}(\{\tilde{Z}_{t-1}<s''\}\cup \{\tilde{Z}_{t-1}\geq s'',\tilde{W}_{t-1}=0\}) \nonumber\\
  &\leq&\mathbb{P}(\mathcal{V}^c)+\mathbb{P}((\mathcal{V}'')^c)+\mathbb{P}(\{\tilde{Z}_{t-1}<s''\}\cap\mathcal{V}\cap\mathcal{V}'')\nonumber\\
  &&+\mathbb{P}(\{\tilde{Z}_{t-1}\geq s'',\tilde{W}_{t-1}=0\}\cap\mathcal{V}\cap\mathcal{V}''). 
\end{eqnarray}

Note that when $\tilde{Z}_{t-1}< s''$, we have
\begin{eqnarray*}
    \sum_{k=1}^{t-1}(\tilde{Z}_{k-1}-\tilde{Z}_k)\mathbbm{1}_{\tilde{Z}_{k-1}\geq s''}\geq s'-s''\geq \frac{\beta^{-2}}{10}.
\end{eqnarray*}
Hence by Markov's inequality, 
\begin{eqnarray}\label{Eq3.4.18}
  &&  \mathbb{P}(\{\tilde{Z}_{t-1}<s''\}\cap\mathcal{V}\cap\mathcal{V}'')\nonumber\\
  &\leq& \mathbb{P}\Big(\Big\{\sum_{k=1}^{t-1}(\tilde{Z}_{k-1}-\tilde{Z}_k)\mathbbm{1}_{\tilde{Z}_{k-1}\geq s''}\geq  \frac{\beta^{-2}}{10}  \Big\}\cap\mathcal{V}\cap\mathcal{V}''\Big) \nonumber\\
  &\leq&  10\beta^2 \mathbb{E}\Big[\Big(\sum_{k=1}^{t-1}(\tilde{Z}_{k-1}-\tilde{Z}_k)\mathbbm{1}_{\tilde{Z}_{k-1}\geq s''}\Big)\mathbbm{1}_{\mathcal{V}\cap\mathcal{V}''}\Big]\nonumber\\
  &=& 10\beta^2\sum_{k=1}^{t-1} \mathbb{E}[(\tilde{Z}_{k-1}-\tilde{Z}_k)\mathbbm{1}_{\tilde{Z}_{k-1}\geq s''}\mathbbm{1}_{\mathcal{V}\cap\mathcal{V}''}].
\end{eqnarray}
For any $k\in [t-1]$, noting that $\tilde{Z}_{k-1}-\tilde{Z}_k=0$ when $\tilde{W}_{k-1}=0$, we obtain that 
\begin{eqnarray}\label{Eq3.4.17}
  &&  \mathbb{E}[(\tilde{Z}_{k-1}-\tilde{Z}_k)\mathbbm{1}_{\tilde{Z}_{k-1}\geq s''}\mathbbm{1}_{\mathcal{V}\cap\mathcal{V}''}]=\mathbb{E}[(\tilde{Z}_{k-1}-\tilde{Z}_k)\mathbbm{1}_{\tilde{Z}_{k-1}\geq s'',\tilde{W}_{k-1}=1}\mathbbm{1}_{\mathcal{V}\cap\mathcal{V}''}]\nonumber\\
  &=& \sum_{j=s''}^{s'} \mathbb{E}[(\tilde{Z}_{k-1}-\tilde{Z}_k)\mathbbm{1}_{\tilde{Z}_{k-1}=j,\tilde{W}_{k-1}=1}\mathbbm{1}_{\mathcal{V}\cap\mathcal{V}''}]\nonumber\\
  &=&\sum_{j=s''}^{s'} \mathbb{E}[(j-\tilde{Z}_k)\mathbbm{1}_{\tilde{Z}_{k-1}=j,\tilde{W}_{k-1}=1}\mathbbm{1}_{\mathcal{V}\cap\mathcal{V}''}]\nonumber\\
  &=& \sum_{j=s''}^{s'} \mathbb{E}[(j-\mathbb{E}[\tilde{Z}_k|\mathcal{B}_{j}])\mathbbm{1}_{\tilde{Z}_{k-1}=j,\tilde{W}_{k-1}=1}\mathbbm{1}_{\mathcal{V}\cap\mathcal{V}''}]\nonumber\\
  &=& \sum_{j=s''}^{s'}\mathbb{E}\bigg[\frac{\sum_{l\in \mathcal{H}_{j}}(j-l)}{N_{j}}\mathbbm{1}_{\tilde{Z}_{k-1}=j,\tilde{W}_{k-1}=1}\mathbbm{1}_{\mathcal{V}\cap\mathcal{V}''}\bigg]\nonumber\\
  &\leq& \mathbb{E}\bigg[\frac{\sum_{l\in\mathcal{H}_{\tilde{Z}_{k-1}}}(\tilde{Z}_{k-1}-l)}{N_{\tilde{Z}_{k-1}}}\mathbbm{1}_{\tilde{Z}_{k-1}\geq s''}\mathbbm{1}_{\mathcal{V}\cap\mathcal{V}''}\bigg],
\end{eqnarray}
where we note that $\{\tilde{Z}_{k-1}=j,\tilde{W}_{k-1}=1\}\in \mathcal{B}_{j}$ for any $j\in [s'',s']\cap\mathbb{N}$ in the fourth line. Note that when $\tilde{Z}_{k-1}\geq s''$ and the event $\mathcal{V}\cap\mathcal{V}''$ holds, we have $\tilde{Z}_{k-1}\in [s'',s]\cap\mathbb{N}$ and 
\begin{equation*}
    N_{\tilde{Z}_{k-1}}\geq e^{-2}\beta^{-1}\slash 4, \quad \sum_{l\in\mathcal{H}_{\tilde{Z}_{k-1}}}(\tilde{Z}_{k-1}-l)\leq C_2\beta^{-2}.
\end{equation*}
Hence by (\ref{Eq3.4.17}),
\begin{equation}\label{Eq3.4.19}
    \mathbb{E}[(\tilde{Z}_{k-1}-\tilde{Z}_k)\mathbbm{1}_{\tilde{Z}_{k-1}\geq s''}\mathbbm{1}_{\mathcal{V}\cap\mathcal{V}''}]\leq 4e^2C_2\beta^{-1}.
\end{equation}
By (\ref{Eq3.4.18}) and (\ref{Eq3.4.19}),
\begin{equation}\label{Eq3.5.10}
    \mathbb{P}(\{\tilde{Z}_{t-1}<s''\}\cap\mathcal{V}\cap\mathcal{V}'')\leq 40e^2C_2\beta(t-1)\leq 40e^2C_2c_5.
\end{equation}

By the union bound, we have 
\begin{eqnarray}\label{Eq3.5.3}
  &&  \mathbb{P}(\{\tilde{Z}_{t-1}\geq s'',\tilde{W}_{t-1}=0\}\cap\mathcal{V}\cap\mathcal{V}'')\nonumber\\
  &\leq& \sum_{k=1}^{t-1} \mathbb{P}(\{\tilde{Z}_{t-1}\geq s'',\tilde{W}_k=0,\tilde{W}_{0}=\cdots=\tilde{W}_{k-1}=1\}\cap\mathcal{V}\cap\mathcal{V}'')\nonumber\\
  &&  +\mathbb{P}(\tilde{W}_0=0).
\end{eqnarray}
When $\beta\in (0,c_4^{1\slash 2})$, $s'\leq s-1$; hence by Lemma \ref{L3.12}, noting that $s-s'\leq c_4\beta^{-2}$, we have 
\begin{eqnarray}\label{Eq3.5.4}
  &&  \mathbb{P}(\tilde{W}_0=0)=1-\mathbb{P}(\mathcal{S}_{s,s'+1})\nonumber\\
  &\leq& 1-(\exp(-C_0'(c_4+\beta^{1\slash 2}))-\exp(-c_0'\beta^{-1\slash 4}))_{+}(1-C_0'\exp(-c_0'\beta^{-1\slash 16}))_{+}.\nonumber\\
  &&
\end{eqnarray}
Below we consider any $k\in [t-1]$. Note that
\begin{eqnarray}\label{Eq3.5.2}
   && \mathbb{P}(\{\tilde{Z}_{t-1}\geq s'',\tilde{W}_k=0,\tilde{W}_{0}=\cdots=\tilde{W}_{k-1}=1\}\cap\mathcal{V}\cap\mathcal{V}'')\nonumber\\
   &\leq& \mathbb{P}(\{\tilde{W}_k=0\}\cap \{\tilde{Z}_{k-1}\geq s'',\tilde{W}_{k-1}=1\}\cap\mathcal{V}\cap\mathcal{V}'')\nonumber\\
   &\leq& \sum_{k'=s''}^{s'} \mathbb{P}(\{\tilde{W}_k=0\}\cap \{\tilde{Z}_{k-1}=k',\tilde{W}_{k-1}=1\}\cap\mathcal{V}\cap\mathcal{V}'')\nonumber\\
   &=& \sum_{k'=s''}^{s'}\mathbb{E}[\mathbb{P}(\tilde{W}_k=0|\mathcal{B}_{k'})\mathbbm{1}_{\tilde{Z}_{k-1}=k',\tilde{W}_{k-1}=1}\mathbbm{1}_{\mathcal{V}\cap\mathcal{V}''}]\nonumber\\
   &=& \sum_{k'=s''}^{s'}\mathbb{E}\Big[\frac{1}{N_{k'}}\cdot\mathbbm{1}_{\tilde{Z}_{k-1}=k',\tilde{W}_{k-1}=1}\mathbbm{1}_{\mathcal{V}\cap\mathcal{V}''}\Big]\nonumber\\
   &=& \mathbb{E}\Big[\frac{1}{N_{\tilde{Z}_{k-1}}}\cdot \mathbbm{1}_{\tilde{Z}_{k-1}\geq s'',\tilde{W}_{k-1}=1}\mathbbm{1}_{\mathcal{V}\cap\mathcal{V}''}\Big],
\end{eqnarray}
where we note that $\{\tilde{Z}_{k-1}=k',\tilde{W}_{k-1}=1\}\in \mathcal{B}_{k'}$ for any $k'\in [s'',s']\cap\mathbb{N}$ in the fourth line. Note that when $\tilde{Z}_{k-1}\geq s''$ and the event $\mathcal{V}\cap\mathcal{V}''$ holds, we have $\tilde{Z}_{k-1}\in [s'',s]\cap\mathbb{N}$ and $N_{\tilde{Z}_{k-1}}\geq e^{-2}\beta^{-1}\slash 4$. Hence by (\ref{Eq3.5.2}), 
\begin{equation}\label{Eq3.5.5}
\mathbb{P}(\{\tilde{Z}_{t-1}\geq s'',\tilde{W}_k=0,\tilde{W}_{0}=\cdots=\tilde{W}_{k-1}=1\}\cap\mathcal{V}\cap\mathcal{V}'')\leq 4e^2\beta.
\end{equation}
By (\ref{Eq3.5.3}), (\ref{Eq3.5.4}), and (\ref{Eq3.5.5}), noting that $t\leq c_5\beta^{-1}$, we obtain that when $\beta\in (0,c_4^{1\slash 2})$, 
\begin{eqnarray}\label{Eq3.5.11}
  &&  \mathbb{P}(\{\tilde{Z}_{t-1}\geq s'',\tilde{W}_{t-1}=0\}\cap\mathcal{V}\cap\mathcal{V}'')\nonumber\\
  &\leq& 1-(\exp(-C_0'(c_4+\beta^{1\slash 2}))-\exp(-c_0'\beta^{-1\slash 4}))_{+}(1-C_0'\exp(-c_0'\beta^{-1\slash 16}))_{+}\nonumber\\
  &&+4e^2c_5.
\end{eqnarray}

By (\ref{Eq3.4.7}), (\ref{Eq3.5.8}), (\ref{Eq3.5.9}), (\ref{Eq3.5.10}), and (\ref{Eq3.5.11}), when $\beta\in (0,c_4^{1\slash 2})$, for any $t\in [2,c_5\beta^{-1}]\cap\mathbb{N}$, 
\begin{eqnarray}\label{Eq3.5.12}
   && \mathbb{P}(\{\tilde{Z}_{t-1}<s''\}\cup \{\tilde{W}_{t-1}=0\})\nonumber\\
   &\leq& 1-(\exp(-C(c_4+\beta^{1\slash 2}))-\exp(-c\beta^{-1\slash 4}))_{+}(1-C\exp(-c\beta^{-1\slash 16}))_{+}\nonumber\\
   &&+Cc_5+C\exp(-c\beta^{-1\slash 16}).
\end{eqnarray}

By (\ref{Eq3.5.6}), (\ref{Eq3.5.7}), and (\ref{Eq3.5.12}), when $\beta\in (0,\min\{c_4^{1\slash 2},c_6\})$, we have
\begin{eqnarray*}
   && \mathbb{E}[|\mathcal{C}_s(\sigma)|] \nonumber\\
   &\geq &-Cc_4c_5\beta^{-2} (1-(\exp(-C(c_4+\beta^{1\slash 2}))-\exp(-c\beta^{-1\slash 4}))_{+}(1-C\exp(-c\beta^{-1\slash 16}))_{+})\nonumber\\
   && +cc_4c_5\beta^{-2}-C\beta^{-1}-Cc_4c_5^2\beta^{-2}.
\end{eqnarray*}
Hence when $\beta>0$ is sufficiently small (depending on $c_4,c_5$),
\begin{equation}
    \mathbb{E}[|\mathcal{C}_s(\sigma)|]\geq cc_4c_5\beta^{-2}-Cc_4c_5(c_4+c_5)\beta^{-2}. 
\end{equation}
Taking $c_4,c_5$ to be sufficiently small, we obtain that
\begin{equation*}
    \mathbb{E}[|\mathcal{C}_s(\sigma)|]\geq c\beta^{-2}
\end{equation*}
when $\beta\in (0,c_0'')$, where $c_0''$ is a positive absolute constant. When $\beta \geq c_0''$, 
\begin{equation*}
    \mathbb{E}[|\mathcal{C}_s(\sigma)|]\geq 1\geq (c_0'')^{2}\beta^{-2}\geq c\beta^{-2}.
\end{equation*}
Thus we conclude that
\begin{equation}\label{Eq3.5.16}
    \mathbb{E}[|\mathcal{C}_s(\sigma)|]\geq c\beta^{-2}\geq c\min\{\max\{\beta^{-2},1\},n\}.
\end{equation}

\bigskip

By (\ref{Eq3.5.14}), (\ref{Eq3.5.15}), and (\ref{Eq3.5.16}), for any $s\in [n]$ such that $s\geq (n+1)\slash 2$, 
\begin{equation}\label{Eq3.5.17}
    \mathbb{E}[|\mathcal{C}_s(\sigma)|]\geq c\min\{\max\{\beta^{-2},1\},n\}.
\end{equation}

Now consider any $s\in [n]$ such that $s< (n+1)\slash 2$. Let $\bar{\sigma}\in S_n$ be such that $\bar{\sigma}(j)=n+1-\sigma(n+1-j)$ for every $j\in [n]$. Note that $|\mathcal{C}_s(\sigma)|=|\mathcal{C}_{n+1-s}(\bar{\sigma})|$. As the distribution of $\bar{\sigma}$ is given by $\mathbb{P}_{n,\beta}$ and $n+1-s\geq (n+1)\slash 2$, by (\ref{Eq3.5.17}),
\begin{equation}\label{Eq3.5.18}
    \mathbb{E}[|\mathcal{C}_s(\sigma)|]=\mathbb{E}[|\mathcal{C}_{n+1-s}(\bar{\sigma})|]\geq c\min\{\max\{\beta^{-2},1\},n\}.
\end{equation}

By (\ref{Eq3.5.17}) and (\ref{Eq3.5.18}), for any $s\in [n]$,
\begin{equation}\label{Eq3.5.20}
     \mathbb{E}[|\mathcal{C}_s(\sigma)|]\geq c\min\{\max\{\beta^{-2},1\},n\}.
\end{equation}
The conclusion of Theorem \ref{Thm1.1} follows from (\ref{Eq3.5.13}) and (\ref{Eq3.5.20}).

\subsubsection{Proofs of Lemmas \ref{L3.11} and \ref{L3.12}}\label{Sect.3.3.2}

In this part, we give the proofs of Lemmas \ref{L3.11} and \ref{L3.12}. We assume the notations in Section \ref{Sect.3.3.1}.

We start with the proof of Lemma \ref{L3.11}.

\begin{proof}[Proof of Lemma \ref{L3.11}]

We consider any $j\in [n]$, and take 
\begin{equation*}
    d_0:=\min\{ \lfloor \beta^{-5\slash 8} \rfloor, n\},  \quad  d_1:= \min\{ \lfloor \beta^{-3\slash 4} \rfloor, n\},  \quad  d_2:= \min\{ \lfloor \beta^{-1} \rfloor, n\}.
\end{equation*}
For each $k\in \{0,1,2\} $, we denote by $\mathcal{T}_{j,k}$ the event that $j$ is contained in an open arc at step $d_k+1$. In the following, we bound $\mathbb{P}(\mathcal{T}_{j,2}|\mathcal{B}_n) \mathbbm{1}_{\mathcal{P}'}$, $\mathbb{P}(\mathcal{T}_{j,1}|\mathcal{B}_{d_2}) \mathbbm{1}_{\mathcal{P}'\cap \mathcal{T}_{j,2}}$, $\mathbb{P}(\mathcal{T}_{j,0}|\mathcal{B}_{d_1}) \mathbbm{1}_{\mathcal{P}'\cap\mathcal{T}_{j,1}}$, and $\mathbb{P}(\mathcal{S}_{j,d+1}|\mathcal{B}_{d_0}) \mathbbm{1}_{\mathcal{P}'\cap\mathcal{T}_{j,0}}$ in \textbf{Parts 1-4}, respectively.

\paragraph{Part 1}

If $j\leq d_2$, we have $\mathbb{P}(\mathcal{T}_{j,2}|\mathcal{B}_n) =1$. Below we consider the case where $j> d_2$.

Recall the definitions of $\{W_t^{(j)}\}_{t=0}^{\infty}$ and $\{Z_t^{(j)}\}_{t=0}^{\infty}$ from Section \ref{Sect.3.1}. For any $t\in \mathbb{N}^{*}$, we have
\begin{eqnarray*}
    \{Z_t^{(j)}\leq d_2\}\cup \{W_t^{(j)}=1\} &=& \{Z_t^{(j)}\leq d_2\}\cup \{Z_t^{(j)}>d_2, W_t^{(j)}=1\}\nonumber\\
    &\subseteq& \mathcal{T}_{j,2}\cup \{Z_t^{(j)}>d_2, W_t^{(j)}=1\},
\end{eqnarray*}
hence
\begin{eqnarray}\label{Eq3.1.7}
    \mathbb{P}(\mathcal{T}_{j,2}|\mathcal{B}_n)\mathbbm{1}_{\mathcal{P}'}&\geq& \mathbb{P}(\{Z_t^{(j)}\leq d_2\}\cup\{W_t^{(j)}=1\}|\mathcal{B}_n)\mathbbm{1}_{\mathcal{P}'}\nonumber\\
    && -\mathbb{P}(Z_t^{(j)}>d_2, W_t^{(j)}=1|\mathcal{B}_n)\mathbbm{1}_{\mathcal{P}'}.
\end{eqnarray}

We bound $\mathbb{P}(\{Z_t^{(j)}\leq d_2\}\cup\{W_t^{(j)}=1\}|\mathcal{B}_n)\mathbbm{1}_{\mathcal{P}'}$ for any $t\in \mathbb{N}^{*}$ as follows. Note that 
\begin{eqnarray}\label{Eq3.109}
  &&  \mathbb{P}(\{Z_t^{(j)}\leq d_2\}\cup\{W_t^{(j)}=1\}|\mathcal{B}_n) \geq \mathbb{P}(\{Z_{t-1}^{(j)}\leq d_2\}\cup\{W_t^{(j)}=1\}|\mathcal{B}_n)\nonumber\\
  &&=\mathbb{E}[\mathbbm{1}_{Z_{t-1}^{(j)}\leq d_2}+\mathbbm{1}_{Z_{t-1}^{(j)}>d_2}\mathbbm{1}_{W_t^{(j)}=1}|\mathcal{B}_n].
\end{eqnarray}
Let $\mathcal{F}_0^{(j)}:=\mathcal{B}_n$. For any $l\in \mathbb{N}^{*}$, let $\mathcal{F}_l^{(j)}$ be the $\sigma$-algebra generated by $\sigma_0$, $\{b_k\}_{k=1}^n$, $\{W_k^{(j)}\}_{k=1}^l$, and $\{Z_k^{(j)}\}_{k=1}^{l}$. Following the proof of Lemma \ref{L3.6}, we can deduce that for any $k\in \mathbb{N}^{*}$,
\begin{equation*}
    \mathbb{P}(W_k^{(j)}=1|\mathcal{F}_{k-1}^{(j)})=\Big(1-\frac{1}{N_{Z_{k-1}^{(j)}}}\Big)\mathbbm{1}_{W_{k-1}^{(j)}=1}.
\end{equation*}
Moreover, when the event $\{Z_{t-1}^{(j)}>d_2\}\cap\mathcal{P}'$ holds, $Z_{t-1}^{(j)}\geq \lfloor\beta^{-1}\rfloor+1> \beta^{-1}$ and $N_{Z_{t-1}^{(j)}}\geq e^{-2}\beta^{-1}\slash 4$. Hence   
\begin{eqnarray}\label{Eq3.110}
  &&  \mathbb{E}[\mathbbm{1}_{Z_{t-1}^{(j)}>d_2}\mathbbm{1}_{W_t^{(j)}=1}|\mathcal{B}_n]\mathbbm{1}_{\mathcal{P}'}=\mathbb{E}[\mathbbm{1}_{Z_{t-1}^{(j)}>d_2}\mathbbm{1}_{\mathcal{P}'} \mathbb{P}(W_t^{(j)}=1|\mathcal{F}_{t-1}^{(j)})|\mathcal{B}_n] \nonumber\\
  &=& \mathbb{E}\Big[\mathbbm{1}_{Z_{t-1}^{(j)}>d_2}\mathbbm{1}_{\mathcal{P}'}\Big(1-\frac{1}{N_{Z_{t-1}^{(j)}}}\Big)\mathbbm{1}_{W_{t-1}^{(j)}=1}\Big|\mathcal{B}_n\Big]\nonumber\\
  &\geq&  (1-4e^2\beta)\mathbb{E}[\mathbbm{1}_{Z_{t-1}^{(j)}>d_2}\mathbbm{1}_{\mathcal{P}'}\mathbbm{1}_{W_{t-1}^{(j)}=1}|\mathcal{B}_n].
\end{eqnarray}
By (\ref{Eq3.109}) and (\ref{Eq3.110}), we have
\begin{eqnarray*}
  &&  \mathbb{P}(\{Z_t^{(j)}\leq d_2\}\cup\{W_t^{(j)}=1\}|\mathcal{B}_n)\mathbbm{1}_{\mathcal{P}'}\nonumber\\
  &\geq& (1-4e^2\beta)\mathbb{E}[\mathbbm{1}_{Z_{t-1}^{(j)}\leq d_2}+\mathbbm{1}_{Z_{t-1}^{(j)}>d_2}\mathbbm{1}_{W_{t-1}^{(j)}=1}|\mathcal{B}_n] \mathbbm{1}_{\mathcal{P}'}\nonumber\\
  &=& (1-4e^2\beta)\mathbb{P}(\{Z_{t-1}^{(j)}\leq d_2\}\cup \{W_{t-1}^{(j)}=1\}|\mathcal{B}_n)\mathbbm{1}_{\mathcal{P}'}.
\end{eqnarray*}
Continuing similarly, noting that $W_0^{(j)}=1$, we obtain that
\begin{eqnarray}\label{Eq3.1.8}
  &&  \mathbb{P}(\{Z_t^{(j)}\leq d_2\}\cup\{W_t^{(j)}=1\}|\mathcal{B}_n)\mathbbm{1}_{\mathcal{P}'}\geq \cdots \nonumber\\
  &\geq& (1-4e^2\beta)^t \mathbb{P}(\{Z_0^{(j)}\leq d_2\}\cup \{W_0^{(j)}=1\}|\mathcal{B}_n)\mathbbm{1}_{\mathcal{P}'}=(1-4e^2\beta)^t \mathbbm{1}_{\mathcal{P}'}. \nonumber\\
  &&
\end{eqnarray}

We bound $\mathbb{P}(Z_t^{(j)}>d_2, W_t^{(j)}=1|\mathcal{B}_n)\mathbbm{1}_{\mathcal{P}'}$ for any $t\in \mathbb{N}^{*}$ as follows. Consider any $\theta\geq 0$. We have
\begin{eqnarray}\label{Eq3.1.6}
  &&  \mathbb{P}(Z_t^{(j)}>d_2, W_t^{(j)}=1|\mathcal{B}_n)\mathbbm{1}_{\mathcal{P}'}\nonumber\\
  &=& \mathbb{E}\Big[\mathbbm{1}_{\sum_{k=1}^t (Z_{k-1}^{(j)}-Z_{k}^{(j)})<j-d_2}\mathbbm{1}_{W_t^{(j)}=1} \mathbbm{1}_{Z_{t-1}^{(j)}>d_2} 
 \mathbbm{1}_{\mathcal{P}'} \Big| \mathcal{B}_n \Big] \nonumber\\
  &\leq& \exp((j-d_2)\theta)\mathbb{E}\Big[\exp\Big(-\theta\sum_{k=1}^t (Z_{k-1}^{(j)}-Z_k^{(j)})\Big)\mathbbm{1}_{W_t^{(j)}=1} \mathbbm{1}_{Z_{t-1}^{(j)}>d_2} \mathbbm{1}_{\mathcal{P}'}\Big|\mathcal{B}_n\Big] .  \nonumber\\
  &&
\end{eqnarray}
Following the proof of Lemma \ref{L3.6}, we obtain that for any $k\in [n]$,
\begin{equation*}
    \mathbb{P}(Z_t^{(j)}=k, W_t^{(j)}=1|\mathcal{F}_{t-1}^{(j)})=\frac{\mathbbm{1}_{k\in \mathcal{H}_{Z_{t-1}^{(j)}}}\mathbbm{1}_{W_{t-1}^{(j)}=1}}{N_{Z_{t-1}^{(j)}}}.
\end{equation*}
Hence 
\begin{eqnarray}\label{Eq3.1.1}
   &&  \mathbb{E}[\exp(-\theta (Z^{(j)}_{t-1}-Z^{(j)}_t))\mathbbm{1}_{W_t^{(j)}=1}|\mathcal{F}_{t-1}^{(j)}] \nonumber\\
   &=& \sum_{k=1}^n \mathbb{E}[\exp(-\theta(Z^{(j)}_{t-1}-k))\mathbbm{1}_{Z_t^{(j)}=k}\mathbbm{1}_{W_t^{(j)}=1}|\mathcal{F}_{t-1}^{(j)} ] \nonumber\\
   &=& \sum_{k=1}^n \exp(-\theta (Z_{t-1}^{(j)}-k))\mathbb{P}(Z_t^{(j)}=k, W_t^{(j)}=1|  \mathcal{F}_{t-1}^{(j)})\nonumber\\
   &=& \frac{\mathbbm{1}_{W_{t-1}^{(j)}=1}}{N_{Z_{t-1}^{(j)}}}\sum_{k\in \mathcal{H}_{Z_{t-1}^{(j)}}}\exp(-\theta (Z_{t-1}^{(j)}-k)).
\end{eqnarray}
Note that by Lemma \ref{L3.2}, $|\mathcal{H}_{Z_{t-1}^{(j)}}|=N_{Z_{t-1}^{(j)}}-1$. By the definition of $\mathcal{H}_{Z_{t-1}^{(j)}}$, 
\begin{equation}\label{Eq3.1.3}
    \sum_{k\in \mathcal{H}_{Z_{t-1}^{(j)}}}(Z_{t-1}^{(j)}-k)\geq \sum_{k=1}^{N_{Z_{t-1}^{(j)}}-1}k=\frac{1}{2}(N_{Z_{t-1}^{(j)}}-1)N_{Z_{t-1}^{(j)}}.
\end{equation}
When the event $\mathcal{P}' \cap \{Z_{t-1}^{(j)}>d_2\}$ holds, we have $Z_{t-1}^{(j)}>\lfloor \beta^{-1}\rfloor$ and 
\begin{equation*}
    N_{Z_{t-1}^{(j)}}\geq \frac{e^{-2}\lfloor \beta^{-1}\rfloor}{4}\geq \frac{e^{-2}\beta^{-1}}{8},
\end{equation*}
hence by (\ref{Eq3.1.3}), 
\begin{equation}\label{Eq3.1.4}
    \sum_{k\in \mathcal{H}_{Z_{t-1}^{(j)}}}(Z_{t-1}^{(j)}-k) \geq \frac{e^{-2}}{16}\beta^{-1}(N_{Z_{t-1}^{(j)}}-1).
\end{equation}
We also note that when the event $\mathcal{P}'$ holds, for any $k\in \mathcal{H}_{Z_{t-1}^{(j)}}$, 
\begin{equation}\label{Eq3.1.2}
    0\leq Z_{t-1}^{(j)}-k \leq \beta^{-9\slash 8}.
\end{equation}
If $N_{Z_{t-1}^{(j)}}=1$, the right-hand side of (\ref{Eq3.1.1}) is $0$. If $N_{Z_{t-1}^{(j)}}\geq 2$ and the event $\mathcal{P}'\cap \{Z_{t-1}^{(j)}>d_2\}$ holds, by (\ref{Eq3.1.4}), (\ref{Eq3.1.2}), and Hoeffding's lemma (see \cite[Lemma 2.6]{MR2319879}), we have
\begin{eqnarray}\label{Eq3.1.5}
  && \frac{1}{N_{Z_{t-1}^{(j)}}-1}\sum_{k\in \mathcal{H}_{Z_{t-1}^{(j)}}}\exp(-\theta (Z_{t-1}^{(j)}-k))\nonumber\\
  &\leq& \exp\Big(-\frac{\theta}{N_{Z_{t-1}^{(j)}}-1}  \sum_{k\in \mathcal{H}_{Z_{t-1}^{(j)}}}(Z_{t-1}^{(j)}-k)\Big)\exp\Big(\frac{1}{2}\beta^{-9\slash 4}\theta^2\Big) \nonumber\\
  &\leq& \exp\Big(-\frac{e^{-2}}{16}\beta^{-1}\theta+\frac{1}{2}\beta^{-9\slash 4}\theta^2\Big).
\end{eqnarray}
By (\ref{Eq3.1.1}) and (\ref{Eq3.1.5}), 
\begin{eqnarray*}
  &&  \mathbb{E}\Big[\exp\Big(-\theta\sum_{k=1}^t (Z_{k-1}^{(j)}-Z_k^{(j)})\Big)\mathbbm{1}_{W_t^{(j)}=1} \mathbbm{1}_{Z_{t-1}^{(j)}>d_2} \mathbbm{1}_{\mathcal{P}'}\Big|\mathcal{B}_n\Big] \nonumber\\
  &=& \mathbb{E}\Big[\mathbb{E}[\exp(-\theta(Z^{(j)}_{t-1}-Z^{(j)}_t))\mathbbm{1}_{W_t^{(j)}=1}|\mathcal{F}_{t-1}^{(j)} ]\nonumber\\
  &&\quad\times\exp\Big(-\theta\sum_{k=1}^{t-1} (Z_{k-1}^{(j)}-Z_k^{(j)})\Big)\mathbbm{1}_{Z_{t-1}^{(j)}>d_2} \mathbbm{1}_{\mathcal{P}'}\Big|\mathcal{B}_n\Big] \nonumber\\
  &\leq& \exp\Big(-\frac{e^{-2}}{16}\beta^{-1}\theta+\frac{1}{2}\beta^{-9\slash 4}\theta^2\Big) \nonumber\\
  && \times \mathbb{E}\Big[ \exp\Big(-\theta\sum_{k=1}^{t-1} (Z_{k-1}^{(j)}-Z_k^{(j)})\Big)\mathbbm{1}_{W_{t-1}^{(j)}=1}  \mathbbm{1}_{Z_{t-1}^{(j)}>d_2} \mathbbm{1}_{\mathcal{P}'} \Big|\mathcal{B}_n\Big],
\end{eqnarray*}
hence by (\ref{Eq3.1.6}), 
\begin{eqnarray*}
    &&  \mathbb{P}(Z_t^{(j)}>d_2, W_t^{(j)}=1|\mathcal{B}_n)\mathbbm{1}_{\mathcal{P}'}\nonumber\\
    &\leq& \exp((j-d_2)\theta)\exp(-e^{-2}\beta^{-1}  \theta \slash 16 +\beta^{-9\slash 4}\theta^2\slash 2)\nonumber\\
    && \times\mathbb{E}\Big[ \exp\Big(-\theta\sum_{k=1}^{t-1} (Z_{k-1}^{(j)}-Z_k^{(j)})\Big)\mathbbm{1}_{W_{t-1}^{(j)}=1}  \mathbbm{1}_{Z_{t-1}^{(j)}>d_2} \mathbbm{1}_{\mathcal{P}'} \Big|\mathcal{B}_n\Big].
\end{eqnarray*}
Continuing similarly, we obtain that for any $\theta\geq 0$,
\begin{eqnarray}\label{Eq3.1.9}
   && \mathbb{P}(Z_t^{(j)}>d_2, W_t^{(j)}=1|\mathcal{B}_n)\mathbbm{1}_{\mathcal{P}'} \nonumber\\
   &\leq& \exp((j-d_2)\theta-e^{-2}\beta^{-1}t\theta\slash 16+\beta^{-9\slash 4}t\theta^2\slash 2)\mathbbm{1}_{\mathcal{P}'}.
\end{eqnarray}

By (\ref{Eq3.1.7}), (\ref{Eq3.1.8}), and (\ref{Eq3.1.9}), for any $t\in \mathbb{N}^{*}$ and $\theta\geq 0$, we have
\begin{equation}\label{Eq3.1.11}
 \mathbb{P}(\mathcal{T}_{j,2}|\mathcal{B}_n)\mathbbm{1}_{\mathcal{P}'}\geq ((1-4e^2\beta)^t -\exp((j-d_2)\theta-e^{-2}\beta^{-1}t\theta\slash 16+\beta^{-9\slash 4}t\theta^2\slash 2))\mathbbm{1}_{\mathcal{P}'}.
\end{equation}
Note that (\ref{Eq3.1.11}) also holds when $j\leq d_2$. We take $t=\lceil 32e^2\max\{\beta^{-1\slash 2},\beta j\}\rceil$ and $\theta=e^{-2}\beta^{5\slash 4}\slash 32$. Note that  
\begin{equation*}
   \beta^{-9\slash 4}t\theta^2\slash 2=e^{-2}\beta^{-1}t\theta\slash 64, \quad e^{-2}  \beta^{-1}t\theta\slash 16   \geq  2j\theta\geq 2(j-d_2)\theta.
\end{equation*}
Hence by (\ref{Eq3.1.11}), noting that $t\in [32e^2\beta^{-1\slash 2},32e^2\beta j+33e^2\beta^{-1\slash 2}]$,\\ $4e^2\beta\in [0,e^{-1}]$, and $1-x\geq e^{-ex},\forall x\in [0,  e^{-1}]$, we have 
\begin{eqnarray}\label{Eq3.3.3}
    \mathbb{P}(\mathcal{T}_{j,2}|\mathcal{B}_n)\mathbbm{1}_{\mathcal{P}'}  &\geq&   (\exp(-4e^3\beta t)-\exp(-e^{-2}\beta^{-1}t\theta\slash 64))_{+}\mathbbm{1}_{\mathcal{P}'} \nonumber\\
    &\geq& (\exp(-C(\beta^{1\slash 2}+\beta^2j))-\exp(-c\beta^{-1\slash 4}))_{+}\mathbbm{1}_{\mathcal{P}'}.
\end{eqnarray}

\paragraph{Part 2}

If $j\leq d_1$ or $d_1=d_2$, we have $\mathbb{P}(\mathcal{T}_{j,1}|\mathcal{B}_{d_2})\mathbbm{1}_{\mathcal{P}'\cap\mathcal{T}_{j,2}}=\mathbbm{1}_{\mathcal{P}'\cap\mathcal{T}_{j,2}}$. Below we consider the case where $j> d_1$ and $d_1<d_2$. We define $N_{n+1}:=1$.

If $j$ is contained in an open arc at step $d_2+1$, we let $\tilde{Z}_0^{(j)}$ be the tail of this open arc and $\tilde{W}_0^{(j)}=1$; if $j$ is contained in a closed arc at step $d_2+1$, we let $\tilde{Z}_0^{(j)}=d_2+1$ and $\tilde{W}_0^{(j)}=0$. In the following, we define $\tilde{W}_t^{(j)}\in \{0,1\}$ and $\tilde{Z}_t^{(j)}\in [n+1]$ for each $t\in \mathbb{N}^{*}$ inductively. For any $t\in \mathbb{N}^{*}$, suppose that $\tilde{W}_{t-1}^{(j)}\in \{0,1\}$ and $\tilde{Z}_{t-1}^{(j)}\in [n+1]$ have been defined. If $\tilde{W}_{t-1}^{(j)}=0$, we let $\tilde{W}_t^{(j)}=0$ and $\tilde{Z}_t^{(j)}=\tilde{Z}_{t-1}^{(j)}$. Below we consider the case where $\tilde{W}_{t-1}^{(j)}=1$. If $j$ is contained in a closed arc at step $\tilde{Z}_{t-1}^{(j)}$, we let $\tilde{W}_t^{(j)}=0$ and $\tilde{Z}_t^{(j)}=\tilde{Z}_{t-1}^{(j)}$; if $j$ is contained in an open arc at step $\tilde{Z}_{t-1}^{(j)}$, we take $\tilde{W}_t^{(j)}=1$ and let $\tilde{Z}_t^{(j)}$ be the tail of this open arc (note that $\tilde{Z}_t^{(j)}<\tilde{Z}_{t-1}^{(j)}$ in this case). 

For any $t\in \mathbb{N}^{*}$, we have
\begin{eqnarray*}
    \{\tilde{Z}_t^{(j)}\leq d_1\}\cup \{\tilde{W}_t^{(j)}=1\} &=& \{\tilde{Z}_t^{(j)}\leq d_1\}\cup \{\tilde{Z}_t^{(j)}>d_1, \tilde{W}_t^{(j)}=1\}\nonumber\\
    &\subseteq& \mathcal{T}_{j,1}\cup \{\tilde{Z}_t^{(j)}>d_1, \tilde{W}_t^{(j)}=1\},
\end{eqnarray*}
hence
\begin{eqnarray}\label{Eq3.2.9}
    \mathbb{P}(\mathcal{T}_{j,1}|\mathcal{B}_{d_2})\mathbbm{1}_{\mathcal{P}'\cap\mathcal{T}_{j,2}}&\geq& \mathbb{P}(\{\tilde{Z}_t^{(j)}\leq d_1\}\cup\{\tilde{W}_t^{(j)}=1\}|\mathcal{B}_{d_2})\mathbbm{1}_{\mathcal{P}'\cap\mathcal{T}_{j,2}}\nonumber\\
    && -\mathbb{P}(\tilde{Z}_t^{(j)}>d_1, \tilde{W}_t^{(j)}=1|\mathcal{B}_{d_2})\mathbbm{1}_{\mathcal{P}'\cap\mathcal{T}_{j,2}}.
\end{eqnarray}

We bound $\mathbb{P}(\{\tilde{Z}_t^{(j)}\leq d_1\}\cup\{\tilde{W}_t^{(j)}=1\}|\mathcal{B}_{d_2})\mathbbm{1}_{\mathcal{P}'\cap\mathcal{T}_{j,2}}$ for any $t\in \mathbb{N}^{*}$ as follows. Note that 
\begin{eqnarray}\label{Eq3.2.1}
    && \mathbb{P}(\{\tilde{Z}_t^{(j)}\leq d_1\}\cup\{\tilde{W}_t^{(j)}=1\}|\mathcal{B}_{d_2})\geq \mathbb{P}(\{\tilde{Z}_{t-1}^{(j)}\leq d_1\}\cup\{\tilde{W}_t^{(j)}=1\}|\mathcal{B}_{d_2}) \nonumber\\
    &&= \mathbb{E}[\mathbbm{1}_{\tilde{Z}_{t-1}^{(j)}\leq d_1}+\mathbbm{1}_{\tilde{Z}_{t-1}^{(j)}> d_1}\mathbbm{1}_{\tilde{W}_t^{(j)}=1}|\mathcal{B}_{d_2}].
\end{eqnarray}
Note that when the event $\{\tilde{Z}_{t-1}^{(j)}>d_1\}\cap\mathcal{P}'\cap\mathcal{T}_{j,2}$ holds, $\tilde{Z}_{t-1}^{(j)} \leq \tilde{Z}_0^{(j)}\leq d_2\leq n$, $\tilde{Z}_{t-1}^{(j)}\geq \lfloor \beta^{-3\slash 4} \rfloor+1\geq \beta^{-3\slash 4}$, and $N_{\tilde{Z}_{t-1}^{(j)}}\geq e^{-2}\beta^{-3\slash 4}\slash 4$. 

Hence
\begin{eqnarray}\label{Eq3.2.2}
  &&  \mathbb{E}[\mathbbm{1}_{\tilde{Z}_{t-1}^{(j)}> d_1}\mathbbm{1}_{\tilde{W}_t^{(j)}=1}|\mathcal{B}_{d_2}] \mathbbm{1}_{\mathcal{P}'\cap\mathcal{T}_{j,2}}\nonumber\\
  &=& \sum_{k=d_1+1}^{d_2} \mathbb{E}[ \mathbbm{1}_{\tilde{Z}_{t-1}^{(j)}=k,\tilde{W}_{t-1}^{(j)}=1}\mathbbm{1}_{\tilde{W}_t^{(j)}=1}\mathbbm{1}_{\mathcal{P}'\cap\mathcal{T}_{j,2}}|\mathcal{B}_{d_2}]\nonumber\\
  &=& \sum_{k=d_1+1}^{d_2} \mathbb{E}[\mathbb{P}(\tilde{W}_t^{(j)}=1|\mathcal{B}_{k}) \mathbbm{1}_{\tilde{Z}_{t-1}^{(j)}=k,\tilde{W}_{t-1}^{(j)}=1} \mathbbm{1}_{\mathcal{P}'\cap\mathcal{T}_{j,2}}|\mathcal{B}_{d_2}]\nonumber\\
  &=& \sum_{k=d_1+1}^{d_2} \mathbb{E}\Big[\Big(1-\frac{1}{N_k}\Big)\mathbbm{1}_{\tilde{W}_{t-1}^{(j)}=1}\mathbbm{1}_{\tilde{Z}_{t-1}^{(j)}=k} \mathbbm{1}_{\mathcal{P}'\cap\mathcal{T}_{j,2}}\Big| \mathcal{B}_{d_2}\Big]\nonumber\\
  &=& \mathbb{E}\Big[\Big(1-\frac{1}{N_{\tilde{Z}_{t-1}^{(j)}}}\Big)\mathbbm{1}_{\tilde{W}_{t-1}^{(j)}=1} \mathbbm{1}_{\tilde{Z}_{t-1}^{(j)}>d_1} \mathbbm{1}_{\mathcal{P}'\cap\mathcal{T}_{j,2}}  \Big| \mathcal{B}_{d_2} \Big]\nonumber\\
  &\geq& (1-4e^2\beta^{3\slash 4}) \mathbb{E}[\mathbbm{1}_{\tilde{W}_{t-1}^{(j)}=1}\mathbbm{1}_{\tilde{Z}_{t-1}^{(j)}>d_1}\mathbbm{1}_{\mathcal{P}'\cap\mathcal{T}_{j,2}}| \mathcal{B}_{d_2} ].
\end{eqnarray}
By (\ref{Eq3.2.1}) and (\ref{Eq3.2.2}), we have
\begin{eqnarray*}
  &&  \mathbb{P}(\{\tilde{Z}_t^{(j)}\leq d_1\}\cup\{\tilde{W}_t^{(j)}=1\}|\mathcal{B}_{d_2})\mathbbm{1}_{\mathcal{P}'\cap\mathcal{T}_{j,2}} \nonumber\\
  &\geq& (1-4e^2\beta^{3\slash 4}) \mathbb{E}[\mathbbm{1}_{\tilde{Z}_{t-1}^{(j)}\leq d_1}+\mathbbm{1}_{\tilde{W}_{t-1}^{(j)}=1}\mathbbm{1}_{\tilde{Z}_{t-1}^{(j)}>d_1}| \mathcal{B}_{d_2} ]\mathbbm{1}_{\mathcal{P}'\cap\mathcal{T}_{j,2}}\nonumber\\
  &=&  (1-4e^2\beta^{3\slash 4}) \mathbb{P}(\{\tilde{Z}_{t-1}^{(j)}\leq d_1\}\cup \{\tilde{W}_{t-1}^{(j)}=1\}| \mathcal{B}_{d_2} )\mathbbm{1}_{\mathcal{P}'\cap\mathcal{T}_{j,2}}.
\end{eqnarray*}
Continuing similarly, noting that $\tilde{W}_0^{(j)}=1$ when the event $\mathcal{T}_{j,2}$ holds, we obtain that
\begin{eqnarray}\label{Eq3.2.10}
    &&  \mathbb{P}(\{\tilde{Z}_t^{(j)}\leq d_1\}\cup\{\tilde{W}_t^{(j)}=1\}|\mathcal{B}_{d_2})\mathbbm{1}_{\mathcal{P}'\cap\mathcal{T}_{j,2}} \nonumber\\
    &\geq& \cdots\geq (1-4e^2\beta^{3\slash 4})^t \mathbb{P}(\{\tilde{Z}_0^{(j)}\leq d_1\}\cup\{\tilde{W}_0^{(j)}=1\}|\mathcal{B}_{d_2})\mathbbm{1}_{\mathcal{P}'\cap\mathcal{T}_{j,2}}\nonumber\\
    &=& (1-4e^2\beta^{3\slash 4})^t \mathbbm{1}_{\mathcal{P}'\cap\mathcal{T}_{j,2}}.
\end{eqnarray}

We bound $\mathbb{P}(\tilde{Z}_t^{(j)}>d_1, \tilde{W}_t^{(j)}=1|\mathcal{B}_{d_2})\mathbbm{1}_{\mathcal{P}'\cap\mathcal{T}_{j,2}}$ for any $t\in \mathbb{N}^{*}$ as follows. Note that when the event $\mathcal{T}_{j,2}$ holds, $\tilde{Z}_0^{(j)}\leq d_2$. Consider any $\theta\geq 0$. We have
\begin{eqnarray}\label{Eq3.2.8}
    && \mathbb{P}(\tilde{Z}_t^{(j)}>d_1, \tilde{W}_t^{(j)}=1|\mathcal{B}_{d_2})\mathbbm{1}_{\mathcal{P}'\cap\mathcal{T}_{j,2}}  \nonumber\\
    & \leq & \mathbb{E}\Big[\mathbbm{1}_{\sum_{k=1}^t(\tilde{Z}_{k-1}^{(j)}-\tilde{Z}_{k}^{(j)})< d_2-d_1}\mathbbm{1}_{\tilde{W}_t^{(j)}=1}\mathbbm{1}_{\tilde{Z}_{t-1}^{(j)}>d_1}\mathbbm{1}_{\mathcal{P}'\cap\mathcal{T}_{j,2}}\Big|\mathcal{B}_{d_2}\Big] \nonumber\\
    &\leq& \exp(\theta(d_2-d_1))\mathbb{E}\Big[\exp\Big(-\theta\sum_{k=1}^t (\tilde{Z}_{k-1}^{(j)}-\tilde{Z}_k^{(j)})\Big)\mathbbm{1}_{\tilde{W}_t^{(j)}=1}\mathbbm{1}_{\tilde{Z}_{t-1}^{(j)}>d_1}\mathbbm{1}_{\mathcal{P}'\cap\mathcal{T}_{j,2}}\Big|\mathcal{B}_{d_2}\Big].\nonumber\\
    && 
\end{eqnarray}
Note that
\begin{eqnarray*}
  &&  \mathbb{E}\Big[\exp\Big(-\theta\sum_{k=1}^t (\tilde{Z}_{k-1}^{(j)}-\tilde{Z}_k^{(j)})\Big)\mathbbm{1}_{\tilde{W}_t^{(j)}=1}\mathbbm{1}_{\tilde{Z}_{t-1}^{(j)}>d_1}\mathbbm{1}_{\mathcal{P}'\cap\mathcal{T}_{j,2}}\Big|\mathcal{B}_{d_2}\Big] \nonumber\\
  &=&   \sum_{k'=d_1+1}^{d_2}   \mathbb{E}\Big[\exp\Big(-\theta\sum_{k=1}^t (\tilde{Z}_{k-1}^{(j)}-\tilde{Z}_k^{(j)})\Big)\mathbbm{1}_{\tilde{W}_t^{(j)}=1}\mathbbm{1}_{\tilde{Z}_{t-1}^{(j)}=k',\tilde{W}_{t-1}^{(j)}=1 }\mathbbm{1}_{\mathcal{P}'\cap\mathcal{T}_{j,2}}\Big|\mathcal{B}_{d_2}\Big]  \nonumber\\
  &=& \sum_{k'=d_1+1}^{d_2} \mathbb{E}\Big[\exp\Big(-\theta\sum_{k=1}^{t-1} (\tilde{Z}_{k-1}^{(j)}-\tilde{Z}_k^{(j)})\Big)\mathbbm{1}_{\tilde{Z}_{t-1}^{(j)}=k',\tilde{W}_{t-1}^{(j)}=1}\mathbbm{1}_{\mathcal{P}'\cap\mathcal{T}_{j,2}}\nonumber\\
  &&\quad\quad\quad\quad\quad\times\mathbb{E}[\exp(-\theta(\tilde{Z}_{t-1}^{(j)}-\tilde{Z}_t^{(j)}))\mathbbm{1}_{\tilde{W}_t^{(j)}=1}|\mathcal{B}_{k'}]\Big|\mathcal{B}_{d_2}\Big],
\end{eqnarray*}
and for any $k'\in [d_1+1,d_2]\cap\mathbb{N}$, 
\begin{eqnarray*}
  &&  \mathbb{E}[\exp(-\theta(\tilde{Z}_{t-1}^{(j)}-\tilde{Z}_t^{(j)}))\mathbbm{1}_{\tilde{W}_t^{(j)}=1}|\mathcal{B}_{k'}]\mathbbm{1}_{\tilde{Z}_{t-1}^{(j)}=k',\tilde{W}_{t-1}^{(j)}=1} \nonumber\\
  &=& \sum_{l=1}^{n} \mathbb{E}[\exp(-\theta(k'-l))\mathbbm{1}_{\tilde{Z}_t^{(j)}=l}\mathbbm{1}_{\tilde{W}_t^{(j)}=1}|\mathcal{B}_{k'}]\mathbbm{1}_{\tilde{Z}_{t-1}^{(j)}=k',\tilde{W}_{t-1}^{(j)}=1}\nonumber\\
  &=& \sum_{l=1}^{n} \exp(-\theta(k'-l))\mathbbm{1}_{\tilde{Z}_{t-1}^{(j)}=k',\tilde{W}_{t-1}^{(j)}=1} \mathbb{P}(\tilde{Z}_t^{(j)}=l,\tilde{W}_t^{(j)}=1|\mathcal{B}_{k'})\nonumber\\
  &=& \sum_{l=1}^n \exp(-\theta(k'-l))\frac{\mathbbm{1}_{\tilde{Z}_{t-1}^{(j)}=k',\tilde{W}_{t-1}^{(j)}=1}\mathbbm{1}_{l\in \mathcal{H}_{k'}}}{N_{k'}}   \nonumber\\                        &=&  \frac{\mathbbm{1}_{\tilde{Z}_{t-1}^{(j)}=k',  \tilde{W}_{t-1}^{(j)}=1}}{N_{k'}}\sum_{l\in \mathcal{H}_{k'}} \exp(-\theta(k'-l)).
\end{eqnarray*}
Hence
\begin{eqnarray}\label{Eq3.2.7}
      &&  \mathbb{E}\Big[\exp\Big(-\theta\sum_{k=1}^t (\tilde{Z}_{k-1}^{(j)}-\tilde{Z}_k^{(j)})\Big)\mathbbm{1}_{\tilde{W}_t^{(j)}=1}\mathbbm{1}_{\tilde{Z}_{t-1}^{(j)}>d_1}\mathbbm{1}_{\mathcal{P}'\cap\mathcal{T}_{j,2}}\Big|\mathcal{B}_{d_2}\Big] \nonumber\\
      &=& \sum_{k'=d_1+1}^{d_2} \mathbb{E}\Big[\exp\Big(-\theta\sum_{k=1}^{t-1} (\tilde{Z}_{k-1}^{(j)}-\tilde{Z}_k^{(j)})\Big)\mathbbm{1}_{\tilde{Z}_{t-1}^{(j)}=k',\tilde{W}_{t-1}^{(j)}=1 }\mathbbm{1}_{\mathcal{P}'\cap\mathcal{T}_{j,2}}\nonumber\\
  &&\quad\quad\quad\quad\quad\times\frac{1}{N_{k'}}\sum_{l\in \mathcal{H}_{k'}} \exp(-\theta(k'-l))\Big|\mathcal{B}_{d_2}\Big].
\end{eqnarray}
Below we consider any $k'\in [d_1+1,d_2]\cap\mathbb{N}$. By Lemma \ref{L3.2}, $|\mathcal{H}_{k'}|=N_{k'}-1$. By the definition of $\mathcal{H}_{k'}$, 
\begin{equation}\label{Eq3.2.3}
    \sum_{l\in\mathcal{H}_{k'}}(k'-l)\geq \sum_{l=1}^{N_{k'}-1}l=\frac{1}{2}(N_{k'}-1)N_{k'}.
\end{equation}
As $d_1+1\leq k' \leq d_2$, we have $k'\geq\lfloor \beta^{-3\slash 4} \rfloor+1\geq \beta^{-3\slash 4}$. Hence by (\ref{Eq3.2.3}), when the event $\mathcal{P}'$ holds, $N_{k'}\geq e^{-2}\beta^{-3\slash 4}\slash 4$ and 
\begin{equation}\label{Eq3.2.4}
    \sum_{l\in\mathcal{H}_{k'}}(k'-l)\geq \frac{e^{-2}}{8}\beta^{-3\slash 4}(N_{k'}-1).
\end{equation}
We also note that for any $l \in \mathcal{H}_{k'}$, 
\begin{equation}\label{Eq3.2.5}
    0\leq k'-l \leq k'\leq d_2\leq \beta^{-1}.
\end{equation}
If $N_{k'}\geq 2$ and the event $\mathcal{P}'$ holds, by (\ref{Eq3.2.4}), (\ref{Eq3.2.5}), and Hoeffding's lemma, 
\begin{eqnarray}\label{Eq3.2.6}
  &&  \frac{1}{N_{k'}}\sum_{l\in \mathcal{H}_{k'}} \exp(-\theta(k'-l)) \leq  \frac{1}{N_{k'}-1}\sum_{l\in \mathcal{H}_{k'}} \exp(-\theta(k'-l)) \nonumber\\
  &\leq&  \exp\Big(-\frac{\theta}{N_{k'}-1}\sum_{l\in\mathcal{H}_{k'}}(k'-l)\Big)\exp\Big(\frac{1}{2}\beta^{-2}\theta^2\Big)\nonumber\\
  &\leq& \exp\Big(-\frac{e^{-2}}{8}\beta^{-3\slash 4}\theta+\frac{1}{2}\beta^{-2}\theta^2\Big).
\end{eqnarray}
Note that if $N_{k'}=1$, (\ref{Eq3.2.6}) also holds. By (\ref{Eq3.2.7}) and (\ref{Eq3.2.6}),  
\begin{eqnarray*}
    &&  \mathbb{E}\Big[\exp\Big(-\theta\sum_{k=1}^t (\tilde{Z}_{k-1}^{(j)}-\tilde{Z}_k^{(j)})\Big)\mathbbm{1}_{\tilde{W}_t^{(j)}=1}\mathbbm{1}_{\tilde{Z}_{t-1}^{(j)}>d_1}\mathbbm{1}_{\mathcal{P}'\cap\mathcal{T}_{j,2}}\Big|\mathcal{B}_{d_2}\Big] \nonumber\\
    &\leq& \exp\Big(-\frac{e^{-2}}{8}\beta^{-3\slash 4}\theta+\frac{1}{2}\beta^{-2}\theta^2\Big)\nonumber\\
    &&  \times\sum_{k'=d_1+1}^{d_2} \mathbb{E}\Big[\exp\Big(-\theta\sum_{k=1}^{t-1} (\tilde{Z}_{k-1}^{(j)}-\tilde{Z}_k^{(j)})\Big)\mathbbm{1}_{\tilde{W}_{t-1}^{(j)}=1}\mathbbm{1}_{\tilde{Z}_{t-1}^{(j)}=k' }\mathbbm{1}_{\mathcal{P}'\cap\mathcal{T}_{j,2}}\Big|\mathcal{B}_{d_2}\Big]\nonumber\\
    &\leq& \exp\Big(-\frac{e^{-2}}{8}\beta^{-3\slash 4}\theta+\frac{1}{2}\beta^{-2}\theta^2\Big)\nonumber\\
    &&  \times
    \mathbb{E}\Big[\exp\Big(-\theta\sum_{k=1}^{t-1} (\tilde{Z}_{k-1}^{(j)}-\tilde{Z}_k^{(j)})\Big)\mathbbm{1}_{\tilde{W}_{t-1}^{(j)}=1}\mathbbm{1}_{\tilde{Z}_{t-1}^{(j)}>d_1 }\mathbbm{1}_{\mathcal{P}'\cap\mathcal{T}_{j,2}}\Big|\mathcal{B}_{d_2}\Big],
\end{eqnarray*}
hence by (\ref{Eq3.2.8}),
\begin{eqnarray*}
    && \mathbb{P}(\tilde{Z}_t^{(j)}>d_1, \tilde{W}_t^{(j)}=1|\mathcal{B}_{d_2})\mathbbm{1}_{\mathcal{P}'\cap\mathcal{T}_{j,2}}  \nonumber\\
    &\leq& \exp((d_2-d_1)\theta)\exp\Big(-\frac{e^{-2}}{8}\beta^{-3\slash 4}\theta+\frac{1}{2}\beta^{-2}\theta^2\Big)\nonumber\\
    &&  \times
    \mathbb{E}\Big[\exp\Big(-\theta\sum_{k=1}^{t-1} (\tilde{Z}_{k-1}^{(j)}-\tilde{Z}_k^{(j)})\Big)\mathbbm{1}_{\tilde{W}_{t-1}^{(j)}=1}\mathbbm{1}_{\tilde{Z}_{t-1}^{(j)}>d_1 }\mathbbm{1}_{\mathcal{P}'\cap\mathcal{T}_{j,2}}\Big|\mathcal{B}_{d_2}\Big].
\end{eqnarray*}
Continuing similarly, we obtain that for any $\theta\geq 0$,
\begin{eqnarray}\label{Eq3.2.11}
    && \mathbb{P}(\tilde{Z}_t^{(j)}>d_1, \tilde{W}_t^{(j)}=1|\mathcal{B}_{d_2})\mathbbm{1}_{\mathcal{P}'\cap\mathcal{T}_{j,2}}  \nonumber\\
    &\leq& \exp( (d_2-d_1)\theta-e^{-2} \beta^{-3\slash 4}t\theta \slash 8 +\beta^{-2} t \theta^2\slash 2 )\mathbbm{1}_{\mathcal{P}'\cap\mathcal{T}_{j,2}}.
\end{eqnarray}

By (\ref{Eq3.2.9}), (\ref{Eq3.2.10}), and (\ref{Eq3.2.11}), for any $t\in \mathbb{N}^{*}$ and $\theta\geq 0$, we have 
\begin{eqnarray}\label{Eq3.2.12}
   && \mathbb{P}(\mathcal{T}_{j,1}|\mathcal{B}_{d_2})\mathbbm{1}_{\mathcal{P}'\cap\mathcal{T}_{j,2}}\geq (1-4e^2\beta^{3\slash 4})^t\mathbbm{1}_{\mathcal{P}'\cap\mathcal{T}_{j,2}}\nonumber\\
   &&\quad\quad -\exp( (d_2-d_1)\theta-e^{-2} \beta^{-3\slash 4}t\theta \slash 8 +\beta^{-2} t \theta^2\slash 2 )\mathbbm{1}_{\mathcal{P}'\cap\mathcal{T}_{j,2}}.
\end{eqnarray}
Note that (\ref{Eq3.2.12}) also holds when $j\leq d_1$ or $d_1= d_2$. We take $t=\lceil 32e^2\beta^{-5\slash 8}\rceil$ and $\theta=e^{-2}\beta^{5\slash 4}\slash 16$. Note that 
\begin{equation*}
    \beta^{-2}t\theta^2\slash 2=e^{-2}\beta^{-3\slash 4}t\theta\slash 32,  e^{-2}\beta^{-3\slash 4}t\theta\slash 8\geq 4\beta^{-11\slash 8}\theta\geq 4\beta^{-1}\theta \geq 4(d_2-d_1)\theta.
\end{equation*}
Hence by (\ref{Eq3.2.12}), noting that $t\in [32e^2\beta^{-5\slash 8},33e^2\beta^{-5\slash 8}]$, $4e^2\beta^{3\slash 4}\in [0,e^{-1}]$, and $1-x\geq e^{-ex},\forall x\in [0,  e^{-1}]$, we have
\begin{eqnarray}\label{Eq3.3.4}
 \mathbb{P}(\mathcal{T}_{j,1}|\mathcal{B}_{d_2})\mathbbm{1}_{\mathcal{P}'\cap\mathcal{T}_{j,2}}&\geq& (\exp(-4e^3\beta^{3\slash 4} t  )-\exp(-e^{-2}\beta^{-3\slash 4}t\theta\slash 16) )_{+}\mathbbm{1}_{\mathcal{P}'\cap\mathcal{T}_{j,2}} \nonumber\\
 &\geq& (\exp(-C\beta^{1\slash 8})-\exp(-c\beta^{-1\slash 8}))_{+}\mathbbm{1}_{\mathcal{P}'\cap\mathcal{T}_{j,2}}.
\end{eqnarray}

\paragraph{Part 3}

Following a similar argument as in \textbf{Part 2}, using $d_1\leq \beta^{-3\slash 4}$ in the analog of (\ref{Eq3.2.5}), we can deduce that for any $t 
  \in \mathbb{N}^{*}$ and $\theta\geq 0$,
 \begin{eqnarray}\label{Eq3.3.1}
     && \mathbb{P}(\mathcal{T}_{j,0}|\mathcal{B}_{d_1}) \mathbbm{1}_{\mathcal{P}'\cap\mathcal{T}_{j,1}}\geq  (1-4e^2\beta^{5\slash 8})^t\mathbbm{1}_{\mathcal{P}'\cap\mathcal{T}_{j,1}}\nonumber\\
    &&\quad\quad -\exp( (d_1-d_0)\theta-e^{-2} \beta^{-5\slash 8}t\theta \slash 8 +\beta^{-3\slash 2} t \theta^2\slash 2 )\mathbbm{1}_{\mathcal{P}'\cap\mathcal{T}_{j,1}}.
\end{eqnarray}
We take $t=\lceil 32e^2\beta^{-3\slash 8} \rceil$ and $\theta=e^{-2}\beta^{7\slash 8}\slash 16$. Note that
\begin{equation*}
     \beta^{-3\slash 2} t \theta^2\slash 2=e^{-2}\beta^{-5\slash 8}t\theta\slash 32, e^{-2} \beta^{-5\slash 8}t\theta \slash 8\geq  4\beta^{-1}\theta\geq 4\beta^{-3\slash 4}\theta\geq 4(d_1-d_0)\theta.
\end{equation*}
Hence by (\ref{Eq3.3.1}), noting that $t\in [32e^2\beta^{-3\slash 8},33e^2\beta^{-3\slash 8}]$, $4e^2\beta^{5\slash 8}\in [0,e^{-1}]$, and $1-x\geq e^{-ex},\forall x\in [0,  e^{-1}]$, we have
\begin{eqnarray}\label{Eq3.3.5}
     \mathbb{P}(\mathcal{T}_{j,0}|\mathcal{B}_{d_1})  \mathbbm{1}_{\mathcal{P}'\cap\mathcal{T}_{j,1}}&\geq& (\exp(-4e^3\beta^{5\slash 8}t)-\exp(-e^{-2}\beta^{-5\slash 8}t\theta\slash 16))_{+}  \mathbbm{1}_{\mathcal{P}'\cap\mathcal{T}_{j,1}} \nonumber\\
     &\geq& (\exp(-C\beta^{1\slash 4})-\exp(-c\beta^{-1\slash 8}))_{+} \mathbbm{1}_{\mathcal{P}'\cap\mathcal{T}_{j,1}}. 
\end{eqnarray}

\paragraph{Part 4}

Following a similar argument as in \textbf{Part 2}, using $d_0\leq \beta^{-5\slash 8}$ in the analog of (\ref{Eq3.2.5}), we can deduce that for any $t 
 \in \mathbb{N}^{*}$ and $\theta\geq 0$,
\begin{eqnarray}\label{Eq3.3.2}
    && \mathbb{P}(\mathcal{S}_{j,d+1}|\mathcal{B}_{d_0}) \mathbbm{1}_{\mathcal{P}'\cap\mathcal{T}_{j,0}}\geq  (1-4e^2\beta^{7\slash 16})^t\mathbbm{1}_{\mathcal{P}'\cap\mathcal{T}_{j,0}}\nonumber\\
   &&\quad\quad -\exp( (d_0-d)\theta-e^{-2} \beta^{-7\slash 16}t\theta \slash 8 +\beta^{-5\slash 4} t \theta^2\slash 2 )\mathbbm{1}_{\mathcal{P}'\cap\mathcal{T}_{j,0}}.
\end{eqnarray}
We take $t=\lceil 32e^2\beta^{-13\slash 32}\rceil$ and $\theta=e^{-2}\beta^{13\slash 16}\slash 16$. Note that
\begin{equation*}
    \beta^{-5\slash 4}t\theta^2\slash 2=e^{-2}\beta^{-7\slash 16}t\theta\slash 32, 
\end{equation*}
\begin{equation*}
    e^{-2} \beta^{-7\slash 16}t\theta \slash 8\geq 4\beta^{-27\slash 32}\theta\geq 4\beta^{-5\slash 8}\theta\geq 4(d_0-d)\theta.
\end{equation*}
Hence by (\ref{Eq3.3.2}), noting that $t\in [32e^2\beta^{-13\slash 32},33e^2\beta^{-13\slash 32}]$, $4e^2\beta^{7\slash 16}\in [0,e^{-0.5}]$, and $1-x\geq e^{-ex},\forall x\in [0,  e^{-0.5}]$, we have
\begin{eqnarray}\label{Eq3.3.6}
    \mathbb{P}(\mathcal{S}_{j,d+1}|\mathcal{B}_{d_0}) \mathbbm{1}_{\mathcal{P}'\cap\mathcal{T}_{j,0}}&\geq& (\exp(-4e^3\beta^{7\slash 16}t)-\exp(-e^{-2}\beta^{-7\slash 16}t\theta\slash 16))_{+}  \mathbbm{1}_{\mathcal{P}'\cap\mathcal{T}_{j,0}} \nonumber\\
    &\geq& (\exp(-C\beta^{1\slash 32})-\exp(-c\beta^{-1\slash 32}))_{+}\mathbbm{1}_{\mathcal{P}'\cap\mathcal{T}_{j,0}}.
\end{eqnarray}

\bigskip

Note that $\mathcal{S}_{j,d+1}\subseteq \mathcal{T}_{j,0}\subseteq \mathcal{T}_{j,1}\subseteq \mathcal{T}_{j,2}$. Hence we conclude that
\begin{eqnarray*}
  &&  \mathbb{P}(\mathcal{S}_{j,d+1}|\mathcal{B}_n)\mathbbm{1}_{\mathcal{P}'}\nonumber\\
  &=&\mathbb{E}[\mathbb{E}[\mathbbm{1}_{\mathcal{S}_{j,d+1}}\mathbbm{1}_{\mathcal{T}_{j,0}}|\mathcal{B}_{d_0}]|\mathcal{B}_n]\mathbbm{1}_{\mathcal{P}'}=\mathbb{E}[\mathbb{E}[\mathbbm{1}_{\mathcal{S}_{j,d+1}}|\mathcal{B}_{d_0}]\mathbbm{1}_{\mathcal{T}_{j,0}}\mathbbm{1}_{\mathcal{P}'}|\mathcal{B}_n] \nonumber\\
  &\geq& (\exp(-C\beta^{1\slash 32})-\exp(-c\beta^{-1\slash 32}))_{+}\mathbb{E}[\mathbbm{1}_{\mathcal{T}_{j,0}}\mathbbm{1}_{\mathcal{P}'}|\mathcal{B}_n]\nonumber\\
  &=& (\exp(-C\beta^{1\slash 32})-\exp(-c\beta^{-1\slash 32}))_{+}\mathbb{E}[\mathbb{E}[\mathbbm{1}_{\mathcal{T}_{j,0}}|\mathcal{B}_{d_1}]\mathbbm{1}_{\mathcal{T}_{j,1}}\mathbbm{1}_{\mathcal{P}'}|\mathcal{B}_n]\nonumber\\
  &\geq& (\exp(-C\beta^{1\slash 32})-\exp(-c\beta^{-1\slash 32}))_{+}^2  \mathbb{E}[\mathbbm{1}_{\mathcal{T}_{j,1}}\mathbbm{1}_{\mathcal{P}'}|\mathcal{B}_n]\nonumber\\
  &=& (\exp(-C\beta^{1\slash 32})-\exp(-c\beta^{-1\slash 32}))_{+}^2 
  \mathbb{E}[\mathbb{E}[\mathbbm{1}_{\mathcal{T}_{j,1}}|\mathcal{B}_{d_2}]\mathbbm{1}_{\mathcal{T}_{j,2}}\mathbbm{1}_{\mathcal{P}'}|\mathcal{B}_n]\nonumber\\
  &\geq& (\exp(-C\beta^{1\slash 32})-\exp(-c\beta^{-1\slash 32}))_{+}^3 \mathbb{E}[\mathbbm{1}_{\mathcal{T}_{j,2}}\mathbbm{1}_{\mathcal{P}'}|\mathcal{B}_n]\nonumber\\
  &=&(\exp(-C\beta^{1\slash 32})-\exp(-c\beta^{-1\slash 32}))_{+}^3 \mathbb{E}[\mathbbm{1}_{\mathcal{T}_{j,2}}|\mathcal{B}_n]\mathbbm{1}_{\mathcal{P}'}\nonumber\\
  &\geq& (\exp(-C\beta^{1\slash 32})-\exp(-c\beta^{-1\slash 32}))_{+}^3\nonumber\\
  &&  \times (\exp(-C(\beta^{1\slash 2}+\beta^2j))-\exp(-c\beta^{-1\slash 4}))_{+}\mathbbm{1}_{\mathcal{P}'},
\end{eqnarray*}
where we use (\ref{Eq3.3.6}) in the third line, (\ref{Eq3.3.5}) in the fifth line, (\ref{Eq3.3.4}) in the seventh line, and (\ref{Eq3.3.3}) in the ninth line.

\end{proof}

Now we give the proof of Lemma \ref{L3.12}.

\begin{proof}[Proof of Lemma \ref{L3.12}]

Consider any $j,j'\in [s'',s]\cap\mathbb{N}$ with $j\geq j'$. For any $t\in\mathbb{N}$, we recall the definitions of $W_t^{(j)}$ and $Z_t^{(j)}$ from Section \ref{Sect.3.1}, and define $\mathcal{F}_t^{(j)}$ as in \textbf{Part 1} of the proof of Lemma \ref{L3.11}. For any $t\in \mathbb{N}^{*}$, we have
\begin{eqnarray*}
    \{Z_t^{(j)}\leq j'-1\}\cup \{W_t^{(j)}=1\}&=&\{Z_t^{(j)}\leq j'-1\}\cup \{Z_t^{(j)}\geq j', W_t^{(j)}=1\} \nonumber\\
    &\subseteq& \mathcal{S}_{j,j'}\cup \{Z_t^{(j)}\geq j', W_t^{(j)}=1\},
\end{eqnarray*}
hence
\begin{eqnarray}\label{Eq3.4.11}
    \mathbb{P}(\mathcal{S}_{j,j'}|\mathcal{B}_n)\mathbbm{1}_{\mathcal{V}\cap\mathcal{V}'}&\geq& \mathbb{P}(\{Z_t^{(j)}\leq j'-1\}\cup \{W_t^{(j)}=1\}|\mathcal{B}_n)\mathbbm{1}_{\mathcal{V}\cap\mathcal{V}'}\nonumber\\
    && -\mathbb{P}(Z_t^{(j)}\geq j', W_t^{(j)}=1|\mathcal{B}_n)\mathbbm{1}_{\mathcal{V}\cap\mathcal{V}'}.
\end{eqnarray}

For any $t\in \mathbb{N}^{*}$ and $\theta\geq 0$, following the argument between (\ref{Eq3.1.7}) and (\ref{Eq3.1.9}) in \textbf{Part 1} of the proof of Lemma \ref{L3.11} (note that when $Z_{t-1}^{(j)}\geq j'$ and the event $\mathcal{V}\cap\mathcal{V}'$ holds, we have $Z_{t-1}^{(j)}\in [s'',s]\cap\mathbb{N}$, hence $N_{Z_{t-1}^{(j)}}\geq e^{-2}\beta^{-1}\slash 4 $ and $\max_{k\in\mathcal{H}_{Z_{t-1}^{(j)}}}\{Z_{t-1}^{(j)}-k\}\leq\beta^{-9\slash 8}$), we obtain that 
\begin{eqnarray*}
    \mathbb{P}(\{Z_t^{(j)}\leq j'-1\}\cup \{W_t^{(j)}=1\}|\mathcal{B}_n)\mathbbm{1}_{\mathcal{V}\cap\mathcal{V}'}\geq   (1-4e^2\beta)^t \mathbbm{1}_{\mathcal{V}\cap\mathcal{V}'},
\end{eqnarray*}
\begin{eqnarray*}
  &&  \mathbb{P}(Z_t^{(j)}\geq j', W_t^{(j)}=1|\mathcal{B}_n)\mathbbm{1}_{\mathcal{V}\cap\mathcal{V}'} \nonumber\\
  &\leq& \exp((j-j')\theta-e^{-2}\beta^{-1}t\theta\slash 16+\beta^{-9\slash 4}t\theta^2\slash 2)\mathbbm{1}_{\mathcal{V}\cap\mathcal{V}'}.
\end{eqnarray*}
Hence by (\ref{Eq3.4.11}), for any $t\in \mathbb{N}^{*}$ and $\theta\geq 0$,
\begin{eqnarray}\label{Eq3.4.12}
  &&  \mathbb{P}(\mathcal{S}_{j,j'}|\mathcal{B}_n)\mathbbm{1}_{\mathcal{V}\cap\mathcal{V}'} \nonumber\\
  &\geq& ((1-4e^2\beta)^t-\exp((j-j')\theta-e^{-2}\beta^{-1}t\theta\slash 16+\beta^{-9\slash 4}t\theta^2\slash 2))\mathbbm{1}_{\mathcal{V}\cap\mathcal{V}'}. \nonumber\\
  &&
\end{eqnarray}
We take $t=\lceil 16e^2\max\{4(j-j')\beta,\beta^{-1\slash 2}\}\rceil$ and $\theta=e^{-2}\beta^{5\slash 4}\slash 32$. Note that 
\begin{equation*}
    \beta^{-9\slash 4}t\theta^2\slash 2=e^{-2}\beta^{-1}t\theta\slash 64, \quad e^{-2}\beta^{-1}t\theta\slash 16\geq 4(j-j')\theta.
\end{equation*}
Hence by (\ref{Eq3.4.12}), noting that $1-x\geq e^{-ex},\forall x\in [0,  e^{-1}]$, $4e^2\beta\leq e^{-1}$, and $t\in [16e^2\beta^{-1\slash 2},64e^2(j-j')\beta+17 e^2\beta^{-1\slash 2}]$, we have
\begin{eqnarray}\label{Eq3.5.1}
  \mathbb{P}(\mathcal{S}_{j,j'}|\mathcal{B}_n)\mathbbm{1}_{\mathcal{V}\cap\mathcal{V}'}&\geq& (\exp(-4e^3\beta t)-\exp(-e^{-2}\beta^{-1}t\theta\slash 32))_{+}\mathbbm{1}_{\mathcal{V}\cap\mathcal{V}'}\nonumber\\
     &\geq& (\exp(-C\beta^2(j-j')-C\beta^{1\slash 2})-\exp(-c\beta^{-1\slash 4}))_{+}\mathbbm{1}_{\mathcal{V}\cap\mathcal{V}'}.\nonumber\\
     &&
\end{eqnarray}

By (\ref{Eq3.4.7}), (\ref{Eq3.4.14}), and the union bound,
\begin{equation*}
    \mathbb{P}(\mathcal{V}\cap\mathcal{V}')\geq 1-\mathbb{P}(\mathcal{V}^c)-\mathbb{P}((\mathcal{V}')^c)\geq 1-C\exp(-c\beta^{-1\slash 16}).
\end{equation*}
Hence by (\ref{Eq3.5.1}),
\begin{eqnarray}
  &&  \mathbb{P}(\mathcal{S}_{j,j'})\geq \mathbb{E}[\mathbb{P}(\mathcal{S}_{j,j'}|\mathcal{B}_n)\mathbbm{1}_{\mathcal{V}\cap\mathcal{V}'}] \nonumber\\
  &\geq& (\exp(-C\beta^2(j-j')-C\beta^{1\slash 2})-\exp(-c\beta^{-1\slash 4}))_{+}\mathbb{P}(\mathcal{V}\cap\mathcal{V}')\nonumber\\
  &\geq& (\exp(-C\beta^2(j-j')-C\beta^{1\slash 2})-\exp(-c\beta^{-1\slash 4}))_{+}(1-C\exp(-c\beta^{-1\slash 16}))_{+}.\nonumber\\
  &&
\end{eqnarray}

\end{proof}

\subsubsection{Proof of the lower bound in Theorem \ref{Thm1.3.1}}\label{Sect.3.3.3}

In this part, we give the proof of the lower bound in Theorem \ref{Thm1.3.1}. 

We first consider any $s\in [n]$ such that $s\geq 2$. By (\ref{Basic_n}) and Lemma \ref{L3.1}, 
\begin{equation}\label{Eq3.20.1}
    \mathbb{P}(N_s\geq 2)\geq \mathbb{P}(b_{s-1}\geq s)=\mathbb{E}[e^{-2\beta(s-\max\{s-1,\sigma_0(s-1)\})_{+}}]\geq e^{-2\beta}.
\end{equation}
We also note that
\begin{equation}\label{Eq3.20.2}
    \mathbb{P}(\sigma(s)\neq s|\mathcal{B}_s)\mathbbm{1}_{N_s\geq 2}=\mathbb{P}(Y_s\neq s|\mathcal{B}_s)\mathbbm{1}_{N_s\geq 2}\geq \frac{1}{2}\mathbbm{1}_{N_s\geq 2}.
\end{equation}
By (\ref{Eq3.20.1}) and (\ref{Eq3.20.2}), we have 
\begin{eqnarray}\label{Eq3.20.3}
 \mathbb{P}(\max(\mathcal{C}_s(\sigma))-\min(\mathcal{C}_s(\sigma))\geq 1)&\geq&\mathbb{P}(\sigma(s)\neq s)\geq \mathbb{E}[\mathbb{P}(\sigma(s)\neq s|\mathcal{B}_s)\mathbbm{1}_{N_s\geq 2}] \nonumber\\
 &\geq&\frac{1}{2}\mathbb{P}(N_s\geq 2)\geq \frac{1}{2}e^{-2\beta}. 
\end{eqnarray}

Let $\bar{\sigma}\in S_n$ be such that $\bar{\sigma}(j)=n+1-\sigma(n+1-j)$ for any $j\in [n]$. As the distribution of $\bar{\sigma}$ is $\mathbb{P}_{n,\beta}$, if $n\geq 2$, by (\ref{Eq3.20.3}), we have
\begin{equation}\label{Eq3.20.4}
   \mathbb{P}(\max(\mathcal{C}_1(\sigma))-\min(\mathcal{C}_1(\sigma))\geq 1) \geq \mathbb{P}(\sigma(1)\neq 1)=\mathbb{P}(\bar{\sigma}(n)\neq n)\geq \frac{1}{2}e^{-2\beta}.
\end{equation}

For any $s\in [n]$, by (\ref{Eq3.20.3}) and (\ref{Eq3.20.4}), 
\begin{equation}\label{Eq3.20.6}
    \mathbb{E}[\max(\mathcal{C}_s(\sigma))-\min(\mathcal{C}_s(\sigma))]\geq \frac{1}{2}\min\{e^{-2\beta},n-1\}.
\end{equation}
Moreover, as $|\mathcal{C}_s(\sigma)|\leq \max(\mathcal{C}_s(\sigma))-\min(\mathcal{C}_s(\sigma))+1$, by Theorem \ref{Thm1.1} (which we have established in Sections \ref{Sect.3.3.1} and \ref{Sect.3.3.2}), we have
\begin{equation}\label{Eq3.20.7}
    \mathbb{E}[\max(\mathcal{C}_s(\sigma))-\min(\mathcal{C}_s(\sigma))]\geq \mathbb{E}[|\mathcal{C}_s(\sigma)|]-1\geq c \min\{\max\{\beta^{-2},1\},n\}-1.
\end{equation}
If $\beta\leq 1$ and $n\geq 2$, by (\ref{Eq3.20.6}), $\mathbb{E}[\max(\mathcal{C}_s(\sigma))-\min(\mathcal{C}_s(\sigma))]\geq e^{-2}\slash 2$, hence by (\ref{Eq3.20.7}),
\begin{equation*}
    \mathbb{E}[\max(\mathcal{C}_s(\sigma))-\min(\mathcal{C}_s(\sigma))]\geq c \min\{\max\{\beta^{-2},1\},n\}.
\end{equation*}
If $\beta>1$ and $n\geq 2$, we have $\beta^{-2}\leq 1$, hence by (\ref{Eq3.20.6}), 
\begin{equation*}
    \mathbb{E}[\max(\mathcal{C}_s(\sigma))-\min(\mathcal{C}_s(\sigma))]\geq \frac{1}{2}\min\{e^{-2\beta}\max\{\beta^{-2},1\}, n-1\}.
\end{equation*}
Therefore, for any $\beta>0$, $n\in \mathbb{N}^{*}$, and $s\in [n]$,
\begin{equation}
    \mathbb{E}[\max(\mathcal{C}_s(\sigma))-\min(\mathcal{C}_s(\sigma))]\geq c\min\{e^{-2\beta}\max\{\beta^{-2},1\}, n-1\}.
\end{equation}

\subsection{Proof of Theorem \ref{Thm1.2}}\label{Sect.3.4}

In this subsection, we give the proof of Theorem \ref{Thm1.2}. We assume the setup in Section \ref{Sect.3.1}, taking $\beta=\beta_n$. We recall the definitions of $\{W_t^{(j)}\}_{t=0}^{\infty}$ and $\{Z_t^{(j)}\}_{t=0}^{\infty}$ for any $j\in [n]$ from Section \ref{Sect.3.1}. 

The following lemma is adapted from \cite[Claim 4.14]{GP}.

\begin{lemma}\label{L3.13}
Let $I$ be a non-empty finite set. Let $A_1,\cdots,A_k$, where $k\in\mathbb{N}$ is random, be pairwise disjoint random subsets of $I$. We further assume that for each $j\in [k]$, there is a random linear order $\preceq$ on $A_j$. For each $i\in I$ such that $i\in A_j$ for some $j\in [k]$ (note that $j$ is uniquely determined by $i$), we let $\alpha_i$ be the size of $A_j$ and $l_i$ the size of the set $\{a\in A_j: a\preceq i\}$. For each $i\in I$ such that $i\notin \bigcup_{j=1}^k A_j$, we let $\alpha_i=0$ and $l_i=0$. Then we have
\begin{equation}\label{Eq3.7.1}
    \Big(\frac{1}{|I|}\mathbb{E}[\max_{i\in I}\alpha_i]\Big)^3\leq \frac{128}{|I|}\max_{i\in I}\mathbb{E}[l_i].
\end{equation}
\end{lemma}
\begin{proof}

Let $L:=\max_{i\in I}\alpha_i$ and $\gamma:=\mathbb{E}[L]\slash |I|$. As $0\leq L\leq |I|$, we have $\gamma\in [0,1]$. If $\gamma=0$, (\ref{Eq3.7.1}) holds. Below we assume that $\gamma\in (0,1]$. Applying Markov's inequality to $|I|-L$, we obtain that
\begin{equation*}
    \mathbb{P}(L<\gamma|I|\slash 2)=\mathbb{P}(|I|-L>(1-\gamma\slash 2)|I|)\leq \frac{\mathbb{E}[|I|-L]}{(1-\gamma\slash 2)|I|}=\frac{1-\gamma}{1-\gamma\slash 2}.
\end{equation*}
Hence
\begin{equation}\label{EE1}
    \mathbb{P}(L\geq \gamma |I|\slash 2) \geq 1-\frac{1-\gamma}{1-\gamma\slash 2}\geq  \frac{1}{2}\gamma.
\end{equation}

Now assume that the event $\{L\geq\gamma |I|\slash 2\}$ holds. Then there exists some $i_0\in I$ such that $\alpha_{i_0}\geq \gamma |I|\slash 2$. Suppose that $i_0\in A_{j_0}$, where $j_0\in [k]$. We have $|A_{j_0}|=\alpha_{i_0}\geq \gamma|I|\slash 2$. If $|A_{j_0}|=1$, then $i_0$ is the unique element of $A_{j_0}$, and $l_{i_0}=1=\alpha_{i_0}\geq \gamma |I|\slash 2$. If $|A_{j_0}|\geq 2$, then for any $i\in I$ that is one of the largest $\lfloor |A_{j_0}|\slash 2 \rfloor$ elements of $A_{j_0}$ (under the linear order $\preceq$), $l_i\geq |A_{j_0}|\slash 2$. Note that when $|A_{j_0}|\geq 2$,
\begin{equation*}
    |A_{j_0}|\slash 2\geq \lfloor |A_{j_0}|\slash 2 \rfloor\geq \frac{|A_{j_0}|-1}{2}\geq \frac{|A_{j_0}|}{4}\geq \frac{\gamma}{8}|I|.
\end{equation*}
Hence for at least $\lceil \gamma|I|\slash 8\rceil$ of $i\in A_{j_0}$, $l_i\geq \gamma|I|\slash 8$. Thus we have
\begin{equation}\label{Eq3.7.3}
    \sum_{i\in I} l_i\geq \Big(\sum_{i\in I}l_i\Big) \mathbbm{1}_{L\geq \gamma|I|\slash 2}\geq \frac{\gamma^2}{64} |I|^2 \mathbbm{1}_{L\geq \gamma |I|\slash 2}.
\end{equation}

By (\ref{EE1}) and (\ref{Eq3.7.3}), we have
\begin{equation*}
    |I|\max_{i\in I}\mathbb{E}[l_i]\geq \mathbb{E}\Big[\sum_{i\in I} l_i\Big]\geq \frac{\gamma^2}{64}|I|^2\mathbb{P}(L\geq \gamma |I|\slash 2)\geq \frac{\gamma^3}{128}|I|^2,
\end{equation*}
which leads to (\ref{Eq3.7.1}).

\end{proof}

We recall Definition \ref{Defi3.1}. Throughout the rest of this subsection, for any $l \in [n+1]$ and any open arc $\mathbf{a}=(a_1,a_2,\cdots,a_s)$ (where $s\in \mathbb{N}^{*}$) at step $l$, we take the linear order $a_1 \succ a_2 \succ \cdots \succ a_s$ on $\mathbf{a}$.

In the following, we assume that $n>100$, $\beta_n \in (0,1\slash 10000)$, and $n<\beta_n^{-2}$. We define $d$, $\mathcal{P}$, $\mathcal{P}'$, and $\mathcal{S}_{j,j'}$ (where $j\in [n]$ and $j'\in [n+1]$) as in \textbf{Case 1} of Section \ref{Sect.3.3.1} (with $\beta=\beta_n$). For any $j\in [n]$, if $j$ is contained in an open arc at step $d+1$, we let $\mathscr{L}^{(j)}$ be the number of elements in this open arc that are $\preceq j$; otherwise we let $\mathscr{L}^{(j)}=0$. 

We have the following two lemmas. 

\begin{lemma}\label{L3.14}
Assume the preceding setup. There exists a positive absolute constant $C_0$, such that for any $j\in [n]$, we have
\begin{equation*}
    \mathbb{E}[\mathscr{L}^{(j)}|\mathcal{B}_n]\mathbbm{1}_{\mathcal{P}'}\leq C_0(\beta_n^2n+\beta_n^{1\slash 32}+\exp(-c\beta_n^{-1\slash 32}))n+1.
\end{equation*}
\end{lemma}
\begin{proof}

We consider any $j\in [n]$. For each $k\in\{0,1,2\}$, we let $d_k$ and $\mathcal{T}_{j,k}$ be defined as in the proof of Lemma \ref{L3.11} in Section \ref{Sect.3.3.2} (with $\beta=\beta_n$), and define $\tau_k^{(j)}:=\inf\{t\in\mathbb{N}:Z_t^{(j)}\leq d_k\}$. We also define $\tau^{(j)}:=\inf\{t\in\mathbb{N}:Z_t^{(j)}\leq d\}$. Note that $\tau^{(j)}\geq \tau_0^{(j)}\geq \tau_1^{(j)}\geq\tau_2^{(j)}$.

If $j$ is contained in a closed arc at step $j$, we let $Q_0^{(j)}=0$; if $j$ is contained in an open arc at step $j$, we let $Q_0^{(j)}$ be the number of elements in this open arc that are $\prec j$. For any $t\in \mathbb{N}^{*}$, if $j$ is contained in a closed arc at step $Z_t^{(j)}$, we let $Q_t^{(j)}=0$; if $j$ is contained in an open arc at step $Z_t^{(j)}$ (note that $j$ is also contained in an open arc at step $Z_{t-1}^{(j)}$ in this case), we let $Q_t^{(j)}$ be the number of elements in the open arc containing $j$ at step $Z_t^{(j)}$ that are $\prec j$ minus the number of elements in the open arc containing $j$ at step $Z_{t-1}^{(j)}$ that are $\prec j$.

We let 
\begin{equation*}
    \mathcal{L}^{(j)}:=\mathbbm{1}_{\tau^{(j)}<\infty}\Big(\sum_{t\in[\tau_0^{(j)},\tau^{(j)}-1]\cap\mathbb{N}}Q_t^{(j)}\Big),
\end{equation*}
\begin{equation*}
    \mathcal{L}_0^{(j)}:=\mathbbm{1}_{\tau_0^{(j)}<\infty}\Big(\sum_{t\in[\tau_1^{(j)},\tau_0^{(j)}-1]\cap\mathbb{N}}Q_t^{(j)}\Big),
\end{equation*}
\begin{equation*}
    \mathcal{L}_1^{(j)}:=\mathbbm{1}_{\tau_1^{(j)}<\infty}\Big(\sum_{t\in [\tau_2^{(j)},\tau_1^{(j)}-1]\cap\mathbb{N}}Q_t^{(j)}\Big),
\end{equation*}
\begin{equation*}
    \mathcal{L}_2^{(j)}:=\mathbbm{1}_{\tau_2^{(j)}<\infty}\Big(\sum_{t\in[0,\tau_2^{(j)}-1]\cap\mathbb{N}}Q_t^{(j)}\Big).
\end{equation*}
Note that
\begin{equation}\label{Eq3.9.20}
    \mathscr{L}^{(j)} \leq  \mathbbm{1}_{\tau^{(j)}<\infty}\Big(1+\sum_{t\in [0,\tau^{(j)}-1]\cap\mathbb{N}} Q_t^{(j)}\Big)\leq 1+\mathcal{L}^{(j)}+\mathcal{L}_0^{(j)}+\mathcal{L}_1^{(j)}+\mathcal{L}_2^{(j)}.
\end{equation}
In the following, we bound $\mathbb{E}[\mathcal{L}_k^{(j)}|\mathcal{B}_n]\mathbbm{1}_{\mathcal{P}'}$ for $k=2,1,0$ and $\mathbb{E}[\mathcal{L}^{(j)}|\mathcal{B}_n]\mathbbm{1}_{\mathcal{P}'}$ in \textbf{Parts 1-4}, respectively.

\paragraph{Part 1}

If $j\leq d_2$, we have $\tau_2^{(j)}=0$ and $\mathcal{L}_2^{(j)}=0$. Below we consider the case where $j>d_2$. 

For any $k\in [d_2+1,n]\cap\mathbb{N}$, $k\geq \beta_n^{-1}$; hence when the event $\mathcal{P}'$ holds, we have $N_k\geq e^{-2}\beta_n^{-1}\slash 4$. Hence for any $t\in \mathbb{N}$ and $k\in [d_2+1,n]\cap\mathbb{N}$, 
\begin{eqnarray}\label{Eq3.9.1}
   && \mathbb{E}[Q_t^{(j)}|\mathcal{B}_{k}]\mathbbm{1}_{Z_t^{(j)}=k, W_t^{(j)}=1}\mathbbm{1}_{\mathcal{P}'}\leq \frac{\sum_{\mathbf{a}\in\mathcal{A}_O(k+1)}|\mathbf{a}|}{N_k}\mathbbm{1}_{Z_t^{(j)}=k,W_t^{(j)}=1}\mathbbm{1}_{\mathcal{P}'}\nonumber\\
   &\leq& \frac{n}{N_k}\mathbbm{1}_{Z_t^{(j)}=k,W_t^{(j)}=1}\mathbbm{1}_{\mathcal{P}'} \leq 4e^2\beta_n n\mathbbm{1}_{Z_t^{(j)}=k}.
\end{eqnarray}
For any $t\in \mathbb{N}$, if $t\leq \tau_2^{(j)}-1$, then $Z_t^{(j)}>d_2$; if $W_t^{(j)}=0$, then $Q_t^{(j)}=0$. Hence by (\ref{Eq3.9.1}), for any $t\in \mathbb{N}$,
\begin{eqnarray}\label{Eq3.9.2}
&& \mathbb{E}[Q_t^{(j)}\mathbbm{1}_{t\leq \tau_2^{(j)}-1}|\mathcal{B}_n]\mathbbm{1}_{\mathcal{P}'}\nonumber\\
&=&\sum_{k\in [d_2+1,n]\cap\mathbb{N}} \mathbb{E}[Q_t^{(j)}\mathbbm{1}_{t\leq \tau_2^{(j)}-1}\mathbbm{1}_{Z_{t}^{(j)}=k,W_t^{(j)}=1}|\mathcal{B}_n]\mathbbm{1}_{\mathcal{P}'} \nonumber\\
    &\leq& \sum_{k\in[d_2+1,n]\cap\mathbb{N}} \mathbb{E}[Q_t^{(j)} \mathbbm{1}_{Z_{t}^{(j)}=k,W_t^{(j)}=1}|\mathcal{B}_n]\mathbbm{1}_{\mathcal{P}'} \nonumber\\
    &=& \sum_{k\in[d_2+1,n]\cap\mathbb{N}} \mathbb{E}[\mathbb{E}[Q_t^{(j)}|\mathcal{B}_{k}]\mathbbm{1}_{Z_{t}^{(j)}=k,W_t^{(j)}=1}\mathbbm{1}_{\mathcal{P}'}|\mathcal{B}_n]\nonumber\\
    &\leq& 4e^2\beta_n n \sum_{k\in [d_2+1,n]\cap\mathbb{N}}\mathbb{E}[\mathbbm{1}_{Z_t^{(j)}=k}|\mathcal{B}_n]\leq 4e^2\beta_n n.
\end{eqnarray}

Note that for any $t\in \mathbb{N}^{*}$, when $t<\tau_2^{(j)}<\infty$, we have $Z_t^{(j)}>d_2$ and $W_t^{(j)}=1$ (if $W_t^{(j)}=0$, then $Z_s^{(j)}=Z_t^{(j)}>d_2$ for any $s\in\mathbb{N}^{*}$ with $s\geq t$, which leads to $\tau_2^{(j)}=\infty$). Hence by (\ref{Eq3.1.9}) (with $\beta=\beta_n$), for any $t\in \mathbb{N}^{*}$ and $\theta\geq 0$,
\begin{eqnarray}\label{Eq3.9.3}
 && \mathbb{P}(t<\tau_2^{(j)}<\infty|\mathcal{B}_n)\mathbbm{1}_{\mathcal{P}'}\leq \mathbb{P}(Z_t^{(j)}>d_2, W_t^{(j)}=1|\mathcal{B}_n)\mathbbm{1}_{\mathcal{P}'} \nonumber\\
   &\leq& \exp((j-d_2)\theta-e^{-2}\beta_n^{-1}t\theta\slash 16+\beta_n^{-9\slash 4}t\theta^2\slash 2).
\end{eqnarray}

Consider any $t\in \mathbb{N}^{*}$ and $\theta\geq 0$. By (\ref{Eq3.9.2}) and the definition of $\mathcal{L}_2^{(j)}$, 
\begin{eqnarray}\label{Eq3.9.4}
  &&  \mathbb{E}[\mathcal{L}_2^{(j)}\mathbbm{1}_{\tau_2^{(j)}\leq t}|\mathcal{B}_n] \mathbbm{1}_{\mathcal{P}'}=\mathbb{E}\Big[\Big(\sum_{s=0}^{\infty}Q_s^{(j)}\mathbbm{1}_{s  \leq  \tau_2^{(j)}-1}\Big)\mathbbm{1}_{\tau_2^{(j)}\leq t}\Big|\mathcal{B}_n\Big] \mathbbm{1}_{\mathcal{P}'} \nonumber\\
  &\leq& \sum_{s=0}^{t-1}\mathbb{E}[Q_s^{(j)}\mathbbm{1}_{s\leq \tau_2^{(j)}-1}|\mathcal{B}_n] \mathbbm{1}_{\mathcal{P}'}\leq 4e^2\beta_n n t.
\end{eqnarray}
By (\ref{Eq3.9.3}) and the definition of $\mathcal{L}_2^{(j)}$, 
\begin{eqnarray}\label{Eq3.9.5}
  &&  \mathbb{E}[\mathcal{L}_2^{(j)}\mathbbm{1}_{\tau_2^{(j)}>t}|\mathcal{B}_n]\mathbbm{1}_{\mathcal{P}'}=\mathbb{E}\Big[\Big(\sum_{s\in [0,\tau_2^{(j)}-1]\cap\mathbb{N}}Q_s^{(j)}\Big)\mathbbm{1}_{t<\tau_2^{(j)}<\infty}\Big|\mathcal{B}_n\Big]\mathbbm{1}_{\mathcal{P}'} \nonumber\\
  &\leq& n \mathbb{P}(t<\tau_2^{(j)}<\infty|\mathcal{B}_n)\mathbbm{1}_{\mathcal{P}'} \nonumber\\
  &\leq& n\exp((j-d_2)\theta-e^{-2}\beta_n^{-1}t\theta\slash 16+\beta_n^{-9\slash 4}t\theta^2\slash 2).
\end{eqnarray}
Combining (\ref{Eq3.9.4}) and (\ref{Eq3.9.5}), we obtain that
\begin{equation}\label{Eq3.9.6}
    \mathbb{E}[\mathcal{L}_2^{(j)}|\mathcal{B}_n]\mathbbm{1}_{\mathcal{P}'}\leq 4e^2\beta_n n t+n\exp((j-d_2)\theta-e^{-2}\beta_n^{-1}t\theta\slash 16+\beta_n^{-9\slash 4}t\theta^2\slash 2).
\end{equation}
We take $t=\lceil 64 e^2\max\{\beta_n n, \beta_n^{-1\slash 2}\} \rceil$ and $\theta=e^{-2}\beta_n^{5\slash 4}\slash 32$. Note that 
\begin{equation*}
    \beta_n^{-9\slash 4}t\theta^2\slash 2=e^{-2}\beta_n^{-1}t\theta\slash 64, \quad e^{-2}\beta_n^{-1}t\theta\slash 16\geq  4n\theta\geq 4(j-d_2)\theta.
\end{equation*}
Hence by (\ref{Eq3.9.6}), we have
\begin{eqnarray}\label{Eq3.9.7}
    \mathbb{E}[\mathcal{L}_2^{(j)}|\mathcal{B}_n]\mathbbm{1}_{\mathcal{P}'}&\leq& 4e^2\beta_n n t +n\exp(-e^{-2}\beta_n^{-1}t\theta\slash 32)\nonumber\\
    &\leq& C(\beta_n^2 n+\beta_n^{1\slash 2}+\exp(-c\beta_n^{-1\slash 4}))n.
\end{eqnarray}
Note that (\ref{Eq3.9.7}) also holds when $j\leq d_2$. 

\paragraph{Part 2}

If $j\leq d_1$, we have $\tau_1^{(j)}=0$ and $\mathcal{L}_1^{(j)}=0$. Below we consider the case where $j>d_1$. 

For any $k\in [d_1+1,d_2]\cap\mathbb{N}$, $k\geq \beta_n^{-3\slash 4}$; hence when the event $\mathcal{P}'$ holds, we have $N_k\geq e^{-2}\beta_n^{-3\slash 4}\slash 4$. Hence for any $t\in \mathbb{N}$ and $k\in [d_1+1,d_2]\cap\mathbb{N}$, 
\begin{eqnarray}\label{Eq3.9.10}
  && \mathbb{E}[Q_t^{(j)}|\mathcal{B}_k]\mathbbm{1}_{Z_t^{(j)}=k,W_t^{(j)}=1}\mathbbm{1}_{\mathcal{P}'}\leq \frac{\sum_{\mathbf{a}\in \mathcal{A}_O(k+1)}|\mathbf{a}|}{N_k}\mathbbm{1}_{Z_t^{(j)}=k,W_t^{(j)}=1}\mathbbm{1}_{\mathcal{P}'}\nonumber\\
  &\leq& \frac{n}{N_k}\mathbbm{1}_{Z_t^{(j)}=k,W_t^{(j)}=1}\mathbbm{1}_{\mathcal{P}'}\leq 4e^2\beta_n^{3\slash 4} n \mathbbm{1}_{Z_t^{(j)}=k}.
\end{eqnarray}
Now note that for any $s,t\in \mathbb{N}$, if $W_{s+t}^{(j)}=0$, then $Q_{s+t}^{(j)}=0$; if $\tau_2^{(j)}=s$ and $s+t+1\leq \tau_1^{(j)}<\infty$, then $d_1+1\leq Z_{s+t}^{(j)}\leq Z_s^{(j)}\leq d_2$. It can also be checked that for any $k\in [d_1+1,d_2]\cap\mathbb{N}$, $\{\tau_2^{(j)}=s\}\cap\{Z_{s+t}^{(j)}=k, W_{s+t}^{(j)}=1\}\in \mathcal{B}_k$. Hence by (\ref{Eq3.9.10}) (replacing $t$ by $s+t$), for any $s,t\in\mathbb{N}$, 
\begin{eqnarray}\label{Eq3.9.11}
  &&  \mathbb{E}[Q_{s+t}^{(j)}\mathbbm{1}_{\tau_2^{(j)}=s}\mathbbm{1}_{s+t+1\leq \tau_1^{(j)}<\infty}|\mathcal{B}_n]\mathbbm{1}_{\mathcal{P}'} \nonumber\\
  &=& \sum_{k\in [d_1+1,d_2]\cap\mathbb{N}} \mathbb{E}[Q_{s+t}^{(j)}\mathbbm{1}_{\tau_2^{(j)}=s}\mathbbm{1}_{s+t+1\leq \tau_1^{(j)}<\infty}\mathbbm{1}_{Z_{s+t}^{(j)}=k,W_{s+t}^{(j)}=1}|\mathcal{B}_n]\mathbbm{1}_{\mathcal{P}'}\nonumber\\
  &\leq& \sum_{k\in [d_1+1,d_2]\cap\mathbb{N}} \mathbb{E}[Q_{s+t}^{(j)}\mathbbm{1}_{\tau_2^{(j)}=s}\mathbbm{1}_{Z_{s+t}^{(j)}=k,W_{s+t}^{(j)}=1}|\mathcal{B}_n]\mathbbm{1}_{\mathcal{P}'}\nonumber\\
  &=& \sum_{k\in [d_1+1,d_2]\cap\mathbb{N}} \mathbb{E}[\mathbb{E}[Q_{s+t}^{(j)}|\mathcal{B}_k]\mathbbm{1}_{\tau_2^{(j)}=s}\mathbbm{1}_{Z_{s+t}^{(j)}=k,W_{s+t}^{(j)}=1}\mathbbm{1}_{\mathcal{P}'}|\mathcal{B}_n] \nonumber\\
  &\leq& 4e^2\beta_n^{3\slash 4}n \sum_{k\in [d_1+1,d_2]\cap\mathbb{N}} \mathbb{E}[\mathbbm{1}_{\tau_2^{(j)}=s}\mathbbm{1}_{Z_{s+t}^{(j)}=k}|\mathcal{B}_n]\nonumber\\
  &\leq& 4e^2\beta_n^{3\slash 4}n\mathbb{E}[\mathbbm{1}_{\tau_2^{(j)}=s}|\mathcal{B}_n]. 
\end{eqnarray}

Let $\{\tilde{Z}_t^{(j)}\}_{t=0}^{\infty}$ and $\{\tilde{W}_t^{(j)}\}_{t=0}^{\infty}$ be defined as in \textbf{Part 2} of the proof of Lemma \ref{L3.11} (in Section \ref{Sect.3.3.2}). For any $t\in\mathbb{N}^{*}$, if the event $\{\tau_1^{(j)}<\infty, \tau_1^{(j)}-\tau_2^{(j)}-1>t\}$ holds, as $\tau_2^{(j)}\leq \tau_1^{(j)}<\infty$, $j$ is contained in an open arc at step $d_2+1$, and the event $\mathcal{T}_{j,2}$ holds; moreover, for any $l\in [0,\tau_1^{(j)}-\tau_2^{(j)}-1]\cap\mathbb{N}$, $\tilde{W}_l^{(j)}=1$ and $\tilde{Z}_l^{(j)}=Z_{\tau_2^{(j)}+l}^{(j)}$, hence $\tilde{W}_t^{(j)}=1$ and $\tilde{Z}_t^{(j)}= Z_{\tau_2^{(j)}+t}^{(j)}\geq Z_{\tau_1^{(j)}-1}^{(j)}>d_1$. Hence by (\ref{Eq3.2.11}) (with $\beta=\beta_n$), for any $t\in \mathbb{N}^{*}$ and $\theta\geq 0$,  
\begin{eqnarray}\label{Eq3.9.12}
  &&  \mathbb{P}(\tau_1^{(j)}<\infty, \tau_1^{(j)}-\tau_2^{(j)}-1>t|\mathcal{B}_{d_2})\mathbbm{1}_{\mathcal{P}'} \nonumber\\
  &\leq& \mathbb{P}(\{\tilde{Z}_t^{(j)}>d_1,\tilde{W}_t^{(j)}=1\}\cap\mathcal{T}_{j,2}|\mathcal{B}_{d_2})\mathbbm{1}_{\mathcal{P}'} \nonumber\\
  &=& \mathbb{P}(\tilde{Z}_t^{(j)}>d_1,\tilde{W}_t^{(j)}=1|\mathcal{B}_{d_2})\mathbbm{1}_{\mathcal{T}_{j,2}}\mathbbm{1}_{\mathcal{P}'} \nonumber\\
  &\leq& \exp( (d_2-d_1)\theta-e^{-2} \beta_n^{-3\slash 4}t\theta \slash 8 +\beta_n^{-2} t \theta^2\slash 2 ).
\end{eqnarray}

By the definition of $\mathcal{L}_1^{(j)}$, 
\begin{eqnarray}\label{Eq3.9.15}
  && \mathbb{E}[\mathcal{L}_1^{(j)}|\mathcal{B}_n]=\mathbb{E}\Big[\mathbbm{1}_{\tau_1^{(j)}<\infty}\Big(\sum_{l\in [\tau_2^{(j)},\tau_1^{(j)}-1]\cap\mathbb{N}}Q_l^{(j)}\Big)\Big|\mathcal{B}_n\Big] \nonumber\\
  &=& \mathbb{E}\Big[\sum_{s=0}^{\infty}\mathbbm{1}_{\tau_2^{(j)}=s}\mathbbm{1}_{\tau_1^{(j)}<\infty}\Big(\sum_{l\in [s,\tau_1^{(j)}-1]\cap\mathbb{N}}Q_l^{(j)}\Big)\Big|\mathcal{B}_n\Big] \nonumber\\
  &=& \mathbb{E}\Big[\sum_{s=0}^{\infty}\mathbbm{1}_{\tau_2^{(j)}=s}\mathbbm{1}_{\tau_1^{(j)}<\infty}\Big(\sum_{l=0}^{\infty}  Q_{s+l}^{(j)}\mathbbm{1}_{l\leq\tau_1^{(j)}-s-1}  \Big)\Big|\mathcal{B}_n\Big]\nonumber\\
  &=& \sum_{s=0}^{\infty}\sum_{l=0}^{\infty}\mathbb{E}\Big[ \mathbbm{1}_{\tau_2^{(j)}=s}\mathbbm{1}_{l+s+1\leq \tau_1^{(j)}<\infty} Q_{s+l}^{(j)} \Big|\mathcal{B}_n\Big].
\end{eqnarray}
Below we consider any $t\in \mathbb{N}^{*}$ and $\theta\geq 0$. By (\ref{Eq3.9.11}), 
\begin{eqnarray}\label{Eq3.9.13}
  && \sum_{s=0}^{\infty}\sum_{l=0}^{t} \mathbb{E}\Big[ \mathbbm{1}_{\tau_2^{(j)}=s}\mathbbm{1}_{l+s+1\leq \tau_1^{(j)}<\infty} Q_{s+l}^{(j)} \Big|\mathcal{B}_n\Big] \mathbbm{1}_{\mathcal{P}'} \nonumber\\
  &\leq& 4e^2\beta_n^{3\slash 4}n\sum_{s=0}^{\infty}\sum_{l=0}^{t} \mathbb{E}[\mathbbm{1}_{\tau_2^{(j)}=s}|\mathcal{B}_n]\leq  8e^2\beta_n^{3\slash 4}n t.
\end{eqnarray}
By (\ref{Eq3.9.12}),
\begin{eqnarray}\label{Eq3.9.14}
    && \sum_{s=0}^{\infty}\sum_{l=t+1}^{\infty} \mathbb{E}\Big[ \mathbbm{1}_{\tau_2^{(j)}=s}\mathbbm{1}_{l+s+1\leq \tau_1^{(j)}<\infty} Q_{s+l}^{(j)} \Big|\mathcal{B}_n\Big] \mathbbm{1}_{\mathcal{P}'} \nonumber\\
    &\leq&  \sum_{s=0}^{\infty} \mathbb{E}\Big[ \mathbbm{1}_{\tau_2^{(j)}=s}\mathbbm{1}_{\tau_1^{(j)}<\infty,\tau_1^{(j)}-\tau_2^{(j)}-1>t}\Big(\sum_{l=t+1}^{\infty} Q_{s+l}^{(j)}\Big) \Big|\mathcal{B}_n\Big] \mathbbm{1}_{\mathcal{P}'}\nonumber\\
    &\leq& n \mathbb{P}(\tau_1^{(j)}<\infty,\tau_1^{(j)}-\tau_2^{(j)}-1>t|\mathcal{B}_n)\mathbbm{1}_{\mathcal{P}'}\nonumber\\
    &=& n\mathbb{E}[\mathbb{P}(\tau_1^{(j)}<\infty,\tau_1^{(j)}-\tau_2^{(j)}-1>t|\mathcal{B}_{d_2})\mathbbm{1}_{\mathcal{P}'}|\mathcal{B}_n]\nonumber\\
    &\leq& n \exp( (d_2-d_1)\theta-e^{-2} \beta_n^{-3\slash 4}t\theta \slash 8 +\beta_n^{-2} t \theta^2\slash 2 ).
\end{eqnarray}
Combining (\ref{Eq3.9.15})-(\ref{Eq3.9.14}), we obtain that
\begin{equation}\label{Eq3.9.16}
    \mathbb{E}[\mathcal{L}_1^{(j)}|\mathcal{B}_n]\mathbbm{1}_{\mathcal{P}'}\leq 8e^2\beta_n^{3\slash 4}nt+ n \exp( (d_2-d_1)\theta-e^{-2} \beta_n^{-3\slash 4}t\theta \slash 8 +\beta_n^{-2} t \theta^2\slash 2 ).
\end{equation}
We take $t=\lceil 32 e^2 \beta_n^{-5\slash 8}\rceil$ and $\theta=e^{-2}\beta_n^{5\slash 4}\slash 16$. Note that  
\begin{equation*}
    \beta_n^{-2} t \theta^2\slash 2=e^{-2} \beta_n^{-3\slash 4}t\theta \slash 32, \quad e^{-2} \beta_n^{-3\slash 4}t\theta \slash 8\geq 4\beta_n^{-11\slash 8}\theta \geq 4\beta_n^{-1}\theta \geq 4(d_2-d_1)\theta.
\end{equation*}
Hence by (\ref{Eq3.9.16}), we have
\begin{eqnarray}\label{Eq3.9.17}
    \mathbb{E}[\mathcal{L}_1^{(j)}|\mathcal{B}_n]\mathbbm{1}_{\mathcal{P}'}&\leq& 8e^2\beta_n^{3\slash 4}nt+n\exp(-e^{-2} \beta_n^{-3\slash 4}t\theta \slash 16) \nonumber\\
    &\leq& C(\beta_n^{1\slash 8}+\exp(-c\beta_n^{-1\slash 8}))n.
\end{eqnarray}
Note that (\ref{Eq3.9.17}) also holds when $j\leq d_1$.

\paragraph{Part 3}

Following a similar argument as in \textbf{Part 2}, we can deduce that for any $t\in \mathbb{N}^{*}$ and $\theta\geq 0$,  
\begin{equation}\label{Eq3.9.18}
    \mathbb{E}[\mathcal{L}_0^{(j)}|\mathcal{B}_n]\mathbbm{1}_{\mathcal{P}'}\leq 8e^2\beta_n^{5\slash 8}nt+n\exp( (d_1-d_0)\theta-e^{-2} \beta_n^{-5\slash 8}t\theta \slash 8 +\beta_n^{-3\slash 2} t \theta^2\slash 2 ).
\end{equation}
We take $t=\lceil 32 e^2 \beta_n^{-1\slash 2}\rceil$ and $\theta=e^{-2}\beta_n^{7\slash 8}\slash 16$. Note that 
\begin{equation*}
    \beta_n^{-3\slash 2} t \theta^2\slash 2=e^{-2}\beta_n^{-5\slash 8}t\theta\slash 32, \quad e^{-2} \beta_n^{-5\slash 8}t\theta \slash 8\geq 4\beta_n^{-9\slash 8}\theta\geq 4(d_1-d_0)\theta.
\end{equation*}
Hence by (\ref{Eq3.9.18}), we have
\begin{eqnarray}\label{Eq3.9.21}
    \mathbb{E}[\mathcal{L}_0^{(j)}|\mathcal{B}_n]\mathbbm{1}_{\mathcal{P}'}&\leq& 8e^2\beta_n^{5\slash 8}nt+n\exp(-e^{-2} \beta_n^{-5\slash 8}t\theta \slash 16) \nonumber\\
    &\leq& C(\beta_n^{1\slash 8}+\exp(-c\beta_n^{-1\slash 4}))n.
\end{eqnarray}

\paragraph{Part 4}

Following a similar argument as in \textbf{Part 2}, we can deduce that for any $t\in \mathbb{N}^{*}$ and $\theta\geq 0$,  
\begin{equation}\label{Eq3.9.19}
    \mathbb{E}[\mathcal{L}^{(j)}|\mathcal{B}_n]\mathbbm{1}_{\mathcal{P}'}\leq 8e^2\beta_n^{7\slash 16}nt+n\exp( (d_0-d)\theta-e^{-2} \beta_n^{-7\slash 16}t\theta \slash 8 +\beta_n^{-5\slash 4} t \theta^2\slash 2 ).
\end{equation}
We take $t=\lceil 32e^2\beta_n^{-13\slash 32}\rceil$ and $\theta=e^{-2}\beta_n^{13\slash 16}\slash 16$. Note that
\begin{equation*}
   \beta_n^{-5\slash 4} t \theta^2\slash 2= e^{-2} \beta_n^{-7\slash 16}t\theta \slash 32,\quad e^{-2} \beta_n^{-7\slash 16}t\theta \slash 8\geq 4\beta_n^{-27\slash 32}\theta\geq 4(d_0-d)\theta.
\end{equation*}
Hence by (\ref{Eq3.9.19}), we have
\begin{eqnarray}\label{Eq3.9.22}
    \mathbb{E}[\mathcal{L}^{(j)}|\mathcal{B}_n]\mathbbm{1}_{\mathcal{P}'}&\leq& 8e^2\beta_n^{7\slash 16}nt+n\exp(-e^{-2} \beta_n^{-7\slash 16}t\theta \slash 16) \nonumber\\
    &\leq& C(\beta_n^{1\slash 32}+\exp(-c\beta_n^{-1\slash 32}))n.
\end{eqnarray}

\bigskip

By (\ref{Eq3.9.20}), (\ref{Eq3.9.7}), (\ref{Eq3.9.17}), (\ref{Eq3.9.21}), and (\ref{Eq3.9.22}), we conclude that
\begin{equation}
    \mathbb{E}[\mathscr{L}^{(j)}|\mathcal{B}_n]\mathbbm{1}_{\mathcal{P}'}\leq C(\beta_n^2n+\beta_n^{1\slash 32}+\exp(-c\beta_n^{-1\slash 32}))n+1.
\end{equation}

\end{proof}

\begin{lemma}\label{L3.15}
Assume the preceding setup. There exist positive absolute constants $C_0,c_0$, such that the following holds. We have
\begin{eqnarray}\label{Eq3.12.3}
   && \mathbb{E}\Big[\sum_{\mathbf{a}\in \mathcal{A}_O(d+1)}|\mathbf{a}|\Big|\mathcal{B}_n\Big]\mathbbm{1}_{\mathcal{P}'} \geq  n\mathbbm{1}_{\mathcal{P}'} (\exp(-C_0\beta_n^{1\slash 32})-\exp(-c_0\beta_n^{-1\slash 32}))_{+}^3   \nonumber\\
  && \quad\quad\quad\quad\quad\quad\quad\quad\times(\exp(-C_0(\beta_n^{1\slash 2}+\beta_n^2 n ))-\exp(-c_0\beta_n^{-1\slash 4}))_{+}.
\end{eqnarray}
Moreover, for any $\delta>0$, we have
\begin{eqnarray}\label{Eq3.9.23}
 && \mathbb{P}\Big(\sum_{\mathbf{a}\in \mathcal{A}_O(d+1)}|\mathbf{a}|\leq (1-\delta)n \Big| \mathcal{B}_n\Big)\mathbbm{1}_{\mathcal{P}'} \nonumber\\
 &\leq& (1-(\exp(-C_0\beta_n^{1\slash 32})-\exp(-c_0\beta_n^{-1\slash 32}))_{+}^3  \nonumber\\
 &&\quad\quad\times(\exp(-C_0(\beta_n^{1\slash 2}+\beta_n^2 n ))-\exp(-c_0\beta_n^{-1\slash 4}))_{+})  \delta^{-1}  \mathbbm{1}_{\mathcal{P}'},
\end{eqnarray}
\begin{eqnarray}\label{Eq3.9.24}
  &&  \mathbb{P}\Big(\max_{\mathbf{a}\in \mathcal{A}_O(d+1)}|\mathbf{a}|\geq \delta n\Big|\mathcal{B}_n\Big)\mathbbm{1}_{\mathcal{P}'}\leq (\delta n)^{-1}\mathbb{E}\Big[\max_{\mathbf{a}\in\mathcal{A}_O(d+1)}|\mathbf{a}|\Big|\mathcal{B}_n\Big]\mathbbm{1}_{\mathcal{P}'} \nonumber\\
  &\leq& C_0(\beta_n^2n+\beta_n^{1\slash 32}+\exp(-c_0\beta_n^{-1\slash 32})+n^{-1})^{1\slash 3}\delta^{-1}\mathbbm{1}_{\mathcal{P}'}.
\end{eqnarray}
\end{lemma}
\begin{proof}

We start with the proof of (\ref{Eq3.12.3}) and (\ref{Eq3.9.23}). Note that 
\begin{equation*}
    \sum_{\mathbf{a}\in \mathcal{A}_O(d+1)}|\mathbf{a}|\geq \sum_{j=1}^n  \mathbbm{1}_{\mathcal{S}_{j,d+1}}.
\end{equation*}
Hence by Lemma \ref{L3.11} (with $\beta=\beta_n$),
\begin{eqnarray*}
  &&  \mathbb{E}\Big[\sum_{\mathbf{a}\in \mathcal{A}_O(d+1)}|\mathbf{a}|\Big|\mathcal{B}_n\Big]\mathbbm{1}_{\mathcal{P}'} \geq 
  \sum_{j=1}^n \mathbb{P}(\mathcal{S}_{j,d+1}|\mathcal{B}_n)\mathbbm{1}_{\mathcal{P}'}  \nonumber\\
  &\geq& (\exp(-C\beta_n^{1\slash 32})-\exp(-c\beta_n^{-1\slash 32}))_{+}^3   \nonumber\\
  &&  \times  (\exp(-C(\beta_n^{1\slash 2}+\beta_n^2 n ))-\exp(-c\beta_n^{-1\slash 4}))_{+}n\mathbbm{1}_{\mathcal{P}'},
\end{eqnarray*}
which leads to
\begin{eqnarray*}
 && \mathbb{P}\Big(\sum_{\mathbf{a}\in \mathcal{A}_O(d+1)}|\mathbf{a}|\leq (1-\delta)n \Big| \mathcal{B}_n\Big)\mathbbm{1}_{\mathcal{P}'} \nonumber\\
 &=& \mathbb{P}\Big(n-\sum_{\mathbf{a}\in \mathcal{A}_O(d+1)}|\mathbf{a}|\geq \delta n \Big| \mathcal{B}_n\Big)\mathbbm{1}_{\mathcal{P}'} \leq (\delta n)^{-1} \mathbb{E}\Big[n-\sum_{\mathbf{a}\in \mathcal{A}_O(d+1)}|\mathbf{a}|\Big|\mathcal{B}_n\Big]\mathbbm{1}_{\mathcal{P}'}\nonumber\\
 &\leq& (1-(\exp(-C\beta_n^{1\slash 32})-\exp(-c\beta_n^{-1\slash 32}))_{+}^3  \nonumber\\
 &&\quad\quad\times(\exp(-C(\beta_n^{1\slash 2}+\beta_n^2 n ))-\exp(-c\beta_n^{-1\slash 4}))_{+})  \delta^{-1}  \mathbbm{1}_{\mathcal{P}'}.
\end{eqnarray*}

Now we turn to the proof of (\ref{Eq3.9.24}). In Lemma \ref{L3.13}, we take $I=[n]$, and let $A_1,\cdots,A_k$ be the open arcs at step $d+1$ (with the linear order $\preceq$ on each open arc specified as in the preceding). Note that for any $i\in I$, $l_i=\mathscr{L}^{(i)}$. Moreover, for any $i\in I$, if $i$ is contained in an open arc at step $d+1$, $\alpha_i$ is the length of this open arc; if $i$ is contained in a closed arc at step $d+1$, $\alpha_i=0$. Hence we have $\max_{i\in I}\alpha_i=\max_{\mathbf{a}\in\mathcal{A}_O(d+1)}|\mathbf{a}|$. Thus by Lemma \ref{L3.13}, we have
\begin{equation*}
   \frac{1}{n}\mathbb{E}\Big[\max_{\mathbf{a}\in\mathcal{A}_O(d+1)}|\mathbf{a}|\Big|\mathcal{B}_n\Big]\leq  C\Big(\frac{1}{n}\max_{j\in [n]}\mathbb{E}[\mathscr{L}^{(j)}|\mathcal{B}_n]\Big)^{1\slash 3},
\end{equation*}
which by Lemma \ref{L3.14} leads to
\begin{eqnarray*}
  &&  \mathbb{E}\Big[\max_{\mathbf{a}\in\mathcal{A}_O(d+1)}|\mathbf{a}|\Big|\mathcal{B}_n\Big]\mathbbm{1}_{\mathcal{P}'}\leq Cn\Big(\frac{1}{n}\max_{j\in [n]}\mathbb{E}[\mathscr{L}^{(j)}|\mathcal{B}_n]\mathbbm{1}_{\mathcal{P}'}\Big)^{1\slash 3}\mathbbm{1}_{\mathcal{P}'} \nonumber\\
    &\leq& C(\beta_n^2n+\beta_n^{1\slash 32}+\exp(-c\beta_n^{-1\slash 32})+n^{-1})^{1\slash 3} n\mathbbm{1}_{\mathcal{P}'}.
\end{eqnarray*}
Therefore, we have
\begin{eqnarray*}
  &&  \mathbb{P}\Big(\max_{\mathbf{a}\in \mathcal{A}_O(d+1)}|\mathbf{a}|\geq \delta n\Big|\mathcal{B}_n\Big)\mathbbm{1}_{\mathcal{P}'}\leq (\delta n)^{-1}\mathbb{E}\Big[\max_{\mathbf{a}\in\mathcal{A}_O(d+1)}|\mathbf{a}|\Big|\mathcal{B}_n\Big]\mathbbm{1}_{\mathcal{P}'} \nonumber\\
  &\leq& C(\beta_n^2n+\beta_n^{1\slash 32}+\exp(-c\beta_n^{-1\slash 32})+n^{-1})^{1\slash 3}\delta^{-1}\mathbbm{1}_{\mathcal{P}'}.
\end{eqnarray*}

\end{proof}

In the following, we give the proof of (\ref{PoissonDirichlet}). We denote by $\mathcal{X}$ the space of sequences $(x_j)_{j=1}^{\infty}$ of nonnegative real numbers satisfying $x_1\geq x_2\geq \cdots$ and $\sum_{j=1}^{\infty} x_j<\infty$, and endow $\mathcal{X}$ with the topology induced from the product topology on $\mathbb{R}^{\mathbb{N}}$. Following \cite[Section 4.5]{GP}, we identify $\mathcal{X}$ with the space $\mathcal{M}$ of multisets consisting of countably infinite nonnegative real numbers with finite sum (summing elements of a multiset in terms of their multiplicities), and denote by $\mathscr{I}: \mathcal{X}\rightarrow\mathcal{M}$ the identification map. The topology on $\mathcal{M}$ is induced from the topology on $\mathcal{X}$ by the identification map (that is, $U$ is an open subset of $\mathcal{M}$ if and only if $\mathscr{I}^{-1}(U)$ is an open subset of $\mathcal{X}$). For a multiset that consists of finitely many nonnegative real numbers, we also regard it as an element of $\mathcal{M}$ by adding countably infinite zeros to it. 

We denote by $\mathscr{D}_0$ the $\ell^2$ metric on $\mathcal{X}$: for any $(x_j)_{j=1}^{\infty},(y_j)_{j=1}^{\infty}\in\mathcal{X}$,
\begin{equation*}
    \mathscr{D}_0((x_j)_{j=1}^{\infty},(y_j)_{j=1}^{\infty}):=\sqrt{\sum_{j=1}^{\infty} (x_j-y_j)^2}.
\end{equation*}
Note that convergence in $\mathscr{D}_0$ implies convergence in the topology of $\mathcal{X}$. We denote by $\mathscr{D}$ the following metric on $\mathcal{M}$: for any $A,B\in\mathcal{M}$,
\begin{equation*}
    \mathscr{D}(A,B):=\mathscr{D}_0(\mathscr{I}^{-1}(A),\mathscr{I}^{-1}(B))=\min_{\substack{\phi:A\rightarrow B\\\text{ is a bijection}}}\sqrt{\sum_{x\in A}(\phi(x)-x)^2},
\end{equation*}
where the second equality follows from the rearrangement inequality. Note that if $\sum_{x\in A} x\leq 1$, and $\sum_{x\in B}x\leq 1$, then $\sup_{x\in A \cup B}x\leq 1$ and
\begin{equation}\label{Eq3.14.1}
    \mathscr{D}^2(A,B)\leq\sum_{x\in A}(\phi(x)-x)^2\leq \sum_{x\in B}x^2+\sum_{x\in A}x^2\leq \sum_{x\in B}x+\sum_{x\in A}x\leq 2,
\end{equation}
where $\phi$ is any bijection from $A$ to $B$.

We recall the argument in the two paragraphs after (\ref{Eq3.3.11}) (in Section \ref{Sect.3.3.1}), and define $\{\mathbf{a}_j\}_{j=1}^d$, $\{H_j\}_{j=1}^d$, $\{T_j\}_{j=1}^d$, and $\tau$ as there. Note that conditional on $\mathcal{B}_d$, when the event $\mathcal{P}$ holds, $\tau$ is a uniformly random permutation of $[d]$. Hence for any $j,j'\in [d]$,
\begin{equation}\label{Eq3.12.2}
    \mathbb{P}(j'\in\mathcal{C}_j(\tau)|\mathcal{B}_d)\mathbbm{1}_{\mathcal{P}}=\Big(\frac{1}{2}+\frac{1}{2}\mathbbm{1}_{j=j'}\Big)\mathbbm{1}_{\mathcal{P}}.
\end{equation}
We also define $\tau_0\in S_n$ as follows: conditional on $\mathcal{B}_d$, when the event $\mathcal{P}$ holds, we let $\tau_0=\tau$; otherwise, we draw $\tau_0$ uniformly from $S_d$. Note that $\tau_0$ is a uniformly random permutation of $[d]$.

For any $k\in \mathbb{N}^{*}$, $\gamma\in S_k$, and $w=(w_1,\cdots,w_k)\in\mathbb{R}^k$, we define $\mathscr{M}(\gamma,w)$ to be the multiset $\big\{\sum_{j\in \mathcal{C}^1}w_j,\cdots,\sum_{j\in \mathcal{C}^m}w_j\big\}$, where $\mathcal{C}^1,\cdots,\mathcal{C}^m$ are the cycles of $\gamma$. Note that
\begin{equation*}
    \mathscr{M}\Big(\sigma,\frac{(1,\cdots,1)}{n}\Big)=\Big\{\frac{\sum_{j\in \mathcal{C}}|\mathbf{a}_j|}{n}:\mathcal{C}\text{ is a cycle of }\tau\Big\}\bigcup\Big\{\frac{|\mathbf{a}|}{n}:\mathbf{a}\in \mathcal{A}_C(d+1)\Big\}.
\end{equation*}
Hence 
\begin{eqnarray}\label{Eq3.12.1}
 && \mathscr{D}^2\Big(\mathscr{M}\Big(\sigma,\frac{(1,\cdots,1)}{n}\Big),\mathscr{M}\Big(\tau,\frac{(|\mathbf{a}_1|,\cdots,|\mathbf{a}_d|)}{n}\Big)\Big) \nonumber\\
 &\leq& \sum_{\mathbf{a}\in\mathcal{A}_C(d+1)}\frac{|\mathbf{a}|^2}{n^2}\leq \sum_{\mathbf{a}\in \mathcal{A}_C(d+1)} \frac{|\mathbf{a}|}{n}\leq 1-\sum_{\mathbf{a}\in\mathcal{A}_O(d+1)}\frac{|\mathbf{a}|}{n}.
\end{eqnarray}
By (\ref{Eq3.3.9}), (\ref{Eq3.12.1}), and Lemma \ref{L3.15}, we have
\begin{eqnarray}\label{Eq3.15.1}
&& \mathbb{E}\Big[\mathscr{D}^2\Big(\mathscr{M}\Big(\sigma,\frac{(1,\cdots,1)}{n}\Big),\mathscr{M}\Big(\tau,\frac{(|\mathbf{a}_1|,\cdots,|\mathbf{a}_d|)}{n}\Big)\Big)\Big] \nonumber\\
&\leq& 1-\mathbb{E}\Big[\Big(\sum_{\mathbf{a}\in\mathcal{A}_O(d+1)}\frac{|\mathbf{a}|}{n}\Big)\mathbbm{1}_{\mathcal{P}'}\Big]=1-\mathbb{E}\Big[\frac{1}{n} \mathbb{E}\Big[\sum_{\mathbf{a}\in\mathcal{A}_O(d+1)}|\mathbf{a}|\Big|\mathcal{B}_n\Big]\mathbbm{1}_{\mathcal{P}'}\Big]\nonumber\\
&\leq& 1-(\exp(-C\beta_n^{1\slash 32})-\exp(-c\beta_n^{-1\slash 32}))_{+}^3   \nonumber\\
  && \quad\quad\times(\exp(-C(\beta_n^{1\slash 2}+\beta_n^2 n ))-\exp(-c\beta_n^{-1\slash 4}))_{+}\mathbb{P}(\mathcal{P}')\nonumber\\
  &\leq& 1-(\exp(-C\beta_n^{1\slash 32})-\exp(-c\beta_n^{-1\slash 32}))_{+}^3   \nonumber\\
  && \quad\times(\exp(-C(\beta_n^{1\slash 2}+\beta_n^2 n ))-\exp(-c\beta_n^{-1\slash 4}))_{+}(1-C\exp(-c\beta_n^{-1\slash 16}))_{+}.\nonumber\\
   &&
\end{eqnarray}
Hence as $n\rightarrow\infty$, 
\begin{eqnarray}\label{Eq3.14.2}
&& \mathbb{E}\Big[\mathscr{D}\Big(\mathscr{M}\Big(\sigma,\frac{(1,\cdots,1)}{n}\Big),\mathscr{M}\Big(\tau,\frac{(|\mathbf{a}_1|,\cdots,|\mathbf{a}_d|)}{n}\Big)\Big)\Big] \nonumber\\
&\leq& \sqrt{\mathbb{E}\Big[\mathscr{D}^2\Big(\mathscr{M}\Big(\sigma,\frac{(1,\cdots,1)}{n}\Big),\mathscr{M}\Big(\tau,\frac{(|\mathbf{a}_1|,\cdots,|\mathbf{a}_d|)}{n}\Big)\Big)\Big]}\rightarrow 0.
\end{eqnarray}

We denote by $\mathscr{C}_1,\cdots,\mathscr{C}_m$ the cycles of $\tau$. Note that
\begin{eqnarray*}
 && \mathscr{D}^2\Big(\mathscr{M}\Big(\tau,\frac{(|\mathbf{a}_1|,\cdots,|\mathbf{a}_d|)}{n}\Big),\mathscr{M}\Big(\tau,\frac{(1,\cdots,1)}{d}\Big)\Big) \nonumber\\
 &\leq & \sum_{k=1}^m \Big(\sum_{j\in \mathscr{C}_k}\Big(\frac{|\mathbf{a}_j|}{n}-\frac{1}{d}\Big)\Big)^2=\sum_{k=1}^m\sum_{j,j'\in\mathscr{C}_k}\Big(\frac{|\mathbf{a}_j|}{n}-\frac{1}{d}\Big)\Big(\frac{|\mathbf{a}_{j'}|}{n}-\frac{1}{d}\Big)\nonumber\\
 &=& \sum_{j=1}^d\sum_{j'=1}^d\Big(\frac{|\mathbf{a}_j|}{n}-\frac{1}{d}\Big)\Big(\frac{|\mathbf{a}_{j'}|}{n}-\frac{1}{d}\Big)\mathbbm{1}_{j'\in\mathcal{C}_j(\tau)}.
\end{eqnarray*}
Hence by (\ref{Eq3.12.2}), 
\begin{eqnarray}\label{Eq3.12.4}
    &&  \mathbb{E}\Big[\mathscr{D}^2\Big(\mathscr{M}\Big(\tau,\frac{(|\mathbf{a}_1|,\cdots,|\mathbf{a}_d|)}{n}\Big),\mathscr{M}\Big(\tau,\frac{(1,\cdots,1)}{d}\Big)\Big)\Big|\mathcal{B}_d\Big] \mathbbm{1}_{\mathcal{P}}\nonumber\\
    &\leq& \sum_{j=1}^d\sum_{j'=1}^d\Big(\frac{|\mathbf{a}_j|}{n}-\frac{1}{d}\Big)\Big(\frac{|\mathbf{a}_{j'}|}{n}-\frac{1}{d}\Big)\mathbb{P}(j'\in\mathcal{C}_j(\tau)|\mathcal{B}_d)\mathbbm{1}_{\mathcal{P}}\nonumber\\
    &=& \sum_{j=1}^d\sum_{j'=1}^d\Big(\frac{|\mathbf{a}_j|}{n}-\frac{1}{d}\Big)\Big(\frac{|\mathbf{a}_{j'}|}{n}-\frac{1}{d}\Big)\Big(\frac{1}{2}+\frac{1}{2}\mathbbm{1}_{j=j'}\Big)\mathbbm{1}_{\mathcal{P}}\nonumber\\
    &\leq& \frac{1}{2}\sum_{j=1}^d\Big(\frac{|\mathbf{a}_j|}{n}-\frac{1}{d}\Big)^2+\frac{1}{2}\Big(\frac{\sum_{j=1}^d|\mathbf{a}_j|}{n}-1\Big)^2\nonumber\\
    &\leq& \frac{1}{2d}+\frac{\sum_{j=1}^d|\mathbf{a}_j|^2}{2n^2}+\frac{1}{2}\Big(\frac{\sum_{j=1}^d|\mathbf{a}_j|}{n}-1\Big)^2\nonumber\\
    &\leq& \frac{1}{2d}+\frac{\max_{j\in [d]}|\mathbf{a}_j|}{2n}+\frac{1}{2}\Big(\frac{\sum_{j=1}^d|\mathbf{a}_j|}{n}-1\Big)^2.
\end{eqnarray}
For any $\delta>0$, let $\mathcal{N}_{\delta}$ be the event that $\sum_{\mathbf{a}\in \mathcal{A}_O(d+1)}|\mathbf{a}|\geq (1-\delta)n$ and $\max_{\mathbf{a}\in \mathcal{A}_O(d+1)}|\mathbf{a}|\leq \delta n$. By (\ref{Eq3.3.9}) and Lemma \ref{L3.15}, for any $\delta>0$, 
\begin{eqnarray}\label{Eq3.12.9}
   && \mathbb{P}((\mathcal{N}_{\delta})^c)\leq\mathbb{P}((\mathcal{P}')^c)+\mathbb{P}\Big(\Big\{\max_{\mathbf{a}\in \mathcal{A}_O(d+1)}|\mathbf{a}|\geq \delta n\Big\}\cap\mathcal{P}'\Big)\nonumber\\
    &&\quad\quad\quad\quad\quad+\mathbb{P}\Big(\Big\{\sum_{\mathbf{a}\in \mathcal{A}_O(d+1)}|\mathbf{a}|\leq (1-\delta)n\Big\}\cap\mathcal{P}'\Big)\nonumber\\
    &\leq& \mathbb{P}((\mathcal{P}')^c)+\mathbb{E}\Big[\mathbb{P}\Big(\max_{\mathbf{a}\in \mathcal{A}_O(d+1)}|\mathbf{a}|\geq \delta n\Big|\mathcal{B}_n\Big)\mathbbm{1}_{\mathcal{P}'}\Big]\nonumber\\
    &&+\mathbb{E}\Big[\mathbb{P}\Big(\sum_{\mathbf{a}\in \mathcal{A}_O(d+1)}|\mathbf{a}|\leq (1-\delta)n\Big|\mathcal{B}_n\Big)\mathbbm{1}_{\mathcal{P}'}\Big]\nonumber\\
    &\leq& C\exp(-c\beta_n^{-1\slash 16})+C(\beta_n^2n+\beta_n^{1\slash 32}+\exp(-c\beta_n^{-1\slash 32})+n^{-1})^{1\slash 3}\delta^{-1}\nonumber\\
    &&+ (1-(\exp(-C\beta_n^{1\slash 32})-\exp(-c\beta_n^{-1\slash 32}))_{+}^3  \nonumber\\
 &&\quad\quad\times(\exp(-C(\beta_n^{1\slash 2}+\beta_n^2 n ))-\exp(-c\beta_n^{-1\slash 4}))_{+})  \delta^{-1}.
\end{eqnarray}
By (\ref{Eq3.12.4}), (\ref{Eq3.12.9}), and Lemma \ref{L3.15}, for any $\delta\in (0,1)$,
\begin{eqnarray}\label{Eq3.15.2}
    && \mathbb{E}\Big[\mathscr{D}^2\Big(\mathscr{M}\Big(\tau,\frac{(|\mathbf{a}_1|,\cdots,|\mathbf{a}_d|)}{n}\Big),\mathscr{M}\Big(\tau,\frac{(1,\cdots,1)}{d}\Big)\Big) \mathbbm{1}_{\mathcal{P}} \Big] \nonumber\\
    &=& \mathbb{E}\Big[\mathbb{E}\Big[\mathscr{D}^2\Big(\mathscr{M}\Big(\tau,\frac{(|\mathbf{a}_1|,\cdots,|\mathbf{a}_d|)}{n}\Big),\mathscr{M}\Big(\tau,\frac{(1,\cdots,1)}{d}\Big)\Big)\Big|\mathcal{B}_d\Big] \mathbbm{1}_{\mathcal{P}}\Big] \nonumber\\
    &\leq& \frac{1}{2d}+\frac{\delta}{2}+\frac{\delta^2}{2}+2\mathbb{P}((\mathcal{N}_{\delta})^c)\nonumber\\
    &\leq& \max \{\beta_n^{7\slash 16},n^{-1}\}+\delta+
    C\exp(-c\beta_n^{-1\slash 16})\nonumber\\
    &&  +C(\beta_n^2n+\beta_n^{1\slash 32}+\exp(-c\beta_n^{-1\slash 32})+n^{-1})^{1\slash 3}\delta^{-1}\nonumber\\
    &&  + 2 (1-(\exp(-C\beta_n^{1\slash 32})-\exp(-c\beta_n^{-1\slash 32}))_{+}^3  \nonumber\\
 &&\quad\quad\times(\exp(-C(\beta_n^{1\slash 2}+\beta_n^2 n ))-\exp(-c\beta_n^{-1\slash 4}))_{+})  \delta^{-1}.
\end{eqnarray}
First letting $n\rightarrow\infty$, and then letting $\delta\rightarrow 0^{+}$, we obtain that
\begin{equation}\label{Eq3.12.6}
 \lim_{n\rightarrow\infty}\mathbb{E}\Big[\mathscr{D}^2\Big(\mathscr{M}\Big(\tau,\frac{(|\mathbf{a}_1|,\cdots,|\mathbf{a}_d|)}{n}\Big),\mathscr{M}\Big(\tau,\frac{(1,\cdots,1)}{d}\Big)\Big) \mathbbm{1}_{\mathcal{P}} \Big]=0.
\end{equation}
By (\ref{Eq3.14.1}), we have
\begin{equation*}
    \mathscr{D}^2\Big(\mathscr{M}\Big(\tau,\frac{(|\mathbf{a}_1|,\cdots,|\mathbf{a}_d|)}{n}\Big),\mathscr{M}\Big(\tau,\frac{(1,\cdots,1)}{d}\Big)\Big)\leq 2,
\end{equation*}
which by (\ref{Eq3.3.14}) leads to
\begin{eqnarray*}
  &&\mathbb{E}\Big[\mathscr{D}^2\Big(\mathscr{M}\Big(\tau,\frac{(|\mathbf{a}_1|,\cdots,|\mathbf{a}_d|)}{n}\Big),\mathscr{M}\Big(\tau,\frac{(1,\cdots,1)}{d}\Big)\Big)\mathbbm{1}_{\mathcal{P}^c}\Big]\nonumber\\
  &\leq& 2\mathbb{P}(\mathcal{P}^c)\leq 2(1-\exp(-2\beta_n^{1\slash 8})).
\end{eqnarray*}
Hence
\begin{equation}\label{Eq3.12.7}
    \lim_{n\rightarrow\infty} \mathbb{E}\Big[\mathscr{D}^2\Big(\mathscr{M}\Big(\tau,\frac{(|\mathbf{a}_1|,\cdots,|\mathbf{a}_d|)}{n}\Big),\mathscr{M}\Big(\tau,\frac{(1,\cdots,1)}{d}\Big)\Big)\mathbbm{1}_{\mathcal{P}^c}\Big]=0.
\end{equation}
By (\ref{Eq3.12.6}) and (\ref{Eq3.12.7}), as $n\rightarrow\infty$, 
\begin{eqnarray}\label{Eq3.14.3}
    && \mathbb{E}\Big[\mathscr{D}\Big(\mathscr{M}\Big(\tau,\frac{(|\mathbf{a}_1|,\cdots,|\mathbf{a}_d|)}{n}\Big),\mathscr{M}\Big(\tau,\frac{(1,\cdots,1)}{d}\Big)\Big)\Big]\nonumber\\
    &\leq& \sqrt{\mathbb{E}\Big[\mathscr{D}^2\Big(\mathscr{M}\Big(\tau,\frac{(|\mathbf{a}_1|,\cdots,|\mathbf{a}_d|)}{n}\Big),\mathscr{M}\Big(\tau,\frac{(1,\cdots,1)}{d}\Big)\Big)\Big]}\rightarrow 0.
\end{eqnarray}

By (\ref{Eq3.14.1}), we have
\begin{equation*}
\mathscr{D}^2\Big(\mathscr{M}\Big(\tau,\frac{(1,\cdots,1)}{d}\Big),\mathscr{M}\Big(\tau_0,\frac{(1,\cdots,1)}{d}\Big)\Big)  \leq  2\mathbbm{1}_{\tau\neq \tau_0}.
\end{equation*}
Hence by (\ref{Eq3.3.14}), as $n\rightarrow\infty$,
\begin{eqnarray}\label{Eq3.14.4}
&& \mathbb{E}\Big[\mathscr{D}\Big(\mathscr{M}\Big(\tau,\frac{(1,\cdots,1)}{d}\Big),\mathscr{M}\Big(\tau_0,\frac{(1,\cdots,1)}{d}\Big)\Big)\Big] \nonumber\\
&\leq& \sqrt{\mathbb{E}\Big[\mathscr{D}^2\Big(\mathscr{M}\Big(\tau,\frac{(1,\cdots,1)}{d}\Big),\mathscr{M}\Big(\tau_0,\frac{(1,\cdots,1)}{d}\Big)\Big)\Big]}\nonumber\\
&\leq& \sqrt{2\mathbb{P}(\tau\neq \tau_0)}\leq \sqrt{2\mathbb{P}(\mathcal{P}^c)}\leq \sqrt{2(1-\exp(-2\beta_n^{1\slash 8}))}\rightarrow 0.
\end{eqnarray}

By (\ref{Eq3.14.2}), (\ref{Eq3.14.3}), and (\ref{Eq3.14.4}), as $n\rightarrow\infty$,
\begin{equation}\label{Eq3.14.5}
\mathbb{E}\Big[\mathscr{D}\Big(\mathscr{M}\Big(\sigma,\frac{(1,\cdots,1)}{n}\Big),\mathscr{M}\Big(\tau_0,\frac{(1,\cdots,1)}{d}\Big)\Big)\Big]\rightarrow      0.
\end{equation}
When $n\rightarrow\infty$, we have $d\rightarrow\infty$; as $\tau_0$ is a uniformly random permutation of $[d]$, $\mathscr{M}\big(\tau_0,\frac{(1,\cdots,1)}{d}\big)$ converges in distribution to the Poisson-Dirichlet law with parameter $1$ (see e.g. \cite{Kin,Feng}). This combined with (\ref{Eq3.14.5}) shows that $\mathscr{M}\big(\sigma,\frac{(1,\cdots,1)}{n}\big)$ also converges in distribution to the Poisson-Dirichlet law with parameter $1$.

Now we turn to the proof of (\ref{Unif}). We consider any non-random $s_n\in [n]$. If $s_n$ is contained in a closed arc at step $d+1$, we let $I=0$; otherwise we let $I\in [d]$ be such that $s_n\in \mathbf{a}_I$. We also let $\mathcal{C}_0(\gamma):=\emptyset$ for any $\gamma\in S_d$. Note that 
\begin{equation}\label{Eq3.15.3}
    \Big|\frac{|\mathcal{C}_{s_n}(\sigma)|}{n}-\frac{|\mathcal{C}_I(\tau_0)|}{d}\Big|\leq \frac{|\mathcal{C}_{s_n}(\sigma)|}{n}\mathbbm{1}_{I=0}+\Big|\frac{|\mathcal{C}_{s_n}(\sigma)|}{n}-\frac{|\mathcal{C}_I(\tau)|}{d}\Big|\mathbbm{1}_{I\neq 0}+\mathbbm{1}_{\tau\neq \tau_0}.
\end{equation}
By the argument in (\ref{Eq3.15.1}), as $n\rightarrow\infty$,
\begin{eqnarray}
   && \mathbb{E}\Big[\frac{|\mathcal{C}_{s_n}(\sigma)|}{n}\mathbbm{1}_{I=0}\Big]\leq \mathbb{E}\Big[1-\sum_{\mathbf{a}\in\mathcal{A}_O(d+1)}\frac{|\mathbf{a}|}{n}\Big]\nonumber\\
   &\leq& 1-\mathbb{E}\Big[\Big(\sum_{\mathbf{a}\in\mathcal{A}_O(d+1)}\frac{|\mathbf{a}|}{n}\Big)\mathbbm{1}_{\mathcal{P}'}\Big]\rightarrow 0.
\end{eqnarray}
Recall that the cycles of $\tau$ are denoted by $\mathscr{C}_1,\cdots,\mathscr{C}_m$. Note that 
\begin{eqnarray*}
   &&  \Big(\frac{|\mathcal{C}_{s_n}(\sigma)|}{n}-\frac{|\mathcal{C}_I(\tau)|}{d}\Big)^2\mathbbm{1}_{I\neq 0}=\Big(\sum_{j\in\mathcal{C}_I(\tau)}\Big(\frac{|\mathbf{a}_j|}{n}-\frac{1}{d}\Big)\Big)^2\mathbbm{1}_{I\neq 0}\nonumber\\
   &\leq & \sum_{k=1}^m \Big(\sum_{j\in \mathscr{C}_k}\Big(\frac{|\mathbf{a}_j|}{n}-\frac{1}{d}\Big)\Big)^2= \sum_{j=1}^d\sum_{j'=1}^d \Big(\frac{|\mathbf{a}_j|}{n}-\frac{1}{d}\Big)\Big(\frac{|\mathbf{a}_{j'}|}{n}-\frac{1}{d}\Big)\mathbbm{1}_{j'\in\mathcal{C}_j(\tau)}.
\end{eqnarray*}
Hence by the argument in (\ref{Eq3.12.4})-(\ref{Eq3.15.2}), 
\begin{eqnarray*}
  &&  \mathbb{E}\Big[\Big|\frac{|\mathcal{C}_{s_n}(\sigma)|}{n}-\frac{|\mathcal{C}_I(\tau)|}{d}\Big|\mathbbm{1}_{I\neq 0}\mathbbm{1}_{\mathcal{P}}\Big]^2\leq \mathbb{E}\Big[\Big(\frac{|\mathcal{C}_{s_n}(\sigma)|}{n}-\frac{|\mathcal{C}_I(\tau)|}{d}\Big)^2\mathbbm{1}_{I\neq 0}\mathbbm{1}_{\mathcal{P}}\Big] \nonumber\\
  &\leq& \mathbb{E}\Big[\sum_{j=1}^d\sum_{j'=1}^d\Big(\frac{|\mathbf{a}_j|}{n}-\frac{1}{d}\Big)\Big(\frac{|\mathbf{a}_{j'}|}{n}-\frac{1}{d}\Big)\mathbb{P}(j'\in\mathcal{C}_j(\tau)|\mathcal{B}_d)\mathbbm{1}_{\mathcal{P}}\Big] \nonumber\\
  &\leq& \frac{1}{2d}+\frac{\delta}{2}+\frac{\delta^2}{2}+2\mathbb{P}((\mathcal{N}_{\delta})^c)\nonumber\\
    &\leq& \max \{\beta_n^{7\slash 16},n^{-1}\}+\delta+
    C\exp(-c\beta_n^{-1\slash 16})\nonumber\\
    &&  +C(\beta_n^2n+\beta_n^{1\slash 32}+\exp(-c\beta_n^{-1\slash 32})+n^{-1})^{1\slash 3}\delta^{-1}\nonumber\\
    &&  + 2(1-(\exp(-C\beta_n^{1\slash 32})-\exp(-c\beta_n^{-1\slash 32}))_{+}^3  \nonumber\\
 &&\quad\quad\times(\exp(-C(\beta_n^{1\slash 2}+\beta_n^2 n ))-\exp(-c\beta_n^{-1\slash 4}))_{+})  \delta^{-1}.
\end{eqnarray*}
First letting $n\rightarrow\infty$, and then letting $\delta\rightarrow 0^{+}$, we obtain that
\begin{equation}
    \lim_{n\rightarrow\infty} \mathbb{E}\Big[\Big|\frac{|\mathcal{C}_{s_n}(\sigma)|}{n}-\frac{|\mathcal{C}_I(\tau)|}{d}\Big|\mathbbm{1}_{I\neq 0}\mathbbm{1}_{\mathcal{P}}\Big]=0.
\end{equation}
By (\ref{Eq3.3.14}), as $n\rightarrow\infty$,
\begin{equation}
    \mathbb{E}\Big[\Big|\frac{|\mathcal{C}_{s_n}(\sigma)|}{n}-\frac{|\mathcal{C}_I(\tau)|}{d}\Big|\mathbbm{1}_{I\neq 0}\mathbbm{1}_{\mathcal{P}^c}\Big]\leq \mathbb{P}(\mathcal{P}^c)\leq 1-\exp(-2\beta_n^{1\slash 8})\rightarrow 0,
\end{equation}
\begin{equation}\label{Eq3.15.4}
    \mathbb{P}(\tau\neq \tau_0)\leq\mathbb{P}(\mathcal{P}^c)\leq 1-\exp(-2\beta_n^{1\slash 8})\rightarrow 0.
\end{equation}
By (\ref{Eq3.15.3})-(\ref{Eq3.15.4}), as $n\rightarrow\infty$,
\begin{equation}\label{Eq3.15.5}
    \mathbb{E}\Big[\Big|\frac{|\mathcal{C}_{s_n}(\sigma)|}{n}-\frac{|\mathcal{C}_I(\tau_0)|}{d}\Big|\Big]\rightarrow 0.
\end{equation}

As $\tau_0$ is a uniformly random permutation of $[d]$ (conditional on $\mathcal{B}_d$), for any $x \in [0,1]$, we have
\begin{equation*}
    \mathbb{P}\Big(\frac{|\mathcal{C}_I(\tau_0)|}{d}\leq x\Big|\mathcal{B}_d\Big)\mathbbm{1}_{I\neq 0}= \frac{\lfloor xd\rfloor}{d} \mathbbm{1}_{I\neq 0},
\end{equation*}
hence
\begin{equation}\label{Eq3.15.7}
    \mathbb{P}\Big(\frac{|\mathcal{C}_I(\tau_0)|}{d}\leq x\Big)
    \begin{cases}
        \geq \frac{\lfloor x d \rfloor}{d}\mathbb{P}(I\neq 0)\\
        \leq \frac{\lfloor x d \rfloor}{d}\mathbb{P}(I\neq 0)+\mathbb{P}(I=0)
    \end{cases}. 
\end{equation}
By Lemma \ref{L3.11} and (\ref{Eq3.3.9}) (taking $\beta=\beta_n$), as $n\rightarrow\infty$,  
\begin{eqnarray}\label{Eq3.15.6}
  &&  \mathbb{P}(I\neq 0)=\mathbb{P}(\mathcal{S}_{s_n,d+1})\geq\mathbb{E}[\mathbb{P}(\mathcal{S}_{s_n,d+1}|\mathcal{B}_n)\mathbbm{1}_{\mathcal{P}'}] \nonumber\\
  &\geq& (\exp(-C\beta_n^{1\slash 32})-\exp(-c\beta_n^{-1\slash 32}))_{+}^3\nonumber\\
  &&  \times (\exp(-C(\beta_n^{1\slash 2}+\beta_n^2n))-\exp(-c\beta_n^{-1\slash 4}))_{+}\mathbb{P}(\mathcal{P}')\nonumber\\
  &\geq& (1-C\exp(-c\beta_n^{-1\slash 16}))_{+}(\exp(-C\beta_n^{1\slash 32})-\exp(-c\beta_n^{-1\slash 32}))_{+}^3\nonumber\\
  &&\times (\exp(-C(\beta_n^{1\slash 2}+\beta_n^2n))-\exp(-c\beta_n^{-1\slash 4}))_{+}\rightarrow 1.
\end{eqnarray}
By (\ref{Eq3.15.7}) and (\ref{Eq3.15.6}), for any $x\in [0,1]$, $\lim_{n\rightarrow\infty}\mathbb{P}(|\mathcal{C}_I(\tau_0)|\slash d\leq x) \rightarrow x$. Hence $|\mathcal{C}_I(\tau_0)|\slash d$ converges in distribution to $U([0,1])$. Noting (\ref{Eq3.15.5}), we conclude that $|\mathcal{C}_{s_n}(\sigma)|\slash n$ converges in distribution to $U([0,1])$. 

\begin{appendices}
  
\section{Proof of Proposition \ref{P2.1_n}}

In this section, we give the proof of Proposition \ref{P2.1_n}. We start with the following definition.

\begin{definition}\label{Regions}
For any $j\in [n]$ and $l\in\mathbb{N}$, we define
\begin{equation*}
    \mathcal{E}_{j,l}:=\{(x,y)\in [n]^2:x\leq j,y\geq j+1,  \min\{j-x,y-j-1\}=l \},
\end{equation*}
\begin{equation*}
    \mathcal{E}'_{j,l}:=\{(x,y)\in [n]^2:x\geq j+1, y\leq j,   \min\{x-j-1,j-y\}=l\};
\end{equation*}
\begin{equation*}
\mathcal{G}_{j,l}:=\{(x,y)\in \mathcal{E}_{j,l}:j-x\leq y-j-1\},
\end{equation*}
\begin{equation*}
\mathcal{H}_{j,l}:=\{(x,y)\in \mathcal{E}_{j,l}: j-x>y-j-1\},
\end{equation*}
\begin{equation*}
\mathcal{G}'_{j,l}:=\{(x,y)\in \mathcal{E}'_{j,l}: x-j-1\leq j-y\},
\end{equation*}
\begin{equation*}
\mathcal{H}'_{j,l}:=\{(x,y)\in \mathcal{E}'_{j,l}: x-j-1>j-y\}. 
\end{equation*}
For any $j\in [n]$ and $l,l'\in\mathbb{N}$ such that $l\leq l'$, we define
\begin{equation*}
\mathcal{G}_{j,l,l'}:=\{(x,y)\in \mathcal{G}_{j,l}: y-j-1=l'\},
\end{equation*}
\begin{equation*}
\mathcal{H}_{j,l,l'}:=\{(x,y)\in \mathcal{H}_{j,l}: j-x=l'\},
\end{equation*}
\begin{equation*}
\mathcal{G}'_{j,l,l'}:=\{(x,y)\in \mathcal{G}'_{j,l}: j-y=l' \},
\end{equation*}
\begin{equation*}
\mathcal{H}'_{j,l,l'}:=\{(x,y)\in \mathcal{H}'_{j,l}: x-j-1=l'\}.
\end{equation*}
Note that $|\mathcal{G}_{j,l,l'}|,|\mathcal{H}_{j,l,l'}|,|\mathcal{G}'_{j,l,l'}|,|\mathcal{H}'_{j,l,l'}|\in \{0,1\}$. 
\end{definition}

We also have the following lemma from \cite{Zhong2}.

\begin{lemma}[\cite{Zhong2}, Lemma 5.3.1]
For any $d\in\mathbb{N}^{*}$ such that $d\leq m$, we have
\begin{equation}
    \binom{m}{d}\leq \Big(\frac{em}{d}\Big)^d.
\end{equation}
\end{lemma}

We give the proof of Proposition \ref{P2.1_n} as follows. The proof is modified from the proof of \cite[Proposition 5.3.1]{Zhong2}.

\begin{proof}[Proof of Proposition \ref{P2.1_n}]

Throughout the proof, we consider any $r\in \mathbb{N}^{*}$ and $j\in [n]$. Note that for any $\sigma\in S_n$ and any $u,v\in [n]$ such that $u\leq j$, $v\geq j+1$, $\sigma(u)\geq j+1$, $\sigma(v)\leq j$, we have 
\begin{eqnarray}\label{E1}
    &&(\sigma(u)-u)_{+}+(\sigma(v)-v)_{+}-(\sigma(u)-v)_{+}-(\sigma(v)-u)_{+} \nonumber  \\
    &=&\min\{v,\sigma(u)\}-\max\{\sigma(v),u\}\nonumber\\
    &=&\min\{v-j-1,\sigma(u)-j-1\}+\min\{j-\sigma(v),j-u\}+1. 
\end{eqnarray}
For any $\sigma\in S_n$, we let 
\begin{equation*}
    N_1(\sigma):=S(\sigma)\bigcap\big(\bigcup_{l= 0}^{\infty}\mathcal{E}_{j,l}\big), \quad N_2(\sigma):=S(\sigma)\bigcap \big(\bigcup_{l=0}^{\infty}\mathcal{E}'_{j,l}\big).
\end{equation*}
Note that $|N_1(\sigma)|=|N_2(\sigma)|$.

\paragraph{Pairing and coloring} 

Consider any $\sigma\in S_n$. We pair each point in $N_1(\sigma)$ with a point in $N_2(\sigma)$ in a one-to-one manner, such that a point in $N_1(\sigma)$ with larger first coordinate is paired to a point in $N_2(\sigma)$ with larger first coordinate. We denote this pairing by \textbf{Pairing}. 

Below we consider any $l_1,l_1',l_2,l_2'\in \mathbb{N}$ with $l_1 \leq l_1'$ and $l_2\leq l_2'$. 

For any pair of points $P_1,P_2$ from \textbf{Pairing} such that $P_1\in\mathcal{G}'_{j,l_1,l_1'}$ and $P_2\in\mathcal{H}_{j,l_2,l_2'}$, if $l_1'\leq l_2'$, we mark $P_1$ with red color and $P_2$ with blue color; if $l_1'>l_2'$, we mark $P_2$ with red color and $P_1$ with blue color. 

For any pair of points $P_1,P_2$ from \textbf{Pairing} such that $P_1\in \mathcal{G}_{j,l_1,l_1'}$ and $P_2\in\mathcal{H}'_{j,l_2,l_2'}$, if $l_1'\leq l_2'$, we mark $P_1$ with red color and $P_2$ with blue color; if $l_1'>l_2'$, we mark $P_2$ with red color and $P_1$ with blue color. 

For any pair of points $P_1,P_2$ from \textbf{Pairing} such that $P_1\in \mathcal{G}_{j,l_1,l_1'}$ and $P_2\in \mathcal{G}'_{j,l_2,l_2'}$, if $l_1'\geq l_2$ and $l_2'\geq l_1$, we mark both points with black color. If $l_1'<l_2$ or $l_2'<l_1$, there are two cases: if further $l_1'\leq l_2'$, we mark $P_1$ with red color and $P_2$ with blue color; if further $l_1'>l_2'$, we mark $P_2$ with red color and $P_1$ with blue color. 

For any pair of points $P_1,P_2$ from \textbf{Pairing} such that $P_1\in\mathcal{H}_{j,l_1,l_1'}$ and $P_2\in \mathcal{H}'_{j,l_2,l_2'}$, if $l_1'\geq l_2$ and $l_2'\geq l_1$, we mark both points with black color. If $l_1'<l_2$ or $l_2'<l_1$, there are two cases: if further $l_1'\leq l_2'$, we mark $P_1$ with red color and $P_2$ with blue color; if further $l_1'>l_2'$, we mark $P_2$ with red color and $P_1$ with blue color. 

\paragraph{Switching operation} 

Now we define the ``switching operation''. Starting from a permutation $\sigma_0\in S_n$, the ``switching operation'' corresponding to the points $(j,\sigma_0(j))$ and $(j',\sigma_0(j'))$ (where $j,j'\in [n]$) is defined as the modification from $\sigma_0$ to the new permutation $\sigma_0'=\sigma_0 (j,j')$. Note that $\sigma_0'(j)=\sigma_0(j')$ and $\sigma_0'(j')=\sigma_0(j)$.

For any $\sigma\in S_n$, we define $T(\sigma)\in S_n$ to be the permutation obtained from $\sigma$ by sequentially applying the switching operation for each pair of points in \textbf{Pairing}.

\paragraph{Data associated with each permutation} 

For any $l,l'\in \mathbb{N}$, we denote by $e_{l,l'}(\sigma)$ the number of pairs of black points between $\mathcal{G}_{j,l}$ and $\mathcal{G}'_{j,l'}$, and denote by $f_{l,l'}(\sigma)$ the number of pairs of black points between $\mathcal{H}_{j,l}$ and $\mathcal{H}'_{j,l'}$. We denote by $E_{l,l',1}(\sigma)$ the set of first coordinates of those black points in $\mathcal{G}_{j,l}$ that are paired to a black point in $\mathcal{G}'_{j,l'}$, and denote by $E_{l,l',2}(\sigma)$ the set of first coordinates of those black points in $\mathcal{G}'_{j,l'}$ that are paired to a black point in $\mathcal{G}_{j,l}$. We also denote by $F_{l,l',1}(\sigma)$ the set of second coordinates of those black points in $\mathcal{H}_{j,l}$ that are paired to a black point in $\mathcal{H}'_{j,l'}$, and denote by $F_{l,l',2}(\sigma)$ the set of second coordinates of those black points in $\mathcal{H}'_{j,l'}$ that are paired to a black point in $\mathcal{H}_{j,l}$.  

For any $l,l'\in \mathbb{N}$ such that $l\leq l'$, we denote by $a_{l,l'}(\sigma),a'_{l,l'}(\sigma),b_{l,l'}(\sigma),b'_{l,l'}(\sigma)$ the number of red points in $\mathcal{G}_{j,l,l'},\mathcal{G}'_{j,l,l'},\mathcal{H}_{j,l,l'},\mathcal{H}'_{j,l,l'}$, respectively. We denote by $A_{l,l',1}(\sigma)$ the set of first coordinates of the red points in $\mathcal{G}_{j,l,l'}$ and $A_{l,l',2}(\sigma)$ the set of second coordinates of the red points in $\mathcal{G}_{j,l,l'}$. We can similarly define $A'_{l,l',1}(\sigma), A'_{l,l',2}(\sigma)$ for $\mathcal{G}'_{j,l,l'}$, $B_{l,l',1}(\sigma), B_{l,l',2}(\sigma)$ for $\mathcal{H}_{j,l,l'}$, and $B'_{l,l',1}(\sigma), B'_{l,l',2}(\sigma)$ for $\mathcal{H}'_{j,l,l'}$. 

We note that
\begin{eqnarray}\label{Eq4.4}
  && \quad\quad  \quad\quad e_{l,l'}(\sigma), f_{l,l'}(\sigma)\in \{0,1\}, \text{ for any } l,l'\in \mathbb{N},\nonumber\\
  && a_{l,l'}(\sigma), a'_{l,l'}(\sigma),b_{l,l'}(\sigma),b'_{l,l'}(\sigma) \in \{0,1\}, \text{ for any } l,l'\in \mathbb{N}\text{ with } l\leq l';\nonumber\\
  &&
\end{eqnarray}
\begin{eqnarray}\label{Eq4.5}
 &&  \quad\quad e_{l,l'}(\sigma)=f_{l,l'}(\sigma)=0, \text{ for any } l,l'\in \mathbb{N} \text{ with } \max\{l,l'\}\geq n+1,\nonumber\\
 &&   a_{l,l'}(\sigma)=a'_{l,l'}(\sigma)=b_{l,l'}(\sigma)=b'_{l,l'}(\sigma)=0, \text{ for any } l,l'\in \mathbb{N}\text{ with }  l \leq l', l'\geq n+1.\nonumber\\
 &&
\end{eqnarray}

Now to each $\sigma\in S_n$, we associate the following collection of objects (denoted by $M(\sigma)$ and referred to as the ``data'' associated to $\sigma$ in the sequel):
\begin{equation*}
    e_{l,l'}(\sigma), f_{l,l'}(\sigma), E_{l,l',1}(\sigma), E_{l,l',2}(\sigma), F_{l,l',1}(\sigma), F_{l,l',2}(\sigma) \text{ for every }l,l'\in\mathbb{N};
\end{equation*}
\begin{eqnarray*}
 && a_{l,l'}(\sigma), a'_{l,l'}(\sigma), b_{l,l'}(\sigma), b'_{l,l'}(\sigma), A_{l,l',1}(\sigma), A_{l,l',2}(\sigma), A'_{l,l',1}(\sigma), A'_{l,l',2}(\sigma),\\
 &&  B_{l,l',1}(\sigma), B_{l,l',2}(\sigma), B'_{l,l',1}(\sigma), B'_{l,l',2}(\sigma) \text{ for every }l,l'\in \mathbb{N}\text{ such that } l\leq l'.
\end{eqnarray*}

\paragraph{Injective property of the mapping $T$} 

Consider any choice of data $m$ such that $\mathcal{Q}(m):=\{\sigma\in S_n:M(\sigma)=m\}$ is non-empty. Below we show that the mapping $T$ is injective on $\mathcal{Q}(m)$. We denote by $e_{l,l'},f_{l,l'},E_{l,l',1},E_{l,l',2},\cdots$ the corresponding components of $m$ in the following.

Consider any $\sigma'\in S_n$. Suppose that $T(\sigma)=\sigma'$ for some $\sigma\in\mathcal{Q}(m)$. Note that for any $l,l'\in \mathbb{N}$ with $l\leq l'$, we have $|A_{l,l',1}|=|A_{l,l',2}|\in \{0,1\}$, and we only need to consider the case where $|A_{l,l',1}|=|A_{l,l',2}|=1$ (which we assume in the sequel). Using $A_{l,l',1}, A_{l,l',2}$ together with $\sigma'$, we can identity the red point in $\mathcal{G}_{j,l,l'}$ and the blue point paired with this red point as follows. Suppose that $A_{l,l',1}=\{x\}$ and $A_{l,l',2}=\{y\}$. Note that after the switching operation, $(x,\sigma'(x))$ is paired with $((\sigma')^{-1}(y),y)$. Reversing the switching operation (which corresponds to replacing the points $(x,\sigma'(x))$ and $((\sigma')^{-1}(y),y)$ by $(x,y)$ and $((\sigma')^{-1}(y),\sigma'(x))$), we identify the red point in $\mathcal{G}_{j,l,l'}$ and the blue point paired with this red point based on $\sigma'$ and $m$. Similar procedures can be performed with $\mathcal{G}_{j,l,l'}$ replaced by $\mathcal{G}'_{j,l,l'},\mathcal{H}_{j,l,l'},\mathcal{H}'_{j,l,l'}$.

Moreover, for any $l,l'\in\mathbb{N}$, using $E_{l,l',1},E_{l,l',2}$ together with $\sigma'$, we can identify the set of pairs of black points $P_1,P_2$ with $P_1\in\mathcal{G}_{j,l}$ and $P_2\in\mathcal{G}_{j,l'}'$ as follows. Suppose that $E_{l,l',1}=\{x_1,\cdots,x_t\}$ and $E_{l,l',2}=\{x_1',\cdots,x_t'\}$, where $x_1<\cdots<x_t$ and $x_1'<\cdots<x_t'$. After the switching operation, such black points are changed to $\{(x_j,\sigma'(x_j))\}_{j=1}^t$ and $\{(x_j',\sigma'(x_j'))\}_{j=1}^t$. Due to the construction of \textbf{Pairing}, for all pairs of black points $P_1,P_2$ such that $P_1\in\mathcal{G}_{j,l}$ and $P_2\in\mathcal{G}_{j,l'}'$, we have that those black points in $\mathcal{G}_{j,l}$ with larger first coordinate are paired with black  points in $\mathcal{G}'_{j,l'}$ with larger first coordinate. Therefore, before switching, $(x_j,\sigma'(x_j'))$ is paired with $(x_j',\sigma'(x_j))$ for each $j\in [t]$. Thus we identify the set of pairs of black points $P_1,P_2$ with $P_1\in\mathcal{G}_{j,l}$ and $P_2\in\mathcal{G}_{j,l'}'$ based on $\sigma'$ and $m$.

Similarly, for any $l,l'\in\mathbb{N}$, using $F_{l,l',1},F_{l,l',2}$ together with $\sigma'$, we can identify the set of pairs of black points $P_1,P_2$ such that $P_1\in\mathcal{H}_{j,l}$, $P_2\in\mathcal{H}_{j,l'}'$ as follows. Suppose that $F_{l,l',1}=\{y_1,\cdots,y_t\}$ and $F_{l,l',2}=\{y_1',\cdots,y_t'\}$, where $y_1<\cdots<y_t$ and $y_1'<\cdots<y_t'$. After the switching operation, such black points are changed to $\{((\sigma')^{-1}(y_j),y_j)\}_{j=1}^t$ and $\{(\sigma')^{-1}(y_j'),y_j')\}_{j=1}^t$. Now suppose that $(\sigma')^{-1}(y_{k_1})< \cdots < (\sigma')^{-1}(y_{k_t})$ and $(\sigma')^{-1}(y_{k_1'}')< \cdots < (\sigma')^{-1}(y_{k_t'}')$, where $\{k_1,\cdots,k_t\}=\{k_1',\cdots,k_t'\}=[t]$. Due to the construction of \textbf{Pairing}, for all pairs of black points $P_1,P_2$ such that $P_1\in\mathcal{H}_{j,l}$ and $P_2\in\mathcal{H}_{j,l'}'$, we have that those black points in $\mathcal{H}_{j,l}$ with larger first coordinate are paired with black  points in $\mathcal{H}'_{j,l'}$ with larger first coordinate. Therefore, before switching, $((\sigma')^{-1}(y'_{k_j'}),y_{k_j})$ is paired with $((\sigma')^{-1}(y_{k_j}),y'_{k_j'})$ for each $j\in [t]$. Thus we identify the set of pairs of black points $P_1,P_2$ with $P_1\in\mathcal{H}_{j,l}$ and $P_2\in\mathcal{H}_{j,l'}'$ based on $\sigma'$ and $m$.

Therefore, $\sigma$ can be uniquely reconstructed based on $\sigma'$ and $m$. This shows that $T$ is injective on $\mathcal{Q}(m)$.

\paragraph{Tail bound}

Below we consider any $\sigma\in S_n$.

Consider any $l,l'\in\mathbb{N}$ such that $l\leq l'$. For any red point $(u,\sigma(u))$ in $\mathcal{G}_{j,l,l'}$ and the blue point $(v,\sigma(v))$ paired with it in \textbf{Pairing}, there are two possible cases: (a) $(v,\sigma(v))\in \mathcal{G}'_{j,l_2,l_2'}$ for some $l_2,l_2'\in \mathbb{N}$ with $l_2\leq l_2'$; (b) $(v,\sigma(v))\in \mathcal{H}'_{j,l_2,l_2'}$ for some $l_2,l_2'\in \mathbb{N}$ with $l_2\leq l_2'$. First consider case (a). Note that by the construction of \textbf{Pairing}, we have $l'\leq l_2'$, and either $l'<l_2$ or $l_2'<l$. If $l_2'<l$, then $l_2'<l\leq l'$, which contradicts the fact that $l'\leq l_2'$. Hence $l'\leq l_2'$ and $l'<l_2$. By (\ref{E1}), 
\begin{eqnarray*}
   && (\sigma(u)-u)_{+}+(\sigma(v)-v)_{+}-(\sigma(u)-v)_{+}-(\sigma(v)-u)_{+} \nonumber\\
   &\geq& \min\{l',l_2\}+\min\{l,l_2'\}+1\geq l'+l+1.
\end{eqnarray*}
Now we consider case (b). By the construction of \textbf{Pairing}, $l'\leq l_2'$. By (\ref{E1}),
\begin{eqnarray*}
   && (\sigma(u)-u)_{+}+(\sigma(v)-v)_{+}-(\sigma(u)-v)_{+}-(\sigma(v)-u)_{+} \nonumber\\
   &\geq& \min\{l',l_2'\}+\min\{l,l_2\}+1\geq l'+1.
\end{eqnarray*}
Thus we conclude that in both cases,
\begin{equation}\label{E2}
    (\sigma(u)-u)_{+}+(\sigma(v)-v)_{+}-(\sigma(u)-v)_{+}-(\sigma(v)-u)_{+}\geq l'+1.
\end{equation}
The same bound for $\mathcal{G}'_{j,l,l'},\mathcal{H}_{j,l,l'},\mathcal{H}'_{j,l,l'}$ can be obtained similarly.

Consider any pair of black points $(u,\sigma(u))$ and $(v,\sigma(v))$. Without loss of generality we assume that $(u,\sigma(u))\in\mathcal{G}_{j,l}$ and $(v,\sigma(v))\in\mathcal{G}'_{j,l'}$, where $l,l'\in\mathbb{N}$ (other cases can be dealt with similarly). Further assume that $(u,\sigma(u))\in\mathcal{G}_{j,l,\tilde{l}}$ and $(v,\sigma(v))\in\mathcal{G}'_{j,l',\tilde{l}'}$, where $\tilde{l},\tilde{l}'\in\mathbb{N}$ and $\tilde{l}\geq l,\tilde{l}'\geq l'$. By the construction of \textbf{Pairing}, $\tilde{l}\geq l'$ and $\tilde{l}'\geq l$. By (\ref{E1}), 
\begin{eqnarray}\label{E3}
    &&(\sigma(u)-u)_{+}+(\sigma(v)-v)_{+}-(\sigma(u)-v)_{+}-(\sigma(v)-u)_{+} \nonumber \\
    &\geq& \min\{l',\tilde{l}\}+\min\{\tilde{l}',l\}+1\geq l'+l+1.
\end{eqnarray}

Note that for any $\sigma\in S_n$, 
\begin{equation*}
    \sum_{j=1}^n(\sigma(j)-j)_{+}-\sum_{j=1}^n(\sigma(j)-j)_{-}=\sum_{j=1}^n (\sigma(j)-j)=0,
\end{equation*}
hence 
\begin{equation*}
    H(\sigma,Id)= \sum_{j=1}^n(\sigma(j)-j)_{+}+\sum_{j=1}^n(\sigma(j)-j)_{-}=2\sum_{j=1}^n(\sigma(j)-j)_{+}.
\end{equation*}
By (\ref{E2}) and (\ref{E3}), for any $\sigma\in \mathcal{Q}(m)$, we have
\begin{equation*}
   \mathbb{P}_{n,\beta} (\sigma) \leq  e^{-2\beta (\sum\limits_{l,l'\in\mathbb{N}}(l+l'+1)(e_{l,l'}+f_{l,l'})+\sum\limits_{l,l'\in\mathbb{N}:l\leq l'}(l'+1)(a_{l,l'}+a'_{l,l'}+b_{l,l'}+b'_{l,l'}))  }\mathbb{P}_{n,\beta}(T(\sigma)). 
\end{equation*}
As $T$ is injective on $\mathcal{Q}(m)$, we have 
\begin{eqnarray}\label{Eq4.1}
  &&  \mathbb{P}_{n,\beta}(\mathcal{Q}(m))=\sum_{\sigma\in\mathcal{Q}(m)} \mathbb{P}_{n,\beta}(\sigma) \nonumber\\
  &\leq& \sum_{\sigma\in\mathcal{Q}(m)} e^{-2\beta (\sum\limits_{l,l'\in\mathbb{N}}(l+l'+1)(e_{l,l'}+f_{l,l'})+\sum\limits_{l,l'\in\mathbb{N}:l\leq l'}(l'+1)(a_{l,l'}+a'_{l,l'}+b_{l,l'}+b'_{l,l'}))  }\mathbb{P}_{n,\beta}(T(\sigma)) \nonumber\\
  &\leq& e^{-2\beta (\sum\limits_{l,l'\in\mathbb{N}}(l+l'+1)(e_{l,l'}+f_{l,l'})+\sum\limits_{l,l'\in\mathbb{N}:l\leq l'}(l'+1)(a_{l,l'}+a'_{l,l'}+b_{l,l'}+b'_{l,l'}))  }\mathbb{P}_{n,\beta}(S_n)\nonumber\\
  &=& e^{-2\beta (\sum\limits_{l,l'\in\mathbb{N}}(l+l'+1)(e_{l,l'}+f_{l,l'})+\sum\limits_{l,l'\in\mathbb{N}:l\leq l'}(l'+1)(a_{l,l'}+a'_{l,l'}+b_{l,l'}+b'_{l,l'}))  }.
\end{eqnarray}

For any $\sigma\in S_n$, let $M'(\sigma)$ be the following collection of quantities: 
\begin{equation*}
    e_{l,l'}(\sigma), f_{l,l'}(\sigma) \text{ for every } l,l'\in\mathbb{N};
\end{equation*}
\begin{equation*}
    a_{l,l'}(\sigma), a'_{l,l'}(\sigma), b_{l,l'}(\sigma), b'_{l,l'}(\sigma) \text{ for every } l,l'\in\mathbb{N}  \text{ such that }l\leq l'.
\end{equation*}
Note that for any given $\{a_{l,l'}, a'_{l,l}, b_{l,l'}, b'_{l,l'}\}_{l,l'\in \mathbb{N}:l\leq l'}, \{e_{l,l'}, f_{l,l'}\}_{l,l'\in \mathbb{N}}$ (below we denote by $m'$ such sub-data), there is at most one choice of data $m$ such that this part of $m$ matches $m'$. By (\ref{Eq4.1}), for any sub-data $m'$ such that\\ $\{\sigma\in S_n: M'(\sigma)=m'\}\neq \emptyset$,
\begin{eqnarray}\label{Eq4.2}
  &&  \mathbb{P}_{n,\beta}(M'(\sigma)=m') \nonumber\\
  &\leq& e^{-2\beta (\sum\limits_{l,l'\in\mathbb{N}}(l+l'+1)(e_{l,l'}+f_{l,l'})+\sum\limits_{l,l'\in\mathbb{N}:l\leq l'}(l'+1)(a_{l,l'}+a'_{l,l'}+b_{l,l'}+b'_{l,l'}))  }.
\end{eqnarray}
Note that for any $\sigma\in S_n$,
\begin{equation}\label{Eq4.3}
    |\mathcal{D}_j(\sigma)|=\sum\limits_{l,l'\in\mathbb{N}}(e_{l,l'}(\sigma)+f_{l,l'}(\sigma))+\sum\limits_{l,l'\in\mathbb{N}:l\leq l'}(a_{l,l'}(\sigma)+a'_{l,l'}(\sigma)+b_{l,l'}(\sigma)+b'_{l,l'}(\sigma)).
\end{equation}

For any $t\in \mathbb{N}$, let $Q(t)$ be the number of choices of 
\begin{equation*}
   \{e_{l,l'}, f_{l,l'}\}_{l,l'\in \mathbb{N}:\max\{l,l'\}\leq n},\quad \{a_{l,l'}, a'_{l,l}, b_{l,l'}, b'_{l,l'}\}_{l,l'\in \mathbb{N}:l\leq l'\leq n}
\end{equation*}
such that $e_{l,l'},f_{l,l'},a_{l,l'},a'_{l,l'},b_{l,l'},b'_{l,l'}\in\{0,1\}$ and
\begin{equation*}
\sum_{\substack{l,l'\in \mathbb{N}:\\\max\{l,l'\}\leq n}}(e_{l,l'}+f_{l,l'})+\sum_{\substack{l,l'\in \mathbb{N}:\\ l\leq l'\leq n}}(a_{l,l'}+a'_{l,l'}+b_{l,l'}+b'_{l,l'})\geq t.
\end{equation*}
By (\ref{Eq4.4})-(\ref{Eq4.5}) and (\ref{Eq4.2})-(\ref{Eq4.3}), 
\begin{eqnarray}\label{Eq4.6}
  &&  \mathbb{P}_{n,\beta}(|\mathcal{D}_j(\sigma)|\geq r)\nonumber\\
  &\leq& \sum_{Q(r)}e^{-2\beta (\sum\limits_{l,l'\in\mathbb{N}:\max\{l,l'\}\leq n}(l+l'+1)(e_{l,l'}+f_{l,l'})+\sum\limits_{l,l'\in\mathbb{N}:l\leq l'\leq n}(l'+1)(a_{l,l'}+a'_{l,l'}+b_{l,l'}+b'_{l,l'}))} \nonumber\\
  &\leq & e^{-2\beta r} \sum_{Q(r)}e^{-2\beta (\sum\limits_{l,l'\in\mathbb{N}:\max\{l,l'\}\leq n}(l+l')(e_{l,l'}+f_{l,l'})+\sum\limits_{l,l'\in\mathbb{N}:l\leq l'\leq n}l'(a_{l,l'}+a'_{l,l'}+b_{l,l'}+b'_{l,l'}))}\nonumber\\
  &\leq& e^{-2\beta r} \Big(\prod_{\substack{l,l'\in \mathbb{N}:\\ \max\{l,l'\}\leq n}}(1+e^{-2\beta(l+l')})\Big)^2 \Big(\prod_{\substack{l,l'\in \mathbb{N}:\\ l\leq l'\leq n}}(1+e^{-2\beta l'})\Big)^4.
\end{eqnarray}
As $\beta\geq c_0$, 
\begin{eqnarray*}
&& \prod_{\substack{l,l'\in \mathbb{N}:\\\max\{l,l'\}\leq n}}(1+e^{-2\beta(l+l')})\leq \prod_{\substack{l,l'\in \mathbb{N}:\\l\leq l'}} (1+e^{-2\beta l'}) \prod_{\substack{l,l'\in \mathbb{N}:\\l>l'}}(1+e^{-2\beta l})\nonumber\\
  &\leq& \Big(\prod_{l=0}^{\infty} (1+e^{-2\beta l})^{l +1}\Big)^2 
   \leq  \Big(\prod_{l=0}^{\infty} \exp((l+1) e^{-2\beta l}) \Big)^2\nonumber\\
   &\leq&\exp\Big(2\sum_{l=0}^{\infty} (l+1) e^{-2c_0 l}\Big)=\exp(2(1-e^{-2c_0})^{-2}),
\end{eqnarray*}
\begin{equation*}
    \prod_{\substack{l,l'\in \mathbb{N}:\\ l\leq l'\leq n}}(1+e^{-2\beta l'})\leq \prod_{l=0}^{\infty} (1+e^{-2\beta l})^{l+1}\leq \exp((1-e^{-2c_0})^{-2}).
\end{equation*}
Hence by (\ref{Eq4.6}), we conclude that
\begin{equation}
    \mathbb{P}_{n,\beta}(|\mathcal{D}_j(\sigma)|\geq r)\leq C_0e^{-2\beta r},
\end{equation}
where $C_0$ is a positive constant that only depends on $c_0$. 

\end{proof}

\end{appendices}

\bibliographystyle{acm}
\bibliography{Mallows.bib}

\begin{thebibliography}{10}

\bibitem{ABT}
{\sc Arratia, R., Barbour, A.~D., and Tavar\'{e}, S.}
\newblock {\em Logarithmic combinatorial structures: a probabilistic approach}.
\newblock EMS Monographs in Mathematics. European Mathematical Society (EMS), Z\"{u}rich, 2003.

\bibitem{BV}
{\sc Betz, V.}
\newblock Random permutations of a regular lattice.
\newblock {\em J. Stat. Phys. 155}, 6 (2014), 1222--1248.

\bibitem{BDJ2}
{\sc Borodin, A., Diaconis, P., and Fulman, J.}
\newblock On adding a list of numbers (and other one-dependent determinantal processes).
\newblock {\em Bull. Amer. Math. Soc. (N.S.) 47}, 4 (2010), 639--670.

\bibitem{BLM}
{\sc Boucheron, S., Lugosi, G., and Massart, P.}
\newblock {\em Concentration inequalities}.
\newblock Oxford University Press, Oxford, 2013.
\newblock A nonasymptotic theory of independence, With a foreword by Michel Ledoux.

\bibitem{D}
{\sc Diaconis, P.}
\newblock {\em Group representations in probability and statistics}, vol.~11 of {\em Institute of Mathematical Statistics Lecture Notes---Monograph Series}.
\newblock Institute of Mathematical Statistics, Hayward, CA, 1988.

\bibitem{Feng}
{\sc Feng, S.}
\newblock {\em The {P}oisson-{D}irichlet distribution and related topics}.
\newblock Probability and its Applications (New York). Springer, Heidelberg, 2010.
\newblock Models and asymptotic behaviors.

\bibitem{FM}
{\sc Fyodorov, Y.~V., and Muirhead, S.}
\newblock The band structure of a model of spatial random permutation.
\newblock {\em Probab. Theory Related Fields 179}, 3-4 (2021), 543--587.

\bibitem{GRU}
{\sc Gandolfo, D., Ruiz, J., and Ueltschi, D.}
\newblock On a model of random cycles.
\newblock {\em J. Stat. Phys. 129}, 4 (2007), 663--676.

\bibitem{GP}
{\sc Gladkich, A., and Peled, R.}
\newblock On the cycle structure of {M}allows permutations.
\newblock {\em Ann. Probab. 46}, 2 (2018), 1114--1169.

\bibitem{Kin}
{\sc Kingman, J. F.~C.}
\newblock The population structure associated with the {E}wens sampling formula.
\newblock {\em Theoret. Population Biol. 11}, 2 (1977), 274--283.

\bibitem{Mal}
{\sc Mallows, C.~L.}
\newblock Non-null ranking models. {I}.
\newblock {\em Biometrika 44\/} (1957), 114--130.

\bibitem{MR2319879}
{\sc Massart, P.}
\newblock {\em Concentration inequalities and model selection}, vol.~1896 of {\em Lecture Notes in Mathematics}.
\newblock Springer, Berlin, 2007.
\newblock Lectures from the 33rd Summer School on Probability Theory held in Saint-Flour, July 6--23, 2003, With a foreword by Jean Picard.

\bibitem{SL}
{\sc Shepp, L.~A., and Lloyd, S.~P.}
\newblock Ordered cycle lengths in a random permutation.
\newblock {\em Trans. Amer. Math. Soc. 121\/} (1966), 340--357.

\bibitem{Zho1}
{\sc Zhong, C.}
\newblock Mallows permutation models with {$L^1$} and {$L^2$} distances: hit and run algorithms and mixing times.
\newblock {\em arXiv preprint arXiv:2112.13456\/} (2021).

\bibitem{Zhong2}
{\sc Zhong, C.}
\newblock {\em Mallows permutation model: Sampling algorithms and probabilistic properties}.
\newblock PhD thesis, Stanford University, 2022.

\end{thebibliography}
\end{document}